%% file: arxiv.tex
\documentclass[12pt]{article}
{\begin{Sbox}\begin{minipage}}%
{\end{minipage}\end{Sbox}\fbox{\TheSbox}}%

\usepackage[psamsfonts]{amssymb}
\usepackage{amsmath, amsthm, anysize, enumerate, color, url}
\usepackage{graphicx}
\usepackage{epstopdf}
\usepackage{epsfig}
\usepackage{subfigure}
\usepackage{breakurl}
\usepackage{algorithm}
\usepackage{algpseudocode}
\usepackage{natbib}
\usepackage{booktabs}
\usepackage{threeparttable}
\usepackage{xcolor}
\usepackage{threeparttable}

\newtheorem{theorem}{Theorem} 

\newtheorem{lemma}{Lemma} 

\theoremstyle{definition}
\newtheorem{remark}{Remark}  

\newcommand{\E}{\mathbb{E}}
\newcommand{\R}{\mathbb{R}}

\renewcommand{\P}{\mathbb{P}}

\newcommand{\bs}{\boldsymbol}

\usepackage{setspace}

\begin{document}

\title{Testing degree heterogeneity in directed networks\footnote{Lu Pan and Qiuping Wang contribute equally to this work.}}
\author{Lu Pan\thanks{School of Mathematics and Statistics, Shangqiu Normal University, Shangqiu, 476000, China.
\texttt{Email:} panlulu@mails.ccnu.edu.cn}
\hspace{4mm}
Qiuping Wang\thanks{School of Mathematics and Statistics, Zhaoqing University, Zhaoqing, 526000, China.
\texttt{Email:} qp.wang@mails.ccnu.edu.cn}
\hspace{4mm}
Ting Yan\thanks{Department of Statistics, Central China Normal University, Wuhan, 430079, China.
\texttt{Email:} tingyanty@mail.ccnu.edu.cn}
}
\date{}

\maketitle
\begin{abstract}{In this study, we focus on the likelihood ratio tests in the $p_0$ model for testing degree heterogeneity in directed networks,
which is an exponential family distribution on directed graphs
with the bi-degree sequence as the naturally sufficient statistic.
For testing the homogeneous null hypotheses $H_0: \alpha_1 = \cdots = \alpha_r$, we establish Wilks-type results in both increasing-dimensional and fixed-dimensional settings. For increasing dimensions, the normalized log-likelihood ratio statistic $[2\{\ell(\widehat{\bs\theta})-\ell(\widehat{\bs\theta}^{0})\}-r]/(2r)^{1/2}$ converges in distribution to a standard
normal distribution. For fixed dimensions, $2\{\ell(\widehat{\bs{\theta}})-\ell(\widehat{\bs{\theta}}^0)\}$ converges in distribution to a chi-square distribution with $r-1$ degrees of freedom as $n\rightarrow \infty$, independent of the nuisance parameters.
Additionally,
we present a Wilks-type theorem for the specified null $H_0: \alpha_i=\alpha_i^0$, $i=1,\ldots, r$ in high-dimensional settings, where the normalized log-likelihood ratio statistic also converges in distribution to a standard normal distribution.
These results extend the work of \cite{yan2025likelihood} to directed graphs in a highly non-trivial way,
where  we need to analyze much more expansion terms in the fourth-order asymptotic expansions of the likelihood function
and develop new approximate inverse matrices under the null restricted parameter spaces for approximating the inverse of the Fisher information matrices
in the $p_0$ model.
Simulation studies and real data analyses are presented to verify our theoretical results.}

\end{abstract}

\section{Introduction}

Degree heterogeneity is a very common phenomenon in real-world networks that
describes variation in the degrees of nodes.
In directed networks, there is an out-degree and in-degree for each node, which
represent the number of outgoing edges from a node and number of ingoing edges pointing to the node, respectively.
The $p_0$ model \citep{Yan:Leng:Zhu:2016,Yan2021} can be used to describe degree heterogeneity in directed networks.
It is an exponential random graph model
with the out-degree sequence and in-degree sequence as naturally sufficient statistics.
Specifically, it assumes that all edges are independent and the random edge $a_{i,j}$ between nodes $i$ and $j$ is distributed as a Bernoulli random variable with the edge probability
\begin{equation}\label{p0model}
\P(a_{i,j}=1) = \frac{ e^{\alpha_i + \beta_j} }{ 1 + e^{\alpha_i + \beta_j} },
\end{equation}
where $\alpha_i$ represents the outgoingness parameter of node $i$, which quantifies the effect of an outgoing edge, and $\beta_j$ denotes the popularity parameter of node $j$, which quantifies the effect of an incoming edge.
The $p_0$ model is a simplified version of an earlier $p_1$ model \citep{Holland:Leinhardt:1981} without the reciprocity parameter.

Since the number of parameters increases with the number of nodes in the $p_{0}$ model,
statistical inferences are non-standard and challenging \citep{Goldenberg2010,Fienberg2012}.
\cite{Yan:Leng:Zhu:2016} established consistency
and the asymptotically normal distribution of the maximum likelihood estimation (MLE) when the number of nodes goes to infinity.
However, likelihood ratio tests (LRTs) were not investigated in their study.
It is of interest to test degree heterogeneity amongst a subset of nodes, as most real network data
exhibit some degree of heterogeneity, and LRTs are powerful tools for such analysis.
\cite{Holland:Leinhardt:1981} made the conjecture of a chi-square approximation when testing the reciprocity parameter in the $p_{1}$ model.
\cite{FienbergWasserman1981} suggested a normal approximation for
the scaled LRTs under the null hypothesis that all out-degree parameters (in-degree parameters) are equal to zero. However, these
conjectures have not been formally resolved.

Likelihood ratio theory has experienced tremendous success in parametric hypothesis testing problems due to
the celebrated Wilks theorem \cite{wilks1938}, which states that the minus twice log-likelihood ratio test statistic under the null
converges in distribution to a chi-square distribution, independent of nuisance parameters.
If such a property holds in high-dimensional settings,  it is called the Wilks phenomenon \cite{fan2001}.
\cite{yan2025likelihood} revealed Wilks-type theorems under increasing- and fixed-dimensional null hypotheses
in the $\beta$-model \citep{Chatterjee:Diaconis:Sly:2011} for undirected graphs
and Bradley--Terry model \citep{bradley-terry1952} for paired comparisons.
Since directed networks are equally ubiquitous with undirected networks,
it is crucially important to extend the likelihood ratio methods for testing degree heterogeneity in undirected networks \citep{yan2025likelihood} to directed cases.
 Here,  we aim to investigate whether such results continue to hold in the $p_0$ model.

Our main contributions can be summarized as follows.
For testing homogeneous null hypotheses $H_0: \alpha_1 = \cdots = \alpha_r$,
we establish Wilks-type results in both increasing-dimensional and fixed-dimensional settings.
For an increasing dimension $r$, the normalized log-likelihood ratio statistic $[2\{\ell(\widehat{\bs\theta})-\ell(\widehat{\bs\theta}^{0})\}-r]/(2r)^{1/2}$ converges in distribution to a standard normal distribution. In other words, $2\{\ell(\widehat{\bs{\theta}}) - \ell(\widehat{\bs{\theta}}^0)\}$ is asymptotically a chi-square distribution with a large degree of freedom $r$.
For a fixed dimension $r$, $2\{\ell(\widehat{\bs{\theta}}) - \ell(\widehat{\bs{\theta}}^0)\}$ converges in distribution to a chi-square distribution with $r-1$ degrees of freedom as $n\rightarrow \infty$, independent of the nuisance parameters.
Therefore, the LRTs behave in a classic manner
as long as the difference between the dimension of the full space and dimension of the null space is fixed.
Here, $\ell( \bs{\theta})$ is the log-likelihood function,
$\widehat{\bs{\theta}}$ is the unrestricted MLE of $\bs\theta$,
and $\widehat{\bs{\theta}}^0$ is the restricted MLE of $\bs\theta$ under $H_{0}$.
Additionally, for the specified null $H_0: \alpha_i=\alpha_i^0$, $i=1,\ldots, r$ in high-dimensional settings,
we present a Wilks-type theorem in which
the normalized log-likelihood ratio statistic converges in distribution to a standard normal distribution.
Our proof strategies apply the principled methods developed in \cite{yan2025likelihood}, which
demonstrates that their methods can be extended to other models.
Our results also demonstrate the conjecture of the normal approximation when testing a large number of parameters presented in \cite{FienbergWasserman1981}, in the absence of a reciprocity parameter in the $p_{1}$ model.
However, for the specified null $H_0: \alpha_i=\alpha_i^0$, where $i=1,\ldots, r$ with a fixed $r$,
a different phenomenon is discovered, where
$2\{\ell(\widehat{\bs{\theta}}) - \ell(\widehat{\bs{\theta}}^0)\}$ does not asymptotically follow a chi-square distribution
as observed in \cite{yan2025likelihood} for the Bradley--Terry model.
This clarifies that  the conjecture of a chi-square approximation when testing a single degree parameter is not reasonable,
in contrast to simulation results presented in \cite{Holland:Leinhardt:1981} for testing the reciprocity parameter in the  $p_1$ model.
Simulation studies and several real data analyses are presented to verify our theoretical results.

In contrast to the $\beta$-model for the undirected case in \cite{yan2025likelihood},
there are two new theoretical challenges as follows.
First, we need to analyze much more expansion terms in the fourth-order asymptotic expansions of the likelihood function.
This is because the total number of parameters doubles in the $p_0$ model, where
each node has two parameters--the out- and in-degree parameters while each node in the $\beta$-model has only one parameter.
Second, we develop new approximate inverse matrices under the null restricted parameter spaces for approximating the inverse of the Fisher information matrices in the $p_0$ model since the matrix for approximating the inverse the Fisher information matrix in \cite{yan2025likelihood}
does not work in the directed cases.
In addition, we need to otbain an upper bound of the absolute entry-wise maximum norm between two
approximate inverses under the full space and the restricted null space in the $p_0$ model.

The remainder of this paper is organized as follows.
Section \ref{section-2} presents the Wilks-type theorems for the $p_{0}$ model.
Simulation studies and several real data
analyses are presented in Section \ref{section-numerical}.
Some further discussion is provided in Section \ref{section:discussion}.
Section \ref{section:proof} presents the proofs of Theorems \ref{theorem-LRT-beta} and \ref{theorem-ratio-p-fixed}.
The supporting lemmas are relegated to the Supplemental Material.

\section{Main results}
\label{section-2}
Consider a simple directed graph $\mathcal{G}_n$ with $n$ nodes labelled as ``$1,\ldots,n$".
Let $A=(a_{i,j})_{\{1\leq i,j\leq n\}}\in\{0,1\}^{n\times n}$ be its adjacency matrix,
where $a_{i,j}=1$ if there is a directed edge from $i$ to $j$; otherwise, $a_{i,j}=0$.
We assume that there are no self-loops (i.e., $a_{i,i}=0$ for all $i=1,\ldots, n$).
Let $d_i=\sum_{j \neq i} a_{i,j}$ be the out-degree of node $i$
and $\bs d=(d_1, \ldots, d_n)^\top$ the out-degree sequence of the graph $\mathcal{G}_n$.
Similarly, let $b_j=\sum_{j \neq i} a_{i,j}$ be the in-degree of node $j$
and $\bs b=(b_1, \ldots, b_{n})^\top$ the in-degree sequence.
The pair $( \bs d, \bs b)$ is referred to as the bi-degree sequence.

The $p_{0}$ model postulates that all $a_{i,j}, 1\leq i\neq j\leq n$ are mutually independent Bernoulli random variables with edge probabilities given in (\ref{p0model}).
Let $\bs\alpha=(\alpha_{1},\ldots,\alpha_{n})^{\top}$ and $\bs\beta=(\beta_{1},\ldots,\beta_{n})^{\top}$.
Since $\sum_{i}d_{i}=\sum_{j}b_{j}$, the probability distribution is invariant under the transforms $(\bs{\alpha},\bs{\beta})$ to $(\bs{\alpha}-c,\bs{\beta}+c)$ for a constant $c$.
As in \cite{Yan:Leng:Zhu:2016}, for the identification of the $p_{0}$ model, we set
\begin{equation}\label{eq-identification}
\beta_{n}=0.
\end{equation}
The log-likelihood function under the $p_{0}$ model can be written as
\[
\ell(\bs{\theta})=\sum_{i=1}^{n}\alpha_i d_{i}+\sum_{j=1}^{n-1}\beta_j b_{j}-\sum_{i\neq j}\log(1+e^{\alpha_i + \beta_j}),
\]
where $\bs{\theta}=(\alpha_{1},\ldots,\alpha_{n},\beta_{1},\ldots,\beta_{n-1})^\top$.
The likelihood equations are
\begin{equation}
\label{eq1}
    \begin{split}
&d_{i}=\sum_{j\neq i}\frac{ e^{{\hat\alpha}_i + {\hat\beta}_j} }{ 1 +  e^{{\hat\alpha}_i + {\hat\beta}_j} },~~~i=1,\ldots,n, \\
&b_{j}=\sum_{i\neq j}\frac{ e^{{\hat\alpha}_i + {\hat\beta}_j} }{ 1 + e^{{\hat\alpha}_i + {\hat\beta}_j} },~~~j=1,\ldots,n-1,
\end{split}
\end{equation}
where $\widehat{\bs\theta}=( \hat{\alpha}_1, \ldots, \hat{\alpha}_{n},\hat{\beta}_1, \ldots, \hat{\beta}_{n-1} )^\top$ is the MLE of $\boldsymbol{\theta}$ with $\hat{\beta}_{n}=0$.
The fixed-point iteration algorithm or frequency iterative algorithm in \cite{Holland:Leinhardt:1981}
can be used to solve the above equations.
We present the fixed-point iteration algorithm in Algorithm \ref{alg:homogeneous_mle} for solving the restricted MLE under the null $H_0: \alpha_1 = \cdots = \alpha_r$. The algorithms for other scenarios are provided in the supplementary materials.
Since the log-likelihood function is strictly convex, it guarantees the unique of the MLE.
In our simulations and real data analyses, the fixed-point iterative algorithm converges if the MLE exists.
Conditions for guaranteeing the convergence of the fixed-point iterative algorithm can be found in the book by \cite{Berinde-2007}.
The computational complexity of per iteration is $O(n^{2})$ and the total complexity is $O(K\cdot n^{2})$ if
the number of iterations to reach the required accuracy is $K$.

\begin{algorithm}[!h]
\caption{Restricted MLE under homogeneous null $H_0: \alpha_1 = \cdots = \alpha_r$}
\label{alg:homogeneous_mle}
\begin{algorithmic}[1]
\Require Adjacency matrix $A$, number of restricted nodes $r$, tolerance $\epsilon$
\State Initialize: $\boldsymbol{\alpha}^{(0)} \leftarrow (\alpha_1=\cdots=\alpha_r,\alpha^0_{r+1},\ldots,\alpha^0_n)$, $\boldsymbol{\beta}^{(0)}$, $k \leftarrow 0$
\State Compute out-degrees $d_i \leftarrow \sum_{j\neq i}a_{ij}$ and in-degrees $b_j \leftarrow \sum_{i\neq j}a_{ij}$
\Repeat
    \For{$i = 1$ to $r$}
        \State $\alpha^{(k+1)}_i \leftarrow \log \sum_{i=1}^r d_i - \log \left(\sum_{i=1}^r \sum_{j\neq i} \frac{e^{\beta^{(k)}_j}}{1+e^{\alpha^{(k)}_i+\beta^{(k)}_j}}\right)$
    \EndFor
    \For{$i = r+1$ to $n$}
        \State $\alpha^{(k+1)}_i \leftarrow \log d_i - \log \sum_{j\neq i} \frac{e^{\beta^{(k)}_j}}{1+e^{\alpha^{(k)}_i+\beta^{(k)}_j}}$
    \EndFor
    \For{$j = 1$ to $n$}
        \State $\beta^{(k+1)}_j \leftarrow \log b_j - \log \sum_{i\neq j} \frac{e^{\alpha^{(k)}_i}}{1+e^{\alpha^{(k)}_i+\beta^{(k)}_j}}$
    \EndFor
    \State Compute convergence:
    \State $y_i \leftarrow \frac{1}{d_i}\sum_{j\neq i} \frac{e^{\alpha^{(k+1)}_i+\beta^{(k+1)}_j}}{1+e^{\alpha^{(k+1)}_i+\beta^{(k+1)}_j}}$ for $i = r+1,\ldots,n$
    \State $z_j \leftarrow \frac{1}{b_j}\sum_{i\neq j} \frac{e^{\alpha^{(k+1)}_i+\beta^{(k+1)}_j}}{1+e^{\alpha^{(k+1)}_i+\beta^{(k+1)}_j}}$ for $j = 1,\ldots,n$
   \State $k \gets k + 1$
\Until{$\max(\max_i|y_i-1|, \max_j|z_j-1|) < \epsilon$}
\State \Return $\hat{\boldsymbol{\alpha}} \leftarrow \boldsymbol{\alpha}^{(k)}$, $\hat{\boldsymbol{\beta}} \leftarrow \boldsymbol{\beta}^{(k)}$
\end{algorithmic}
\end{algorithm}

We use $V$ to denote the Hessian matrix of the negative log-likelihood function, and the elements of $V$ ($=(v_{i,j})_{(2n-1)\times (2n-1)}$) are
\begin{equation}\label{definition-v-beta}
\begin{array}{l}%
v_{i,i} = \sum_{j\neq i} \frac{e^{\alpha_i+\beta_j}}{(1 + e^{\alpha_i+\beta_j})^2},~~i=1,\dots,n,\\
v_{n+j,n+j} = \sum_{j\neq i} \frac{e^{\alpha_i+\beta_j}}{(1 + e^{\alpha_i+\beta_j})^2},~~j=1,\dots,n-1,\\
v_{i,j}=0,~~i,j=1,\dots,n,i\neq j,\\
v_{i,j}=0,~~i,j=n+1,\dots,2n-1,i\neq j,\\
v_{i,n+j}=v_{n+j,i}= \frac{e^{\alpha_i+\beta_j}}{(1 + e^{\alpha_i+\beta_j})^2},~~i=1,\dots,n,j=1,\dots,n-1,j\neq i,\\
v_{i,n+i}=v_{n+i,i}=0,~~i=1,\dots,n-1.
\end{array}
\end{equation}
It should be noted that $V$ is also the Fisher information matrix of $\bs{\theta}$ and the covariance matrix of $\mathbf{g}=(d_{1},\ldots,d_{n},b_{1},\ldots,b_{n-1})^{\top}$.
We define the following two notations that play important roles on guaranteeing good properties of $\boldsymbol{\theta}$:
\begin{equation}\label{definition-bncn}
b_n = \max_{i,j} \frac{(1 + e^{\alpha_i+\beta_j})^2}{e^{\alpha_i+\beta_j}}, \quad c_n=\min_{i,j}\frac{(1 + e^{\alpha_i+\beta_j})^2}{e^{\alpha_i+\beta_j}},
\end{equation}
where  $b_n^{-1}$ and $c_n^{-1}$ are equal to the minimum and maximum variances, respectively, of $a_{i,j}$ over $i\neq j$, and $c_n \ge 4$.

We consider a random adjacency matrix $A$ generated from the $p_0$ model with a parameter vector $\boldsymbol{\theta}$.
We consider a homogeneous null
$H_0: \alpha_1=\cdots=\alpha_r$, and a specified null $H_0: \alpha_i=\alpha_i^0$ for $i=1,\ldots,r$, where $\alpha_i^0$'s are known numbers.
Then, the $\log$-likelihood ratio is $2\{ \ell(\boldsymbol{\widehat{\theta}}) - \ell(\boldsymbol{\widehat{\theta}}^0)\}$,
where $\boldsymbol{\widehat{\theta}^{0}}$ denotes the restricted MLE of $\bs{\theta}$ under the null parameter space.
Theorem \ref{theorem-LRT-beta} establishes Wilks-type theorems for testing homogeneous null hypotheses in both increasing-dimensional and fixed-dimensional settings.

\begin{theorem}
\label{theorem-LRT-beta}
\begin{itemize}
\item[\rm{(a)}]If $b_n^{15}/c_n^{12} =o( r^{1/2}/(\log n)^{2})$,
then under the homogenous null $H_0: \alpha_1 = \cdots = \alpha_r$
with increasing dimension, the $\log$-likelihood ratio statistic $2\{ \ell(\boldsymbol{\widehat{\theta}}) - \ell(\boldsymbol{\widehat{\theta}}^0)\}$
is asymptotically normally distributed in the sense that
\begin{equation} \label{statistics-beta}
(2r)^{-1/2}(2\{ \ell(\boldsymbol{\widehat{\theta}}) - \ell(\boldsymbol{\widehat{\theta}}^0)\} - r) \stackrel{L}{\rightarrow} N(0,1), ~~\mbox{as}~~ n\to\infty,
\end{equation}
where $\boldsymbol{\widehat{\theta}^{0}}=\arg\max_{\bs{\theta}\in \Theta_0} \ell(\bs{\theta})$ and $\Theta_0=\{ \bs{\theta}: \bs{\theta}\in \R^{2n-1}, \alpha_1=\cdots=\alpha_r\}$.
\item[\rm{(b)}]If $b_n^{21}/c_n^{17} = o( n^{1/2}/(\log n)^{5/2})$,
then under the homogenous null $H_0: \alpha_1 = \cdots = \alpha_r$, with fixed $r$,
the $\log$-likelihood ratio statistic $2\left\{\ell(\widehat{\bs{\theta}}) - \ell( \widehat{\bs{\theta}}^0) \right\}$
converges in distribution to a chi-square distribution with $r-1$ degrees of freedom.

\end{itemize}
\end{theorem}

The condition in Theorem \ref{theorem-LRT-beta} (a) reflects how sparser networks and greater degree heterogeneity necessitate larger $r$ to satisfy $b_n^{15}/c_n^{12} = o\big(r^{1/2}/(\log n)^2\big)$.
First, if all parameters are bounded by a fixed constant, then the Wilks-type result holds as long as $r \gg (\log n)^6$. Second, when $b_n \asymp c_n$, the condition becomes
$
b_n = o\left(r^{1/6}/(\log n)^{2/3}\right).
$
If all parameters $\alpha_i$  and $\beta_i$ tend to $-\infty$ at the same rate, then \(b_n \asymp c_n\) and $b_n\to\infty$, which correspond to
sparse networks.
In this case, the network density, a defined as \(\frac{\sum_{i,j} A_{i,j}}{n(n-1)}\), is then  \(1/b_n\).
The condition in Theorem 1 (a) becomes \(b_n = o\left( \frac{r^{1/6}}{(\log n)^{2/3}} \right)\). If we let $r=O(n)$,
the lowest network density is in the order of $O( (\log n)^{2/3}/n^{1/6})$ under which the Wilks-type results hold.
Alternatively, the condition in Theorem \ref{theorem-LRT-beta} (b) is stronger than that in Theorem \ref{theorem-LRT-beta} (a). This is because it must ensure the validity of the chi-square approximation, which is more sensitive to parameter growth rates.

Conditions imposed on $b_n$ and $c_n$ in the $p_0$ model and the $\beta$-model are different, owing to
that the derived consistency rates of the restricted MLE and the naive MLE are different.
In the $\beta$-model, we use a tight bound for the inverse of the Fisher information matrix in terms of the $\ell_\infty$-norm
to derive the consistency rate,
which leads to weaker conditions on $b_n$ and $c_n$.
In the $p_0$ model, we use a different method to derive the consistency rate since
there are no such tight bounds for the inverse of the Fisher information matrix.

Next, we present the Wilks-type theorem under the specified testing problem with an increasing dimension.
\begin{theorem}
\label{theorem-ratio-p-fixed}
If $b_n^{15}/c_n^{12} =o( r^{1/2}/(\log n)^{2})$, then under the null
$H_0: \bs{\theta}\in \Theta_0=\{
 \bs{\theta}: \bs{\theta}\in \R^{2n-1}, (\alpha_1, \ldots, \alpha_r) = (\alpha_1^0, \ldots, \alpha_r^0) \}$,
the normalized $\log$-likelihood ratio statistic
in \eqref{statistics-beta} converges in distribution to the standard normal distribution.
\end{theorem}

\begin{remark}
Theorems \ref{theorem-LRT-beta} and \ref{theorem-ratio-p-fixed} only establish the Wilks-type results under the null hypotheses $H_0: \alpha_1 = \cdots = \alpha_r$, and $H_0: \alpha_i = \alpha_{i}^{0}$ for $i = 1,\ldots,r$. For the analogous hypotheses concerning $\bs\beta$ parameters, namely $H_0: \beta_1 = \cdots = \beta_r$ and $H_0: \beta_i = \beta_{i}^{0}, i = 1,\ldots,r$, similar results hold.
\end{remark}

The above theorem does not address the case of a fixed-dimensional specified null $H_0: \alpha_i = \alpha_i^0$ for $i = 1,\ldots,r$, unlike Theorem \ref{theorem-LRT-beta}, which covers both specified and homogeneous null hypotheses. In fact, the Wilks-type result does not hold when testing $H_0: \alpha_i=\alpha_i^0, i=1, \ldots, r$.
Some explanations for this phenomenon are as follows.
By using $S_{22}$ defined in \eqref{definition-S221} to approximate $V_{22}^{-1}$
and similar arguments as in the proof of \eqref{eq-ell-difference} and \eqref{eq-theorem2-B10}, we have
\begin{equation*}
2\left\{ \ell(\widehat{\bs{\theta}}) - \ell( \widehat{\bs{\theta}}^0) \right\}=
\sum_{i=1}^r \frac{ \bar{d}_i^{\,2} }{v_{i,i}}+\frac{ \bar{b}_n^{\,2} }{v_{2n,2n}}-\frac{(\bar{b}_n -\sum_{i=1}^{r}\bar{d}_i)^{2}}{\widetilde{v}_{2n,2n}} + \frac{1}{3}(B_2-B_2^0)+ \frac{1}{12}(B_3-B_3^0),
\end{equation*}
where
$(B_2-B_2^0)$ represents the difference in the third-order expansion terms of the log-likelihood function between the full and null parameter spaces, and $(B_3-B_3^0)$ is the corresponding difference in the fourth-order expansion term.
Here, $B_2-B_2^0$ does not vanish. The magnitude of this term is governed by $\max_{i=r+1,\ldots,2n-1}|\widehat{\theta}_i - \widehat{\theta}_i^0|$,
where the differences under the specified null are represented as follows:
\begin{align*}
\widehat{\alpha}_i - \widehat{\alpha}_i^0 & = \frac{ \bar{b}_n }{ v_{2n,2n} } - \frac{ \bar{b}_n -\sum_{i=1}^r\bar{d}_i }{ \tilde{v}_{2n,2n} } + O_p\left( \frac{ b_n^9 \log n }{ nc_{n}^{7}} \right),~~ i=r+1, \ldots, n,\\
\widehat{\beta}_j - \widehat{\beta}_j^0 & = -\frac{ \bar{b}_n }{ v_{2n,2n} } +\frac{ \bar{b}_n -\sum_{i=1}^r\bar{d}_i }{ \tilde{v}_{2n,2n} } + O_p\left( \frac{ b_n^9 \log n }{ nc_{n}^{7}} \right),~~j=1, \ldots, n-1.
\end{align*}
The difference between the two distributions of $\bar{b}_n / v_{2n,2n} $ and $( \bar{b}_n -\sum_{i=1}^r\bar{d}_i )/\tilde{v}_{2n,2n} $ is much larger than the order of $\log n/n$.
Under the homogenous null,
the difference between the two terms above is not contained in $\widehat{\theta}_i - \widehat{\theta}_i^0$.
As a result, $B_2-B_2^0$ vanishes in the homogenous null but remains in the specified null.
Thus, the Wilks-type result does not hold in the fixed-dimensional specified null.

We plot the density curve of $2\left\{ \ell(\widehat{\bs{\theta}}) - \ell( \widehat{\bs{\theta}}^0) \right\}$ against the chi-square distribution to provide some understanding on the distribution under the specified null. We consider three specified nulls:
(i) $H_0: \alpha_1=0$, (ii) $H_0: (\alpha_1, \alpha_2)=(0,0)$, and (iii) $H_0: (\alpha_1, \alpha_2)=(0.1,0.2)$.
The other parameters are as follows: $\beta_{i}=\alpha_{i}$ and $\alpha_i= 0.1(i-1)\log n/(n-1)$.
The simulation is run $5,000$ times and the plots are presented in Figure \ref{fig:bt-a}.
One can see that the distribution is far away from chi-square distributions with one, two, or three degrees of freedom.

\begin{figure}\centering
\includegraphics[ height=1.5in, width=5in, angle=0]{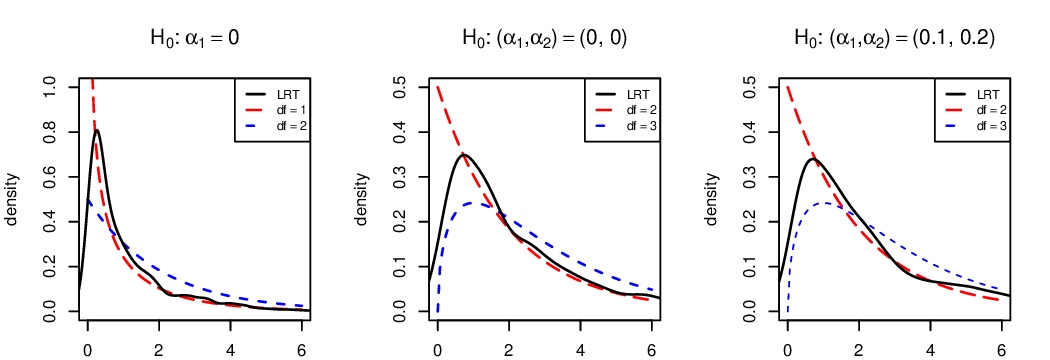}
\caption{Comparisons of density curves between LRT and chi-square distributions with one, two, or three degrees of freedom (n = 200).}
\label{fig:bt-a}
\end{figure}

\section{Numerical experiments}
\label{section-numerical}

\subsection{Simulation studies}

We conduct simulations to evaluate the performance of the log-likelihood ratio statistics
with finite sizes of networks.
We consider three null hypotheses:\\
(i) $H_{01}$: $\alpha_i=(i-1)L_{n}/(n-1)$, $i=1,\ldots,n$;\\
(ii) $H_{02}$: $\alpha_1=\cdots=\alpha_r$, $r=n/2$; and\\
(iii) $H_{03}$: $\alpha_1=\cdots=\alpha_r$, $r=10$;\\
where $L_{n}$ is used to control different asymptotic regimes.
We choose three values for $L_n$: $L_n= 0$, $0.1\log n$, and $0.2\log n$.
Under $H_{02}$ and $H_{03}$, we set the left $n-r$ parameters as $\alpha_i = (i-1)L_{n}/(n-1)$ for $i=r+1,\ldots,n$.
For simplicity, we set $\beta_i=\alpha_i$ as in \cite{Yan:Leng:Zhu:2016} and \cite{Yan2021}, and $\beta_n=0$.
$H_{01}$ aims to test a specified null.
$H_{02}$ and $H_{03}$ aim to test whether a given set of parameters with increasing or fixed dimensions are equal.
For the homogenous nulls $H_{02}$ and $H_{03}$, we set $\alpha_1 = \cdots=\alpha_r=0$. We consider $n=100, 200$, and $n=500$.
Remark that the null $H_{03}$ is different from the null (i.e., $\alpha_1=0, \alpha_2=0$)  in Figure \ref{fig:bt-a},
which  assumes that the parameters take specific fixed values. As a result, the MLEs under these two different null hypotheses are different.
Under each simulation setting, $5,000$ datasets are generated.

We draw QQ-plots to evaluate the asymptotic distributions of the log-likelihood ratio statistics.
Owing to space limitations, we only present plots of quantiles of chi-square distributions or standard normal distributions versus
sample quantiles in the case of $n=500$ under $H_{01}$. Other plots, including those for the cases of $n=100$ or $n=200$, can be found in the Supplementary Material.
Figure \ref{fig:p1} parts (a) and (b) present the results of the normal approximation of the normalized LRT
and chi-square approximation of the LRT under $H_{01}$, respectively. The sample quantiles agree well with theoretical quantiles,
except that there are slight deviations from the reference line $y = x$ in the tail of curves for the normalized LRT.

\begin{figure}
\centering
\subfigure[QQ-plots for normalized log-likelihood ratio statistics]{\includegraphics[width=0.9\textwidth]{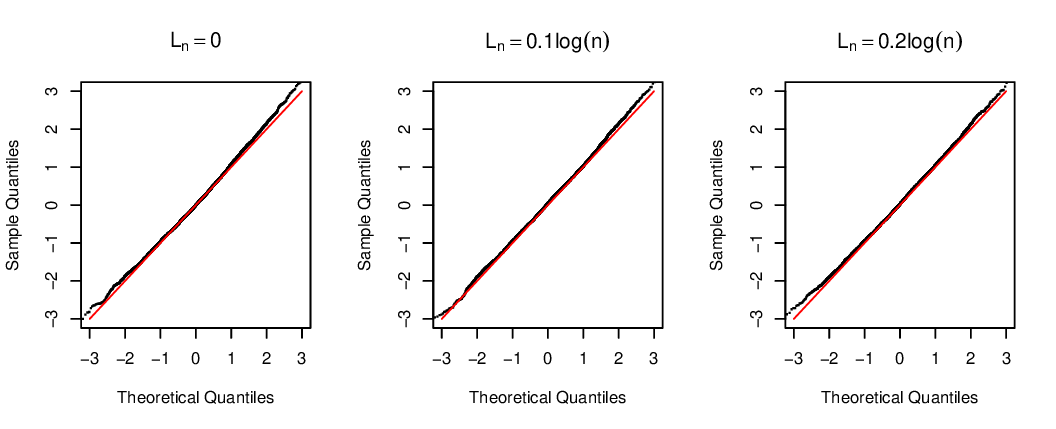}}
\subfigure[QQ-plots for log-likelihood ratio statistics]{\includegraphics[width=0.9\textwidth]{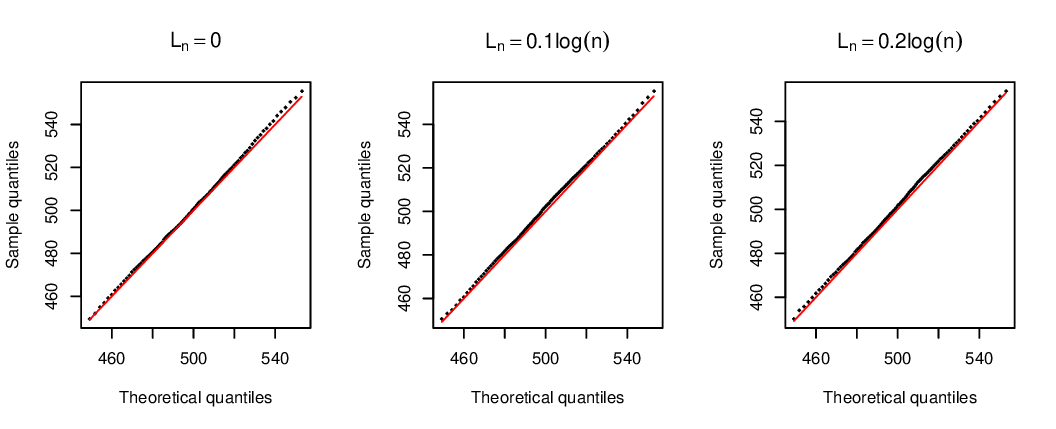}}
\caption{QQ-plots under $H_{01}$ (n=500). The horizontal and vertical axes in each QQ-plot are the theoretical and empirical quantiles, respectively.
The straight lines correspond to $y=x$.}
\label{fig:p1}
\end{figure}

We also record the Type-I errors under two nominal levels $\tau_1=0.05$ and $\tau_2=0.1$ in Table \ref{table-type-I}.
In this table, one can see that the simulated Type-I errors are very close to the nominal level.
For example, the largest relative difference $(a-b)/a$ is $0.084$ when $n=200$ under $H_{01}$,
where $a$ denotes the normal level and $b$ is the simulated Type-I error.

\begin{table}
\centering
\caption{{\color{black}Simulated Type-I errors $(\times 100)$ of the LRTs under two nominal levels are as follows:
$\tau_1=0.05$, $\tau_2=0.1$}}
\label{table-type-I}
\vskip5pt
\small
\begin{tabular}{ccc ccc}
\hline
NULL      &   $n$      & Normal Level    & $ L_{n} =0$  & $L_{n} = 0.1\log(n)$  & $L_{n} = 0.2\log(n) $    \\
\hline
$H_{01}$  &   $100$    & $\tau_1$          &$ 4.92 $&$ 5.06 $&$ 5.30 $  \\
          &            & $\tau_2$          &$ 9.92 $&$ 10.36 $&$ 10.36 $  \\
          &   $200$    & $\tau_1$          &$ 5.42 $&$ 5.12 $&$ 5.18 $   \\
          &            & $\tau_2$           &$ 10.42 $&$ 10.82 $&$ 10.08 $ \\
          &   $500$    & $\tau_1$          &$ 5.20 $&$ 5.20 $&$ 5.14 $    \\
          &            & $\tau_2$           &$ 10.68 $&$ 10.14 $&$ 9.98 $   \\
$H_{02}$  &  $100$     & $\tau_1$           &$ 5.06 $&$ 5.40 $&$ 5.46 $  \\
          &            & $\tau_2$          &$ 9.96 $&$ 10.26 $&$ 11.02 $ \\
          &  $200$     & $\tau_1$         &$ 5.32 $&$ 5.52 $&$ 4.26 $\\
          &            & $\tau_2$            &$ 10.32 $&$ 10.58 $&$ 9.66 $ \\
          &  $500$     & $\tau_1$          &$ 6.00$&$ 4.90 $&$ 4.48 $\\
          &            & $\tau_2$            &$ 10.42 $&$ 9.90 $&$ 9.42 $ \\
$H_{03}$  &  $100$     & $\tau_1$           &$ 4.74 $&$ 5.44 $&$ 5.14 $ \\
          &            & $\tau_2$           &$ 10.16 $&$ 10.50 $&$ 10.26 $  \\
          &  $200$     & $\tau_1$           &$ 4.80 $&$ 4.74 $&$ 5.18 $  \\
          &            & $\tau_2$           &$ 10.52 $&$ 9.40 $&$ 9.88 $\\
          &  $500$     & $\tau_1$          &$ 5.22 $&$ 5.24 $&$ 5.54 $ \\
          &            & $\tau_2$           &$ 10.40 $&$ 9.90 $&$ 9.98 $ \\
\hline
\end{tabular}
\end{table}

Next, we compare the powers of the log-likelihood ratio tests and Wald-type test statistics for the fixed-dimensional null hypotheses $H_0$: $\alpha_1=\cdots=\alpha_r$.
Following \cite{Yan:Leng:Zhu:2016}, the Wald-type test statistic can be constructed as follows:

\begin{equation}
\label{eq-wald-test}
\begin{pmatrix}
\widehat{\alpha}_1 - \widehat{\alpha}_2 \\ \widehat{\alpha}_{2} - \widehat{\alpha}_3\\
 \ldots\\
 \widehat{\alpha}_{r-1} - \widehat{\alpha}_r
\end{pmatrix}^\top
\begin{pmatrix}
\frac{1}{ \hat{v}_{11} } + \frac{1}{ \hat{v}_{22} } & - \frac{1}{ \hat{v}_{22} } &  0 & \ldots & 0 \\
- \frac{1}{ \hat{v}_{22} } & \frac{1}{ \hat{v}_{22} } + \frac{1}{ \hat{v}_{33} } &  -\frac{1}{ \hat{v}_{33} }  & \ldots & 0 \\
& & \vdots & & \\
0  &\ldots & 0 &  -\frac{1}{ \hat{v}_{r-1,r-1} } & \frac{1}{ \hat{v}_{r-1,r-1} } + \frac{1}{\hat{v}_{rr}}
\end{pmatrix}^{-1}
\begin{pmatrix}
\widehat{\alpha}_1 - \widehat{\alpha}_2 \\ \widehat{\alpha}_{2} - \widehat{\alpha}_3\\
 \ldots\\
 \widehat{\alpha}_{r-1} - \widehat{\alpha}_r
\end{pmatrix}
\end{equation}

The true model is set to be $\alpha_i=ic/r$ with $i=1,\ldots,r$.
The other parameters are set to $\alpha_i=0.2i\log n/n$ for $i=r+1,\ldots,n$
and $\beta_i = 0.2 i\log n/n$ for $i=1,\ldots,n-1$ with $\beta_n=0$.
The simulated results are presented in Table \ref{powers-a} with the nominal level $0.05$.
When $H_0$ (i.e., $c=0$) is true, the powers are close to the nominal level, as expected.
Furthermore, when $n$ and $r$ are fixed, the power tends to increase with $c$ and approaches $100\%$ when $c=1.3$.
When $n$ and $c$ are fixed while $r$ increases, a similar phenomenon can be observed.
Figure \ref{fig2} shows the powers of the LRT and the Wald-type test change with $r$ when $n=200$ and $c=1$.
As expected, we can see that the powers of both the LRT and the Wald test increase with $r$  since
the heterogeneity of parameters being tested becomes more severe. In addition, the power of the LRT is a little higher than those of the Wald test,
which agrees with the finding in Table \ref{powers-a}.
\begin{table}
\centering
\caption{Comparison between the powers of the LRTs and Wald-type tests in parentheses.}
\label{powers-a}
\small
\begin{tabular}{ccc ccc ccc ccc ccc}
\hline
  $n$            & $r$    &\multicolumn{2}{c} {$c=0$ }    &\multicolumn{2}{c}{$c=0.2$} &  \multicolumn{2}{c}{$c=0.4$} &  \multicolumn{2}{c}{$c=0.7$}  &\multicolumn{2}{c}{$c=1.0$}  &\multicolumn{2}{c}{$c=1.3$} \\
  \cmidrule(r){3-4} \cmidrule(r){5-6} \cmidrule(r){7-8} \cmidrule(r){9-10} \cmidrule(r){11-12} \cmidrule(r){13-14}
    &     & LRT & Wald & LRT & Wald  & LRT & Wald & LRT & Wald& LRT & Wald & LRT & Wald\\
\hline
 $100$          & $5$    &$ 5.76$&$5.58 $&$ 5.74$&$5.34$&$ 9.56 $&$8.8$&$ 24.52 $&$22.24$&$ 46.04$&$43.62 $&$ 70.8 $&$67.8$             \\
                & $10$   &$ 5.86$&$5.58 $&$ 6.8$&$6.36 $&$ 11.68$&$10.96 $&$ 38.52$&$36.7 $&$ 72.6$&$70.96 $&$ 93.26$&$92.5$              \\
 $200$          & $5$    &$ 4.90 $&$4.78  $&$ 7.08 $&$ 6.82 $&$ 15 $&$ 14.82  $&$ 47.8 $&$ 46.96  $&$ 81.98 $&$ 81.44 $&$ 96.54 $&$96.16  $  \\
                & $10$   &$ 4.66 $&$4.50 $&$ 7.18 $&$6.80 $&$ 22 $&$ 21.12  $&$ 74.1 $&$ 73.48  $&$ 97.78 $&$ 97.66  $&$ 99.98 $&$ 99.96  $  \\
 $500$          & $5$    &$ 5.22$&$5.24 $&$ 10.4$&$10.36 $&$ 40.22$&$39.9 $&$ 92.36$&$92.14 $&$ 100$&$99.76 $ & $100$&$100$ \\
                & $10$   &$ 5.46$&$5.38 $&$ 13.98$&$13.62 $&$ 63.04$&$62.78 $&$ 99.54$&$99.52$&$ 100$&$100 $ &$ 100$&$100 $  \\
\hline
\end{tabular}
\end{table}

\begin{figure}
\centering
{\includegraphics[width=0.4\textwidth]{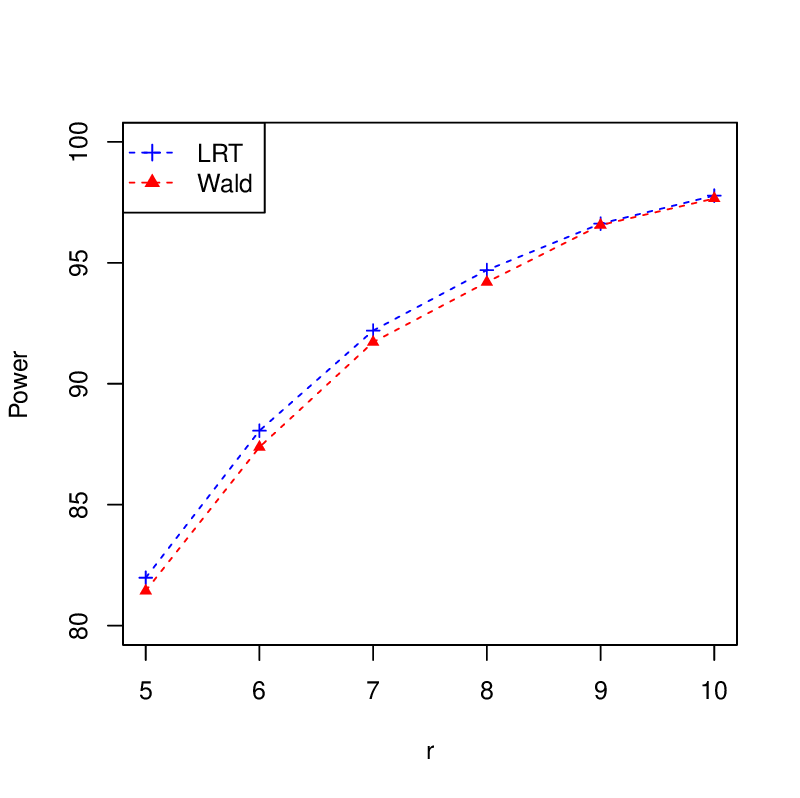}}
\caption{Trends of powers between LRT and Wald Tests changing with $r$ ($n=200, c=1$).}
\label{fig2}
\end{figure}

\subsection{Testing degree heterogeneity in real-world data}

We apply likelihood ratio statistics to test degree heterogeneity in three real-world network
datasets, namely the email-Eu-core network \citep{yin2017local}, UC Irvine messages network \citep{opsahl2009clustering}
and US airports network \citep{kunegis2013konect}.
All three networks are openly available. The first can be downloaded from
\url{https://snap.stanford.edu/data/email-Eu-core.html} and the latter two from \url{http://www.konect.cc/networks}.

The email-Eu-core network consists of $1,005$ nodes and $25,571$ directed edges, where a node denotes a member in a European research institution.
A directed edge between two members represents who sent an email to whom at least once.
The UC Irvine messages network consists of $1,899$ nodes, where a node represents a user.
The original data are weighted but we transform them into an unweighted network  as in \cite{Hu01102020} and \cite{Wang03042023},
where a directed edge between two users denotes who sent a message to whom at least once.
A total of $20,296$ directed edges are generated.
The US airports network illustrates the flight connections between $1,574$ US airports.
The original data are weighted but we treated them as an unweighted network, where a directed edge denotes
that there is a flight from one airport to another at least once.
As a result, there are a total of $28,236$ directed edges.

We test whether there is degree heterogeneity in the subset of nodes. Since we study three different real-world data sets, we choose nodes with the $r$ largest id numbers for illustrating examples in a unified manner. We could also test other sets of nodes of interest.
Table \ref{Test} reports the p-values for testing $H_{01}: \alpha_{n-r+1}=\cdots = \alpha_n$ or $H_{02}: \beta_{n-r+1}=\cdots = \beta_{n-1}$,
where $\beta_n=0$ is the condition for model identification.
One can see that in the email-Eu-core data, the p-value is insignificant when testing $\alpha_{n-1}=\alpha_n$ but
begins to become significant when $r\ge 5$. When $r$ increases to $10$, it is very significant, indicating strong evidence for out-degree heterogeneity.
For testing in-degree heterogeneity, it is insignificant when $r=5$ but very significant when $r=10$.
A similar phenomenon can also observed in the UC Irvine messages dataset, except that it requires a larger $r$ to detect degree heterogeneity.
However, the US airports data exhibit a different phenomenon. When testing $H_{01}$ or $H_{02}$ with $r=10, 20, 50, 100, 200, 300$, the p-values are all insignificant,
indicating that there is no degree heterogeneity in a large set of nodes.
When $r$ increases to $400$, the p-value becomes significant.  Furthermore, the p-values under $r=500$ are much smaller than those under $r=400$.
These results demonstrate that degree heterogeneity is a common phenomenon and varies across different sets of nodes in different datasets.

\begin{table}
\centering
\caption{Testing degree heterogeneity in several network datasets.}
\label{Test}
\vskip5pt
\small
\begin{tabular}{cc cccc cc}
\hline
&&\multicolumn{1}{c}{$\bs\alpha$} &&&\multicolumn{1}{c}{$\bs\beta$}\\
\cmidrule(r){2-4} \cmidrule(r){5-7}
 Data set          & $r$  & \multicolumn{2}{c}{p-value} & $r$  & \multicolumn{2}{c}{p-value} \\
 \cmidrule(r){3-4} \cmidrule(r){6-7}
                   &      & LRT & Wald &      & LRT & Wald \\
\hline
email-Eu-core      & $2$  & 0.432 & 0.239 & $5$   & 0.190 & 0.232 \\
                   & $5$  & 0.015 & 0.037 & $10$  & $<10^{-5}$ & 0.027 \\
                   & $15$ & $<10^{-5}$ & $<10^{-5}$ & $50$ & $<10^{-5}$ & $<10^{-5}$ \\
\hline
UC Irvine messages & $5$  & 0.284 & 0.216 & $10$  & 0.443 & 0.503 \\
                   & $10$ & $7\times10^{-3}$ & 0.019 & $30$ & 0.024 & 0.088 \\
                   & $50$ & $<10^{-5}$ & $<10^{-5}$ & $50$ & $<10^{-5}$ & $<10^{-5}$ \\
\hline
US airports        & $300$ & 0.816 & 0.732 & $300$ & 0.891 & 0.499 \\
                   & $400$ & $8\times10^{-5}$ & $1\times10^{-5}$ & $400$ & $8\times10^{-4}$ & $<10^{-5}$ \\
                   & $500$ & $<10^{-5}$ & $<10^{-5}$ & $500$ & $<10^{-5}$ & $<10^{-5}$ \\
\hline
\end{tabular}
\end{table}

\section{Discussion}
\label{section:discussion}

We studied the asymptotic behaviors of LRTs in the $p_{0}$ model. Wilks-type results were established
for fixed- and increasing-dimensional parameter hypothesis testing problems.
It is worth noting that the conditions imposed on $b_{n}$ and $c_n$ could not be optimal.
Even in cases where these prerequisites are not met, the simulation results demonstrate that there are still reasonable asymptotic approximations.
It is crucial to note that the asymptotic behaviors of likelihood ratio statistics are influenced not only by $b_{n}$ (or $c_n$) but also by the entire parameter configuration.

In the real data analysis, we transform the weighted edges in the UC Irvine and Airports datasets into binary edges as in \cite{Hu01102020} and \cite{Wang03042023}.
This conversion may lead to some potential information loss.
On the other hand, if different weighting-to-binary thresholds are used to
construct the unweighted networks, the inference results may differ.
Although this study focused on binary edges,
there is potential to extend the edge generation scheme from Bernoulli to other discrete or continuous distribution,
where the edge weight may take a set of finite discrete values, the positive natural number or any real number \cite[e.g.][]{Yan:Zhao:Qin:2015,Yan:Leng:Zhu:2016,li2027jssc}.
It is beyond of the current
paper to investigate it. We would like to study this issue in the future.

\section{Appendix}
\label{section:proof}

We introduce some notations.
For a vector $\mathbf{x}=(x_1, \ldots, x_n)^\top\in \R^n$,
let $\|\mathbf{x}\|$ be a general norm on vectors with the special cases
$\|\mathbf{x}\|_\infty = \max_{1\le i\le n} |x_i|$ and $\|\mathbf{x}\|_1=\sum_i |x_i|$ for the $\ell_\infty$- and $\ell_1$-norm of $\mathbf{x}$ respectively.
For an $n\times n$ matrix $J=(J_{ij})$, $\|J\|_\infty$ denotes the matrix norm induced by the $\ell_\infty$-norm on vectors in $\R^n$, i.e.,
\[
\|J\|_\infty = \max_{\mathbf{x}\neq 0} \frac{ \|J\mathbf{x}\|_\infty }{\|\mathbf{x}\|_\infty}
=\max_{1\le i\le n}\sum_{j=1}^n |J_{ij}|,
\]
and $\|J\|$ is a general matrix norm. $\|J\|_{\max}$ denotes the maximum absolute entry-wise norm, i.e., $\|J\|_{\max}=\max_{i,j} |J_{ij}|$.
The  notation $f(n)=O\left(g(n)\right)$ or
$f(n)\lesssim g(n)$ means  there is a constant $c>0$ such
that $\left|f(n)\right|\leq c|g(n)|$. $f(n) \asymp g(n)$ means that $f(n)\lesssim g(n)$ and $g(n)\lesssim f(n)$.
$f(n)=o(g(n))$ means  $\lim_{n\rightarrow\infty}f(n)/g(n)=0$.

Define one matrices $S=(s_{i,j})_{(2n-1)\times (2n-1)}$ with
\begin{equation}\label{definition-S}
  s_{i,j}=
   \begin{cases}
  \displaystyle \frac{\delta_{i,j}}{v_{i,i}}+\frac{1}{v_{2n,2n}}, &i,j=1,\ldots,n,\\
    \displaystyle -\frac{1}{v_{2n,2n}}, &i=1,\ldots,n,j=n+1,\ldots,2n-1,\\
     \displaystyle -\frac{1}{v_{2n,2n}}, &i=n+1,\ldots,2n-1,j=1,\ldots,n,\\
      \displaystyle\frac{\delta_{i,j}}{v_{i,i}}+\frac{1}{v_{2n,2n}}, &i,j=n+1,\ldots,2n-1.\\
   \end{cases}
  \end{equation}
\cite{Yan:Leng:Zhu:2016} proposed to use the matrix $S$ to approximate $V^{-1}$.

Since $a_{i,j}, 1\leq i\neq j\leq n$ are mutually independent Bernoulli random variables, by Theorem 2.1 of \cite{chernozhuokov2022improved},
the vector $(\frac{\bar d_1}{\sqrt{v_{1,1}}},  \frac{\bar d_2 }{\sqrt{v_{2,2}}}, \cdots, \frac{\bar d_r}{\sqrt{v_{r,r}}} )^{\top}$ with an increasing dimension $r$
asymptotically follows a high-dimensional standard normal distribution, where $\bar{d}_i = d_i - \E d_i$.
Therefore, we have the following lemma, whose proofs are omitted.

\begin{lemma}\label{lemma:weighte-degree-al}
Under the $p_{0}$-model,
if  $b_{n}^{2}\log^{5} n/n=o(1)$, then
$\sum_{i=1}^r \bar{d}_i^{\,2}/v_{i,i}$ is asymptotically normally distributed with mean $r$ and variance $2r$,
where $\bar{d}_i= d_i - \E d_i$.
\end{lemma}
We define a function $\mu(x) =  e^x/(1 + e^x)$ and a notation $\pi_{ij}=\alpha_i+\beta_j$ for easy of exposition.
A direct calculation gives that the derivative of $\mu(x)$ up to the third order are
\begin{eqnarray}\label{eq-derivative-mu-various}
\mu^\prime(x) = \frac{e^x}{ (1+e^x)^2 },~~  \mu^{\prime\prime}(x) = \frac{e^x(1-e^x)}{ (1+e^x)^3 },~~ \mu^{\prime\prime\prime}(x) =  \frac{ e^x [ (1-e^x)^2 - 2e^x] }{ (1 + e^x)^4 }.
\end{eqnarray}
According to the definition of $c_n$ in \eqref{definition-bncn}, we have the following inequalities:
\begin{equation}\label{ineq-mu-deriv-bound}
|\mu^\prime(\pi_{ij})| \le \frac{1}{c_n}, ~~ |\mu^{\prime\prime}(\pi_{ij})| \le \frac{1}{c_n},~~ |\mu^{\prime\prime\prime}(\pi_{ij})| \le \frac{1}{c_n}.
\end{equation}
The above inequalities will be used in the proofs repeatedly.
Recall that $\bar{a}_{i,j} = a_{i,j} - \E(a_{i,j})$ denotes
the centered random variable of $a_{i,j}$ and define $\bar{a}_{i,i}=0$ for all $i=1, \ldots, n$.
Correspondingly, denote $\bar{d}_i = d_i - \E(d_i)=\sum_j \bar{a}_{i,j}$, $\bar{b}_j = b_j - \E(b_j)= \sum_i \bar{a}_{i,j}$ and
$\bs{\bar{\mathbf{g}}}=(\bar{d}_1, \ldots, \bar{d}_n, \bar{b}_1,\ldots, \bar{b}_{n-1})^\top$.

\subsection{Proofs for Theorem \ref{theorem-LRT-beta} (a)}
\label{section:theorem1-b}
Let $\bs{\widetilde{\mathbf{g}}}=(\sum_{i=1}^r d_i, d_{r+1},\ldots, d_n,b_{1},\ldots,b_{n-1})^{\top}$ and
 $\widetilde{V}$ denote the Fisher information matrix of $\widetilde{\bs{\theta}}=(\alpha_1, \alpha_{r+1}, \ldots, \alpha_n,\beta_1, \ldots, \beta_{n-1})^\top$
under the null $H_0: \alpha_1 = \cdots= \alpha_r$ with $r\leq n$, where
\begin{equation}\label{definition-tilde-V}
\widetilde{V}=\begin{pmatrix} \tilde{v}_{11} & \bs{\tilde{v}}_{12}^\top \\ \bs{\tilde{v}}_{12} & V_{22} \end{pmatrix},
\end{equation}
where $V_{22}$ is the lower right $(2n-1-r)\times (2n-1-r)$ block of $V$, $\bs{\tilde{v}}_{12} =
(\tilde{v}_{1,r+1}, \ldots, \tilde{v}_{1, 2n-1})^\top$, and
\begin{align*}
  \tilde{v}_{11}& =\sum_{j\neq 1}\frac{  e^{\alpha_1 + \beta_j } }{ ( 1 + e^{\alpha_1 + \beta_j})^2 }+\cdots+ \sum_{j\neq r}\frac{  e^{\alpha_1 + \beta_j } }{ ( 1 + e^{\alpha_1 + \beta_j})^2 },\\
  \tilde{v}_{1,i} & = 0,~i=r+1, \ldots, n,\\
  \tilde{v}_{1,i}&=\frac{ (r-1) e^{\alpha_1 + \beta_j } }{ ( 1 + e^{\alpha_1 + \beta_j})^2 },~i=n+1, \ldots, n+r,~j=i-n,\\
   \tilde{v}_{1,i}&=\frac{ r e^{\alpha_1 + \beta_j } }{ ( 1 + e^{\alpha_1 + \beta_j})^2 },~i=n+r+1, \ldots, 2n-1,~j=i-n.
\end{align*}
Note that $\widetilde{V}$ is also the covariance matrix of $\bs{\widetilde{\mathbf{g}}}$.
Similar to approximate $V^{-1}$ by $S$, we use $\widetilde{S}$ to approximate $\widetilde{V}^{-1}$, where
$\widetilde{S}$($=(s_{i,j})_{(2n-r)\times (2n-r)}$)
\begin{equation}\label{definition-S222}
  s_{i,j}=
   \begin{cases}
  \displaystyle \frac{1}{\tilde{v}_{11}}+\frac{1}{{v}_{2n,2n}}, &i,j=1,\\
    \displaystyle \frac{1}{{v}_{2n,2n}}, &i=1,j=r+1,\ldots,n;i=r+1,\ldots,n,j=1,\\
     \displaystyle -\frac{1}{{v}_{2n,2n}}, &i=1,\ldots,n,j=n+1, \ldots,2n-1;i=n+1, \ldots,2n-1,j=1,\ldots,n,\\
      \displaystyle\frac{\delta_{i,j}}{{v}_{i,i}}+\frac{1}{{v}_{2n,2n}}, &i,j=r+1,\ldots,n;~~i,j=n+1,\ldots,2n-1.\\
   \end{cases}
  \end{equation}
The approximation error is stated in the following lemma, which extends the result of Proposition 1 of \cite{Yan:Leng:Zhu:2016}
to a general case. Therefore, we omit the proof of Lemma \ref{lemma-beta-approx-ho} here.

\begin{lemma}\label{lemma-beta-approx-ho}
For any $r\in\{0, \ldots, n-1\}$ and $n\ge 2$, we have
\begin{equation}\label{approxi-inv2-beta-ho}
\|\widetilde{W}:= \widetilde{V}^{-1}-\widetilde{S} \|_{\max} \le \frac{ b_n^{3} }{ n^2 c_n^2 }.
\end{equation}
\end{lemma}


\begin{lemma}\label{lemma-tilde-W}
For any given $r\in\{1, \ldots, n-1\}$, we have
\[
(\bs{\widetilde{\mathbf{g}}} - \E\bs{\widetilde{\mathbf{g}}} )^\top \widetilde{W}  (\bs{\widetilde{\mathbf{g}}} - \E\bs{\widetilde{\mathbf{g}}} ) = O_p\left( \frac{ b_n^3 }{ c_n^3 } \right).
\]
\end{lemma}

Recall that $\bs{\widehat{\theta}}^0$ denotes
the restricted MLE of $\bs{\theta}=(\alpha_{1},\ldots,\alpha_{n},\beta_{1},\ldots,\beta_{n-1})^\top$. Under the null $H_0: \alpha_1=\cdots = \alpha_r$,
we have $\widehat{\alpha}_1^0 = \cdots = \widehat{\alpha}_r^0$.

\begin{lemma}\label{lemma-con-beta-b}
Under the null $H_0: \alpha_1=\cdots = \alpha_r$, if
\[
\left( \frac{(n-r+2)b_{n}}{c_{n}} + \frac{b_n^3}{c_n^3} \right)\left\{ b_n + \frac{ b_n^3 }{ c_n^2} \left( \frac{ r^{1/2} }{ n } + \frac{2n-r-1}{n} \right) \right\} =o\left(\sqrt{\frac{n}{\log n}}\right),
\]
then with probability at least $1-2(2n-r)/n^2$, $\bs{\widehat{\theta}}^0$ exists and satisfies
\begin{equation*}
\| \bs{\widehat{\theta}}^0 - \bs{\theta} \|_\infty \lesssim \left(  b_n + \frac{ b_n^3 }{ c_n^2} ( \frac{r^{1/2}}{n} + \frac{2n-r-1}{n} )  \right)\sqrt{ \frac{\log n}{n} } .
\end{equation*}
Further, if  $\bs{\widehat{\theta}}^0$ exists, it must be unique.
\end{lemma}

From the above lemma, we can see that the condition to guarantee consistency and the error bound depends on
$r$.
A rough condition in regardless of $r$ to guarantee consistency is $b_n^6/c_n^5=o( (n/\log n)^{1/2})$ that
will be used in the proof of Theorem \ref{theorem-LRT-beta} (a).
It implies an error bound $(b_n^3/c_n^2)(\log n/n)^{1/2}$.

The asymptotic representation of $\bs{\widehat{\theta}}^0$ is given below.
\begin{lemma}
\label{lemma-beta-homo-expan}
Under the null $H_0: \alpha_1=\cdots=\alpha_r$,
if $b_n^6/c_n^5 = o( (n/\log n)^{1/2})$, then for any given $r\in\{0, \ldots, n-1\}$, we have
\begin{eqnarray}
\label{alpha1}
\widehat{\alpha}_1^0 - \alpha_1 & = & \frac{ \sum_{i=1}^r \bar{d}_i }{ \tilde{v}_{11} }+\frac{\bar{b}_{n}}{v_{2n,2n}}+\gamma_1^{0}, \\
\label{alphai}
 \widehat{\alpha}_i^0 - \alpha_i & = & \frac{ \bar{d}_i }{ v_{i,i} }+\frac{\bar{b}_{n}}{v_{2n,2n}} + \gamma_i^{0},~~ i=r+1, \ldots, n,\\
\label{betaj}
\widehat{\beta}_j^0 - \beta_j & = & \frac{ \bar{b}_j }{ v_{n+j,n+j} }-\frac{\bar {b}_{n}}{v_{2n,2n}} + \gamma_{n+j}^{0},~~ j=1, \ldots, n-1,
\end{eqnarray}
where $\gamma_1^{0}, \gamma_{r+1}^{0}, \ldots, \gamma_{2n-1}^{0}$ with probability at least $1- O(n^{-1})$ satisfy
\[
\gamma_i^{0} = (\widetilde{V}^{-1}\bs{\widetilde{h}})_i + [\widetilde{W} (\bs{\widetilde{\mathbf{g}}}- \E \bs{\widetilde{\mathbf{g}}})]_i = O\left( \frac{b_n^9 \log n}{ nc_n^7 } \right),
\]
uniformly, and  $\bs{\widetilde{h}}=(\tilde{h}_1, \tilde{h}_{r+1}, \ldots, \tilde{h}_{2n-1})^\top$ satisfies
\begin{equation}\label{eq-homo-tildeh}
\begin{array}{rcl}
|\tilde{h}_1| & \lesssim & \frac{rb_n^6\log n}{c_n^5}, \\
\max_{i=r+1, \ldots, 2n-1} |\tilde{h}_i| & \lesssim & \frac{ b_n^6 \log n}{c_n^5}.
\end{array}
\end{equation}
\end{lemma}

Now, we are ready to prove  Theorem \ref{theorem-LRT-beta} (a).

\proof [Proof of Theorem \ref{theorem-LRT-beta} \rm(a)]
Under the null $H_0: \alpha_1=\cdots=\alpha_r$, the data generating parameter $\bs{\theta}$ is equal to $(\underbrace{\alpha_1, \ldots, \alpha_1}_r, \alpha_{r+1}, \ldots, \alpha_n, \beta_{1}, \ldots, \beta_{n-1})^\top$.
The following calculations are based on the event $E_n$ that $\bs{\widehat{\theta}}$ and $\bs{\widehat{\theta}}^0$ simultaneously exist and satisfy
\begin{equation}\label{ineq-beta-beta0-upp1}
 \max\left\{ \| \widehat{\bs{\theta}} - \bs{\theta} \|_\infty, \|\widehat{\bs{\theta}}^0 - \bs{\theta} \|_\infty\right\}\lesssim \frac{b_n^3}{c_n^2}\sqrt{ \frac{\log n}{n} }.
\end{equation}
By Lemma \ref{lemma-con-beta-b}, $\P(E_n) \ge 1 - O(n^{-1})$ if $b_n^6/c_n^5 = o( (n/\log n)^{1/2})$.

Applying a fourth-order Taylor expansion to $\ell(\widehat{\bs{\theta}} ) $ at point $ \bs{\theta}$, it yields
\begin{eqnarray*}
\ell(\widehat{\bs{\theta}} ) - \ell( \bs{\theta} ) & = &
\underbrace{
\frac{\partial \ell( \bs{\theta} ) }{ \partial \bs{\theta}^\top } ( \widehat{\bs{\theta}} - \bs{\theta} ) +
\frac{1}{2} ( \widehat{\bs{\theta}} - \bs{\theta} )^\top \frac{ \partial^2 \ell( \bs{\theta} ) }{ \partial \bs{\theta} \partial\bs{\theta}^\top } ( \widehat{\bs{\theta}} - \bs{\theta} ) }_{B_1} \\
& & +  \frac{1}{6} \underbrace{\sum_{i=1}^{2n-1} \sum_{j=1}^{2n-1} \sum_{k=1}^{2n-1} \frac{ \partial^3 \ell(\bs{\theta})}{ \partial \theta_i \partial \theta_j \partial \theta_k }
( \widehat{\theta}_i - \theta_i)( \widehat{\theta}_j - \theta_j)( \widehat{\theta}_k - \theta_k) }_{B_2} \\
&& +  \frac{1}{4!} \underbrace{ \sum_{t=1}^{2n-1}\sum_{i=1}^{2n-1} \sum_{j=1}^{2n-1} \sum_{k=1}^{2n-1} \frac{ \partial^4 \ell(\bs{\tilde{\theta}})}{ \partial \theta_t \partial \theta_i \partial \theta_j \partial \theta_k } ( \widehat{\theta}_t - \theta_t)
( \widehat{\theta}_i - \theta_i)( \widehat{\theta}_j -\theta_j)( \widehat{\theta}_k - \theta_k) }_{B_3},
\end{eqnarray*}
where $\bs{\tilde{\theta}} = t \bs{\theta} + (1-t ) \bs{\widehat{\theta}}$ for some $t\in(0,1)$.
Correspondingly, $\ell(\widehat{\bs{\theta}}^0 )$ has the following expansion:
\begin{equation*}
\ell(\widehat{\bs{\theta}}^0 ) - \ell( \bs{\theta} ) = B_1^0 + \frac{1}{6}B_2^0 + \frac{1}{4!}B_3^0,
\end{equation*}
where $B_i^0$ is the version of $B_i$ with $\widehat{\bs{\theta}}$ replaced by $\widehat{\bs{\theta}}^0$.
Therefore,
\begin{equation}\label{eq-ell-difference}
2\{ \ell(\widehat{\bs{\theta}} ) - \ell(\widehat{\bs{\theta}}^0 ) \}
= 2( B_1 - B_1^0)  + \frac{1}{3}(B_2 - B_2^0) + \frac{1}{12}(B_3 - B_3^0).
\end{equation}
Recall that
\begin{eqnarray*}
\frac{\partial \ell( \bs{\theta} ) }{\partial \bs{\theta}^\top }  =  \bs{\mathbf{g}} - \E \bs{\mathbf{g}},
~~
V= - \frac{\partial^2 \ell( \bs{\theta} )
}{\partial \bs{\theta} \partial \bs{\theta}^\top }.
\end{eqnarray*}
Therefore, $B_1$ can be written as
\begin{equation}
\label{lrt-a-beta-B1}
B_1  =   ( \bs{\widehat{\theta} } - \bs{\theta})^\top \bs{\bar{\mathbf{g}}}
- \frac{1}{2} ( \bs{\widehat{\theta} } - \bs{\theta})^\top V( \bs{\widehat{\theta} } - \bs{\theta}).
\end{equation}
Let $\ell'_{i}(\boldsymbol{\theta})=(\ell'_{i,r+1}(\boldsymbol{\theta}),\ldots,\ell'_{i,2n-1}(\boldsymbol{\theta})):=
(\frac{\partial\ell_{i}}{\partial\alpha_{r+1}},\ldots,\frac{\partial\ell_{i}}{\partial\alpha_{n}},\frac{\partial\ell_{i}}{\partial\beta_{1}},\ldots,
\frac{\partial\ell_{i}}{\partial\beta_{n-1}}).$ For the third-order expansion terms in $B_2$, for $i=r+1,\ldots,n$, we have
\begin{equation}\label{ell-thrid}
\begin{split}
  \frac{\partial^3 \ell_{i}(\boldsymbol{\theta})}{\partial \alpha_i^3 }&=-\sum_{j\neq i} \mu^{\prime\prime}( \pi_{ij} );~~\frac{\partial^3 \ell_{i}( \bs{\theta} )
}{\partial\alpha_i\partial\alpha_j\partial\alpha_k } = 0,i\neq j\neq k,\\
\frac{\partial^3 \ell_{i}(\boldsymbol{\theta})}{\partial \alpha_i^2\partial \beta_j }&= -\mu^{\prime\prime}( \pi_{ij} ),j\neq i;~~
\frac{\partial^3 \ell_{i}(\boldsymbol{\theta})}{\partial \alpha_i^2\partial \beta_i }=0,\\
\frac{\partial^3 \ell_{i}(\boldsymbol{\theta})}{\partial \alpha_i\partial \beta_j^{2} }&= -\mu^{\prime\prime}( \pi_{ij} ),i\neq j;~~
\frac{\partial^3 \ell_{i}(\boldsymbol{\theta})}{\partial \alpha_i\partial \beta_j \partial \beta_k }=0, i\neq j\neq k,
\end{split}
\end{equation}
and if there are at least three different values among the four indices $i, j, k, t$, we have
\begin{equation}\label{ell-fourth}
 \frac{\partial^4 \ell( \bs{\theta} )
}{\partial \alpha_i \partial \alpha_j \partial \alpha_k \partial \alpha_t } = 0.
\end{equation}
For $i=n+1,\ldots,2n-1$, we also have the similar result.
Therefore, $B_2$ and $B_3$ have the following expressions:
\begin{eqnarray}
\nonumber
-B_2  & = &  \sum_{i} (\widehat{\alpha}_i-\alpha_i)^3 \sum_{j\neq i} \mu^{\prime\prime}( \pi_{ij} )
+ 3\sum_{i}\sum_{j, j\neq i}  (\widehat{\alpha}_i-\alpha_i)^2(\widehat{\beta}_j-\beta_j)\mu^{\prime\prime}( \pi_{ij} )\\
\label{lrt-a-beta-B2}
&+&\sum_{j}  (\widehat{\beta}_j-\beta_j)^3 \sum_{j\neq i} \mu^{\prime\prime}( \pi_{ij} )
+3\sum_{i}\sum_{j,j\neq i}  (\widehat{\alpha}_i-\alpha_i)(\widehat{\beta}_j-\beta_j)^2\mu^{\prime\prime}( \pi_{ij} ), \\
\nonumber
-B_3 & = &  \sum_{i}  (\widehat{\alpha}_i-\alpha_i)^4 \sum_{j\neq i} \mu^{\prime\prime\prime}( \bar{\pi}_{ij} )
 +\sum_{j}  (\widehat{\beta}_j-\beta_j)^4 \sum_{j\neq i} \mu^{\prime\prime\prime}( \bar{\pi}_{ij} )\\
\nonumber
&+&4\sum_{i}\sum_{j,j\neq i}  (\widehat{\alpha}_i-\alpha_i)(\widehat{\beta}_j-\beta_j)^3\mu^{\prime\prime\prime}( \bar{\pi}_{ij} )
+6\sum_{i}\sum_{j,j\neq i}  (\widehat{\alpha}_i-\alpha_i)^2(\widehat{\beta}_j-\beta_j)^2\mu^{\prime\prime\prime}( \bar{\pi}_{ij}) \\
\label{lrt-a-beta-B3}
&+&4\sum_{i}\sum_{j, j\neq i}  (\widehat{\alpha}_i-\alpha_i)^3(\widehat{\beta}_j-\beta_j)\mu^{\prime\prime\prime}( \bar{\pi}_{ij} ),
\end{eqnarray}
where $\bar{\pi}_{ij}$ lies between $\widehat{\pi}_{ij}$ and $\pi_{ij}$.

It is sufficient to demonstrate:
(1) $\{2( B_1 - B_1^0)-r\}/(2r)^{1/2}$ converges in distribution to the standard normal distribution;
(2) $(B_2 - B_2^0)/r^{1/2}=o_p(1)$; (3) $(B_3-B_3^0)/r^{1/2}=o_p(1)$.
The fourth-order Taylor expansion for $\ell(\bs{\widehat{\theta}}^0)$ here is with regard to the vector $(\alpha_1, \alpha_{r+1}, \ldots, \alpha_n, \beta_1, \ldots, \beta_{n-1})^\top$
because $\alpha_1, \ldots, \alpha_r$ are the same under the null here. As we shall see, the expressions of $B_1^0$ and $B_2^0$ are a little different from $B_1$ and $B_2$
except from the difference $\bs{\widehat{\theta}}$ and $\bs{\widehat{\theta}}^0$.
In view of \eqref{ineq-mu-deriv-bound} and \eqref{ineq-beta-beta0-upp1}, if $b_n^{12}/c_n^{9}=o( r^{1/2}/(\log n)^2 )$, then
\begin{eqnarray}
\label{ineq-B3-upper}
\frac{ |B_3| }{ r^{1/2} } & \lesssim & \frac{1}{r^{1/2}} \cdot
\frac{n^2}{c_n} \cdot \| \bs{\widehat{\theta}} - \bs{\theta} \|_\infty^4 \lesssim \frac{ b_n^{12}(\log n)^2 }{ r^{1/2} c_n^{9} } = o(1), \\
\label{ineq-B30-upper}
\frac{ |B_3^0| }{ r^{1/2} } & \lesssim & \frac{1}{r^{1/2}} \cdot
\frac{n(n-r)}{c_n} \cdot \| \bs{\widehat{\theta}}^0 - \bs{\theta} \|_\infty^4 \lesssim \frac{ b_n^{12}(\log n)^2 }{ r^{1/2} c_n^{9} } = o(1),
\end{eqnarray}
which shows the third claim.
For second claim, it is sufficient to show $B_2^0/r^{1/2}=o_p(1)$.
Under the null $H_0:  \alpha_1=\cdots=\alpha_r$,  $B_2^0$ can be written as
\begin{eqnarray*}
-B_2^0  & = & \sum_{i=1}^{r}\sum_{j\neq i}\mu^{\prime\prime}(\pi_{ij}) (\widehat{\alpha}_1^0 - \alpha_1)^3
+\sum_{j=1}^{r} \sum_{i=r+1,i\neq j}^n (\mu^{\prime\prime}( \pi_{ij})+(r-1)\mu^{\prime\prime}(\pi_{1j}))( \widehat{\beta}_j^0 - \beta_j)^3  \\
&&+\sum_{j=r+1}^{n-1} \sum_{i=r+1,i\neq j}^n (\mu^{\prime\prime}( \pi_{ij})+r\mu^{\prime\prime}(\pi_{1j}))( \widehat{\beta}_j^0 - \beta_j)^3+ \sum_{i=r+1}^n \sum_{j\neq i}\mu^{\prime\prime}( \pi_{ij}) ( \widehat{\alpha}_i^0  - \alpha_i)^3 \\
&&+3(r -1)\sum_{j=1}^r \mu^{\prime\prime}(\pi_{1j}) ( \widehat{\alpha}_1^0  - \alpha_1)^2 ( \widehat{\beta}_j^0 - \beta_j)
+3r\sum_{j=r+1}^{n-1} \mu^{\prime\prime}(\pi_{1j}) ( \widehat{\alpha}_1^0  - \alpha_1)^2 ( \widehat{\beta}_j^0 - \beta_j)\\
&&+3(r -1)\sum_{j=1}^r \mu^{\prime\prime}(\pi_{1j}) ( \widehat{\alpha}_1^0  - \alpha_1) ( \widehat{\beta}_j^0 - \beta_j)^2
+3r\sum_{j=r+1}^{n-1} \mu^{\prime\prime}(\pi_{1j}) ( \widehat{\alpha}_1^0  - \alpha_1) ( \widehat{\beta}_j^0 - \beta_j)^2\\
&&+3\sum_{i=r+1}^n\sum_{j\neq i} \mu^{\prime\prime}(\pi_{ij}) ( \widehat{\alpha}_i^0  - \alpha_i)^2 ( \widehat{\beta}_j^0 - \beta_j)
+3\sum_{i=r+1}^n\sum_{j\neq i} \mu^{\prime\prime}(\pi_{ij}) ( \widehat{\alpha}_i^0  - \alpha_i) ( \widehat{\beta}_j^0 - \beta_j)^2.
\end{eqnarray*}
With the use of the asymptotic representation of $\bs{\widehat{\theta}}^0$ in Lemma \ref{lemma-beta-homo-expan}, if
\[
b_n^{11}/c_n^8=o\left( \frac{n}{(\log n)^2r} \right)
 \mbox{~~and~~} b_n^{11}/c_n^8=o\left( \frac{n}{(\log n)^2(n-r)} \right),
\]
then we have
\begin{equation}\label{eq-thereom1b-z43}
\frac{B_2^0}{r^{1/2}} = o_p( 1 ),
\end{equation}
whose detailed calculations are given in the supplementary material.

Next, we show first claim. This contains three steps.
Step 1 is about explicit expressions of $\widehat{\bs{\theta}}$ and $\widehat{\bs{\theta}}^0$.
Step 2 is about the explicit expression of $B_1-B_1^0$. Step 3 is about showing that the main term involved with $B_1-B_1^0$
asymptotically follows a normal distribution and the remainder terms goes to zero.

Step 1. We characterize the asymptotic representations of $\widehat{\bs{\theta}}$ and $\widehat{\bs{\theta}}^0$.
Recall that $\pi_{ij}=\alpha_i+\beta_j$.
To simplify notations, define $\widehat{\pi}_{ij} = \widehat{\alpha}_i + \widehat{\beta}_j$.
A second-order Taylor expansion gives that
\begin{eqnarray*}\label{eq-expansion-beta-a}
\mu( \widehat{\pi}_{ij} )
&=& \mu( \pi_{ij} ) + \mu^\prime(\pi_{ij}) (\widehat{\pi}_{ij} - \pi_{ij}) +
\frac{1}{2} \mu^{\prime\prime}( \tilde{\pi}_{ij} ) (\widehat{\pi}_{ij} - \pi_{ij})^2,
\end{eqnarray*}
where $\tilde{\pi}_{ij}$ lies between $\widehat{\pi}_{ij}$ and $\pi_{ij}$.
Let\begin{equation}\label{eq:definition:h}
\begin{split}
  h_{i,j} & = \frac{1}{2}\mu^{\prime\prime}( \tilde{\pi}_{ij} )(\widehat{\pi}_{ij} - \pi_{ij})^2,
    h_{i} =\sum_{k=1;k\neq i}^{n}h_{i,k}, ~~i=1,\ldots,n,\\
    h_{n+i}&=\sum_{k=1;k\neq i}^{n}h_{k,i}, ~~i=1,\ldots,n-1,
    \textbf{h}=(h_{1},\ldots,h_{2n-1})^{\top}.
\end{split}
\end{equation}

In view of \eqref{ineq-mu-deriv-bound}, we have $|h_{i,j}|\leq \frac{1}{2c_{n}}\| \widehat{\boldsymbol{\theta}} - \boldsymbol{\theta}\|_\infty^2,
  |h_{i}|\leq\sum_{j=1;j\neq i}^{n}\Big|h_{i,j}\Big|\leq \frac{ n-1}{2c_{n}}\| \widehat{\boldsymbol{\theta}} - \boldsymbol{\theta}\|_\infty^2.$
Writing the above equations into the matrix form, we have
\begin{equation*}
\bs{\mathbf{g}} - \E( \bs{\mathbf{g}} ) = V ( \widehat{\boldsymbol{\theta}} - \boldsymbol{\theta} ) + \bs{h}.
\end{equation*}
It yields that
\begin{equation}\label{eq-expansion-hatbeta-beta1}
\boldsymbol{\widehat{\theta}} - \boldsymbol{\theta} = V^{-1} \bs{\bar{\mathbf{g}}} - V^{-1}\mathbf{h}.
\end{equation}
Recall  $\bs{\widetilde{\mathbf{g}}}=(\sum_{i=1}^r d_i, d_{r+1},\ldots, d_n,b_{1},\ldots,b_{n-1})^{\top}$. Let $\boldsymbol{\widetilde{\theta}}^0= (\widehat{\alpha}_{1}^0,\widehat{\alpha}_{r+1}^0, \ldots, \widehat{\alpha}_n^0, \widehat{\beta}_{1}^{0}, \ldots, \widehat{\beta}_{n-1}^{0})^\top$ and
$\boldsymbol{\widetilde{\theta}}= ({\alpha}_{1},{\alpha}_{r+1}, \ldots, {\alpha}_n, {\beta}_{1}, \ldots,{\beta}_{n-1})^\top$.
Similar to \eqref{eq-expansion-hatbeta-beta1}, we have
\begin{equation}\label{eq-beta0-exapnsion1}
\boldsymbol{\widetilde{\theta}}^0 - \boldsymbol{\widetilde\theta} = \widetilde{V}^{-1}\bs{\widetilde{\mathbf{g}}}  - \widetilde{V}^{-1}\widetilde{\mathbf{{h}}},
\end{equation}
where $\mathbf{\widetilde{h}} = (\tilde{h}_{1}, \tilde{h}_{r+1},\ldots, \tilde{h}_{2n-1})^\top$ and
\begin{equation}\label{defintion-tilde-h}
 \tilde h_{i,j}  = \frac{1}{2}\mu^{\prime\prime}( \tilde{\pi}_{ij}^{0} )(\widehat{\pi}_{ij}^{0} - \pi_{ij})^2.
\end{equation}
In the above equation, $\tilde{\pi}_{ij}^0$ lies between $\pi_{ij}$ and $\widehat{\pi}_{ij}^0=\widehat{\alpha}_i^0 + \widehat{\beta}_j^0$.

Step 2. We derive the explicit expression of $B_1-B_1^0$.
Substituting \eqref{eq-expansion-hatbeta-beta1} and \eqref{eq-beta0-exapnsion1} into the expressions of $B_1$ in \eqref{lrt-a-beta-B1} and $B_1^0$ respectively, it yields
\begin{eqnarray*}\label{B1-expression}
2B_1 & = & \bs{\bar{\mathbf{g}}}^\top V^{-1} \bs{\bar{\mathbf{g}}} - \bs{h}^\top V^{-1} \bs{h},\\
\label{likelihood-beta-composite}
2B_1^0 & = &( \bs{\widetilde{\mathbf{g}}} - \E \bs{\widetilde{\mathbf{g}}} )^\top \widetilde{V}^{-1} ( \bs{\widetilde{\mathbf{g}}} - \E \bs{\widetilde{\mathbf{g}}} )
-\bs{\widetilde{h}}^\top \widetilde{V}^{-1} \bs{\widetilde{h}}.
\end{eqnarray*}
By using $\widetilde{S}$ and $S$ to approximate $\widetilde{V}^{-1}$ and $V^{-1}$ respectively,
we have
\begin{equation}\label{eq-theorem2-B101}
2(B_1 - B_1^0) =  \sum_{i=1}^r \frac{ \bar{d}_i^{\,2} }{ v_{i,i} } - \frac{ ( \sum_{i=1}^{r}\bar{d}_{i}) ^2 }{ \tilde{v}_{11} }
 +\bs{\bar{\mathbf{g}}}^\top W \bs{\bar{\mathbf{g}}} - ( \mathbf{\widetilde{\mathbf{g}}} - \E \mathbf{\widetilde{\mathbf{g}}} )^\top \widetilde{W} ( \mathbf{\widetilde{\mathbf{g}}} - \E \mathbf{\widetilde{\mathbf{g}}} )
-\bs{{h}}^\top V^{-1} \bs{{h}} + \bs{\widetilde{h}}^\top \widetilde{V}^{-1} \bs{\widetilde{h}}.
\end{equation}

Step 3. By Proposition 1 in \cite{Yan:Leng:Zhu:2016}, we have
\[
|\mathbf{h}^\top V^{-1} \mathbf{h}| \le (2n-1) \| \mathbf{h} \|_\infty \| V^{-1} \mathbf{h} \|_\infty
\lesssim
 (2n-1) \cdot \frac{b_n^6\log n}{c_n^{5}} \cdot \frac{ b_n^9 \log n}{nc_n^{7}} \lesssim \frac{b_n^{15} (\log n)^2}{c_n^{12}}.
\]
If $b_n^{15}/c_n^{12} =o( r^{1/2}/(\log n)^{2})$, then
\begin{equation}\label{eq-simi-aVh}
\frac{1}{r^{1/2}} |\mathbf{h}^\top V^{-1} \mathbf{h}| \lesssim \frac{ b_n^{15}(\log n)^2 }{r^{1/2}c_n^{12}}  = o_{p}(1).
\end{equation}
In view of \eqref{approxi-inv2-beta-ho} and \eqref{eq-homo-tildeh}, setting $\widetilde{V}^{-1}=\widetilde{S}+\widetilde{W}$ yields
\begin{eqnarray}
\nonumber
\bs{\widetilde{h}}^\top \widetilde{V}^{-1} \bs{\widetilde{h}} & \le &
\underbrace{\frac{\tilde{h}_1^2}{\tilde{v}_{11}} + \sum_{i=r+1}^{2n-1} \frac{ h_i^2 }{ v_{i,i} }+\frac{\tilde{h}_{2n}^2}{{v}_{2n,2n}} } + \underbrace{ |\tilde{w}_{11}| \tilde{h}_1^2 + \| \widetilde{W} \|_{\max} \left(
2 |\tilde{h}_1| \sum_{i=r+1}^{2n-1} |\tilde h_i| + \sum_{i,j=r+1}^{2n-1} |\tilde h_i| |\tilde h_j| \right)}\\
\nonumber
& \lesssim &  b_n \left( \frac{b_n^6 \log n}{ c_n^5}  \right)^2 + \frac{ b_n^3 }{ n^2 c_n^2 } \left\{ r^2 + 2r(2n-r-1)+(2n-r-1)^2 \right\} \left( \frac{b_n^6 \log n}{ c_n^5}  \right)^2 \\
\label{eq-thereom1b-a1}
& \lesssim & \frac{ b_n^{15} (\log n)^2 }{ c_n^{12}}.
\end{eqnarray}
This shows that if $b_n^{15}/c_n^{12} = o\left( r^{1/2}/(\log n)^2 \right)$,
\begin{equation*}\label{eq-thereom1b-a}
\frac{ |\bs{\widetilde{h}}^\top \bs{\widetilde{V}}^{-1} \bs{\widetilde{h}} | }{\sqrt{r}} = o_p(1).
\end{equation*}

By Lemma 
\ref{lemma-tilde-W}, if $b_n^3/c_n^3=o(r^{1/2})$, then
\begin{equation}\label{eq-theorem1b-c}
\frac{1}{r^{1/2}} \max\{ \bs{\bar{\mathbf{g}}} W \bs{\bar{\mathbf{g}}}, ( \mathbf{\widetilde{\mathbf{g}}} - \E \mathbf{\widetilde{\mathbf{g}}} )^\top \widetilde{W} ( \mathbf{\widetilde{\mathbf{g}}} - \E \mathbf{\widetilde{\mathbf{g}}} ) \} = o_p(1).
\end{equation}
Since $\sum_{i=1}^{r}\bar{d}_i=\sum_{i=1}^{r}\sum_{j\neq i}\bar{a}_{i,j}$, by the central limit theorem for bounded case [\cite{Loeve:1977} (p.289)], $\tilde{v}_{11}^{-1/2}{\sum_{i=1}^{r}\bar{d}_i}$ converges in distribution to the standard normal distribution if $\tilde{v}_{11}\rightarrow\infty$. Therefore, as $r\rightarrow\infty$,
\[
\frac{ [\sum_{i=1}^r \{ d_i-\E(d_i) \}]^2/\tilde{v}_{11} }{ r } = o_p(1).
\]
By combining \eqref{eq-theorem2-B101}, \eqref{eq-thereom1b-a1} and \eqref{eq-theorem1b-c}, it yields
\[
\frac{2(B_1-B_1^0)}{\sqrt{2r}}
= \frac{1}{\sqrt{2r}} \sum_{i=1}^r \frac{ ( d_i- \E d_i )^2 }{ v_{i,i} }  + o_p(1).
\]
Therefore, the first claim immediately follows from Lemma \ref{lemma:weighte-degree-al}.
This completes the proof.

\subsection{Proofs for Theorem \ref{theorem-LRT-beta} (b)}
\label{section:theorem2}
Recall that $\bs{\bar{\mathbf{g}}}_2 = ( \bar{d}_{r+1}, \ldots, \bar{d}_n, \bar{b}_{1}, \ldots, \bar{b}_{n-1})^\top$, $\bs{\widetilde{\mathbf{g}}}=(\sum_{i=1}^r d_i, d_{r+1},\ldots, d_n,b_{1},\ldots,b_{n-1})^{\top}$ and $\widetilde{V}$ is given in \eqref{definition-tilde-V}.
$\widetilde{V}$ is the Fisher information matrix of $\widetilde{\bs{\theta}}=(\alpha_1, \alpha_{r+1}, \ldots, \alpha_n, \beta_1, \ldots, \beta_{n-1})^\top$
under the null $H_0: \alpha_1 = \cdots= \alpha_r$.
Let $\bs{\bar{\mathbf{g}}}_1 = ( \bar{d}_1, \ldots, \bar{d}_r)^\top$. Remark that $r$ is a fixed constant in this section.
Partition $V$, $W$ and $\widetilde{W}$ into four blocks
\begin{equation}\label{VW-divide}
V =
\begin{pmatrix} V_{11}  & V_{12} \\
V_{21} & V_{22}
\end{pmatrix}, ~~
W = \begin{pmatrix} W_{11} & W_{12} \\
W_{21} & W_{22}
\end{pmatrix},~~
\widetilde{W} = \begin{pmatrix} \tilde{w}_{11} & \bs{\tilde{w}}_{12} \\
\bs{\tilde{w}}_{21} & \widetilde{W}_{22}
\end{pmatrix},
\end{equation}
where $V_{11}$ and $W_{11}$ are respective $r\times r$ dimensional sub-matrices of $V$ and $W$, and $W=V^{-1}-S$.
$\tilde{w}_{11}$ is a scalar and the dimension of $\widetilde{W}_{22}$ is $(2n-r-1)\times (2n-r-1)$.

To prove Theorem \ref{theorem-LRT-beta} (b), we need the following four lemmas. The following lemma gives the upper bounds of five remainder terms in \eqref{eq-theorem2b-B1022} that tend to zero. The proof of Lemma \ref{lemma-W-widetilde-d-2b} is similar to that of Lemma 14 in \cite{yan2025likelihood} and is omitted.

\begin{lemma}\label{lemma-W-widetilde-d-2b}
Suppose $r$ is a fixed constant. \\
(a)If $b_n^3/c_n^2=o( n^{3/2}/(\log n)^{1/2})$, then $\mathbf{\bar{\mathbf{g}}}_1^\top W_{11} \mathbf{\bar{\mathbf{g}}}_1 = o_p(1)$. \\
(b)If $b_n^3/c_n^3=o( n^{1/2} )$, then $\mathbf{\bar{\mathbf{g}}}_1^\top W_{12} \mathbf{\bar{\mathbf{g}}}_2 = o_p(1)$. \\
(c)If $b_n^3/c_n^2=o( n^{3/2}/(\log n)^{1/2})$, then $(\sum_{i=1}^r\bar{d}_i) \tilde{w}_{11} (\sum_{i=1}^r\bar{d}_i) = o_p(1)$. \\
(d)If $b_n^3/c_n^3=o( n^{1/2} )$, then $(\sum_{i=1}^r\bar{d}_i) \bs{\tilde{w}}_{12}^\top \bar{\mathbf{g}}_2 = o_p(1)$. \\
(e)If $b_n^3/c_n^3=o( n^{3/4} )$, then
\[
\bar{\mathbf{g}}_2^\top ( W_{22} - \widetilde{W}_{22}) \bar{\mathbf{g}}_2 = o_p(1).
\]
\end{lemma}

The lemma below gives an upper bound of $\| W_{22} - \widetilde{W}_{22} \|_{\max}$.

\begin{lemma}\label{w2-error-2b}
For a fixed constant $r$, the entry-wise error between $W_{22}$ and $\widetilde{W}_{22}$ is
\begin{equation}\label{ineq-WW-diff-2b}
\begin{split}
( W_{22} - \widetilde{W}_{22})_{ij} & =H_{2}+\frac{1}{v_{2n,2n}}H_{1}\begin{array}{ccc}
\begin{pmatrix} \mathbf{1}& \mathbf{-1}\\
\mathbf{-1}& \mathbf{1} \\
\end{pmatrix}
&
\begin{tiny}
\begin{array}{l}
\left.\rule{-5mm}{4mm}\right\}{n-r}\\
\left.\rule{-5mm}{4mm}\right\}{n-1}
\end{array}
\end{tiny}
\\
\begin{tiny}
\begin{array}{cc}
\underbrace{\rule{8mm}{-15mm}}_{n-r}
\underbrace{\rule{8mm}{-15mm}}_{n-1},
\end{array}
\end{tiny}
\end{array}
\end{split}
\end{equation}
where
\begin{equation}\label{H1h2}
\|H_{1}\|_{\max}=O(\frac{ b_n^4 }{ n^2c_n^3 }),~~\|H_{2}\|_{\max}=O(\frac{ b_n^6 }{ n^3c_n^5 }).
\end{equation}
\end{lemma}

The lemma below establishes the upper bound of $ \widehat{\theta}_i - \widehat{\theta}_i^0 $. This lemma is a direct conclusion of Theorem 2 in \cite{Yan:Leng:Zhu:2016} and Lemma 8.

\begin{lemma}\label{lemma-hat-beta-diff-2b}
Under the null $H_0: \alpha_1=\cdots = \alpha_r$ with a fixed $r$,
if $b_n^6/c_n^{5} = o( (n/\log n)^{1/2})$, then with probability at least $1-O(n^{-1})$,
\[
\max_{i=r+1, \ldots, 2n-1} | \widehat{\theta}_i - \widehat{\theta}_i^0 | \lesssim \frac{b_n^9 \log n}{nc_n^7}.
\]
\end{lemma}

The above error bound is in the magnitude of $n^{-1}$, up to a factor $b_n^9\log n$,
which makes the remainder terms in \eqref{eq-theorem2b-B1022} be asymptotically neglected.

\begin{lemma}\label{lemma-gg0}
If $b_n^6/c_n^5=o( (n/\log n)^{1/2})$, then $\gamma_i - \gamma_i^0$, $i=r+1, \ldots, 2n-1$ are bounded by
\begin{equation}\label{ineq-ggg0}
\max_{i=r+1, \ldots, 2n-1} | \gamma_i - \gamma_i^0 | \lesssim \frac{ b_n^{15} (\log n)^{3/2} }{ n^{3/2}c_n^{12}},
\end{equation}
where
\begin{eqnarray}
\nonumber
\gamma_i & = &  (W\bs{\bar{\mathbf{g}}})_i + (V^{-1}\bs{h})_i, \\
\nonumber
\gamma_i^0 & = &  (\widetilde{V}^{-1}\bs{\widetilde{h}})_i + [\widetilde{W} (\bs{\widetilde{\mathbf{g}}}- \E \bs{\widetilde{\mathbf{g}}})]_i.
\end{eqnarray}
\end{lemma}

Now, we are ready to prove Theorem \ref{theorem-LRT-beta} (b).

\begin{proof}[Proof of Theorem \ref{theorem-LRT-beta} (b)]
Note that $\bs{\widehat{\theta}}^0$ denotes the restricted MLE under the null space $\Theta_0 = \{ \bs{\theta}\in \R^{2n-1}: \alpha_1= \cdots = \alpha_r \}$.
The following calculations are based on the event $E_n$ that $\bs{\widehat{\theta}}$ and $\bs{\widehat{\theta}}^0$ simultaneously exist and satisfy
\begin{equation}\label{ineq-beta-beta0-upp}
\max\left\{ \| \widehat{\bs{\theta}} - \bs{\theta} \|_\infty, \| \widehat{\bs{\theta}}^0 - \bs{\theta} \|_\infty \right\}
\le \frac{b_{n}^{3}}{c_{n}^2}\sqrt{\frac{\log n}{n}}.
\end{equation}.
By Lemma \ref{lemma-hat-beta-diff-2b}, $\P(E_n) \ge 1 - O(n^{-1})$ if $b_n^6/c_n^5 = o( (n/\log n)^{1/2})$.

Similar to the proof of Theorem \ref{theorem-LRT-beta} (a),
it is sufficient to demonstrate: (1) $2( B_1 - B_1^0)$ converges in distribution to the chi-square distribution with $r-1$ degrees of freedom;
(2)
\[
B_2-B_2^0 = O_p\left(\frac{ b_n^{21} (\log n)^{5/2}}{ n^{1/2}c_{n}^{17}}\right), ~~ B_3 - B_3^0 = O_p\left( \frac{ b_n^{18} (\log n)^{5/2} }{ n^{1/2} c_n^{14} } \right).
\]
These two claims are proved in the following three steps, respectively.

Step 1. We show $2( B_1 - B_1^0)\stackrel{L}{\to} \chi^2_{r-1}$.
Using the matrix form in \eqref{VW-divide}, $ B_1 - B_1^0$ can be written as
\begin{eqnarray}
\nonumber
& &2(B_1 - B_1^0) \\
\nonumber
&  = & \underbrace{\sum_{i=1}^r \frac{ \bar{d}_i^{\,2} }{v_{i,i}}  - \frac{ ( \sum_{i=1}^r \bar{d}_i )^2 }{ \tilde{v}_{11} }}_{Z_1}
 + \mathbf{\bar{\mathbf{g}}}_1^\top W_{11} \mathbf{\bar{\mathbf{g}}}_1 + 2\mathbf{\bar{\mathbf{g}}}_1^\top W_{12} \mathbf{\bar{\mathbf{g}}}_2 + (\sum_{i=1}^r \bar{d}_i )^2 \tilde{w}_{11} \\
 \label{eq-theorem2b-B1022}
&&  + 2 ( \sum_{i=1}^r \bar{d}_i) \bs{\tilde{w}}_{12}^\top  \mathbf{\bar{\mathbf{g}}}_2
+  \mathbf{\bar{\mathbf{g}}}_2^\top ( W_{22} - \widetilde{W}_{22}) \mathbf{\bar{\mathbf{g}}}_2
- \underbrace{\mathbf{h}^\top V^{-1} \mathbf{h} + \mathbf{\widetilde{h}}^\top \widetilde{V}^{-1} \mathbf{\widetilde{h}}}_{Z_2},
\end{eqnarray}
where $\mathbf{\widetilde{h}}$ is in \eqref{eq-homo-tildeh}.
In view of Lemma \ref{lemma-W-widetilde-d-2b}, it is sufficient to demonstrate: (i) $Z_1$ converges in distribution to a chi-square distribution with $r-1$ degrees of freedom;
(ii) $Z_2=o_p(1)$.
Because $\bar{d}_i = \sum_{j=r+1}^n a_{i,j}$ is independent over $i=1,\ldots, r$ and $r$ is a fixed constant,
the classical central limit theorem for the bounded case \cite[e.g.][p. 289]{Loeve:1977} gives that
the vector $(\bar{d}_1/v_{11}^{1/2}, \ldots, \bar{d}_r/v_{rr}^{1/2})$
follows a $r$-dimensional standard normal distribution.
Because $v_{11}=\cdots=v_{rr}$ under the null $H_0:\alpha_1=\cdots=\alpha_r$ and $\tilde{v}_{11}=rv_{11}$, we have
\[
\sum_{i=1}^r \frac{ \bar{d}_i^{\,2} }{v_{ii}}  - \frac{ ( \sum_{i=1}^r \bar{d}_i )^2 }{ \tilde{v}_{11} }
= \left( \frac{\bar{d}_1}{v_{11}^{1/2}}, \ldots, \frac{\bar{d}_r}{v_{11}^{1/2}}\right)\left( I_r - \frac{1}{r} \mathbf{1}_r \mathbf{1}_r^\top \right)
\left( \frac{\bar{d}_1}{v_{11}^{1/2}}, \ldots, \frac{\bar{d}_r}{v_{11}^{1/2}}\right)^\top.
\]
Because $\mathrm{rank}( I_r - \mathbf{1}_r \mathbf{1}_r^\top/r )=r-1$, it follows that we have claim (i).
Now, we show $Z_2=o_p(1)$.
By setting $V^{-1}=S+W$ and $\widetilde{V}^{-1}=\widetilde{S}+\widetilde{W}$, we have
\begin{eqnarray*}
\mathbf{h}^\top V^{-1} \mathbf{h} &  = & \sum_{i=1}^{2n-1} \frac{ h_i^2 }{ v_{i,i} } +\frac{h_{2n}^{2}}{v_{2n,2n}}+ \mathbf{h_1}^\top W_{11} \mathbf{h_1}
+ 2 \mathbf{h_1}^\top W_{12} \mathbf{h_2} + \mathbf{h_2}^\top W_{22}\mathbf{h_2}, \\
\mathbf{\widetilde{h}}^\top \widetilde{V}^{-1} \mathbf{\widetilde{h}} & = & \frac{ \tilde{h}_1^2 }{ \tilde{v}_{11}}
+ \sum_{i=r+1}^{2n-1} \frac{ \tilde{h}_i^2 }{ v_{i,i} }+\frac{\tilde{h}_{2n}^{2}}{v_{2n,2n}}
+ \tilde{w}_{11} \tilde{h}_1^2 + 2 \tilde{h}_1 \bs{\tilde{w}}_{12} \bs{\tilde{h}}_2 + \bs{\widetilde{h}}_2^\top \widetilde{W}_{22} \bs{\widetilde{h}}_2,
\end{eqnarray*}
where $\mathbf{h}_1=(h_1, \ldots, h_r)^\top$, $\mathbf{h}_2=(h_{r+1}, \ldots, h_{2n-1})^\top$, and $h_i$ and  $\tilde{h}_i$ are given in \eqref{eq:definition:h} and \eqref{eq-homo-tildeh},
respectively.
Since $\|\bs h\|_{\infty}\lesssim b_{n}^{6}\log n/c_{n}^{5}$, we have
\[
\sum_{i=1}^{2n-1} \frac{ h_i^2 }{ v_{i,i} } - \sum_{i=r+1}^{2n-1} \frac{ \tilde{h}_i^2 }{ v_{i,i} } = \sum_{i=r+1}^{2n-1} \frac{ h_i^2 - \tilde{h}_i^2  }{ v_{i,i} } + O( \frac{ b_n^{13}(\log n)^2 }{nc_n^{11}}).
\]
The difference $h_i-\tilde{h}_i$ is bounded as follows:
\begin{eqnarray}
\nonumber
\max_{i=r+1,\ldots, {2n-1}} |h_i-\tilde{h}_i| & \le &   \sum_{j\neq i} \left\{ |\mu^{\prime\prime}(\tilde{\pi}_{ij})( \widehat{\pi}_{ij} - \pi_{ij})^2 - \mu^{\prime\prime}(\tilde{\pi}_{ij}^0)
( \widehat{\pi}_{ij} - \pi_{ij})^2| \right. \\
\nonumber
&&
\left.+ |\mu^{\prime\prime}(\tilde{\pi}_{ij}^0) ( \widehat{\pi}_{ij} - \pi_{ij})^2 - \mu^{\prime\prime}(\tilde{\pi}_{ij}^0)( \widehat{\pi}_{ij}^{0} - \pi_{ij})^2]| \right\} \\
\nonumber
& \lesssim & \frac{n}{c_n} \left\{ | \tilde{\pi}_{ij}-\tilde{\pi}_{ij}^0| \cdot ( \widehat{\pi}_{ij} - \pi_{ij})^2 +
 | \widehat{\pi}_{ij}-\widehat{\pi}_{ij}^0| \cdot ( |\widehat{\pi}_{ij} - \pi_{ij}|  +  |\widehat{\pi}_{ij}^0- \pi_{ij}|) \right\} \\
\nonumber
&\lesssim & \frac{n}{c_n} \cdot \left(\frac{b_n^{3}}{c_{n}^{2}}\sqrt{\frac{\log n}{n}}\right)^3 + \frac{n}{c_n} \cdot \frac{ b_n^9 \log n}{nc_n^{7}} \cdot \frac{b_n^{3}}{c_{n}^{2}}\sqrt{\frac{\log n}{n}}\\
\label{ineq-h-h-r1}
& \lesssim & \frac{ b_n^{27} (\log n)^2 }{n^{1/2} c_n^{9}},
\end{eqnarray}
where the second inequality is due to \eqref{ineq-mu-deriv-bound} and the mean value theorem.
Therefore, if $ b_n^{13}/c_n^{11} = o( n/(\log n)^2 )$, then
\begin{equation}\label{eq-theorem2-hhd}
| \sum_{i=1}^{2n-1} \frac{ h_i^2 }{ v_{i,i} } - \sum_{i=r+1}^{2n-1}\frac{ \tilde{h}_i^2 }{ v_{i,i} } | =O( \frac{ b_n^{13}(\log n)^2 }{nc_n^{11}}) + O(\frac{ b_n^{28} (\log n)^2 }{n^{3/2} c_n^{9}})  = o(1).
\end{equation}
By \eqref{ineq-V-S-appro-upper-b}, \eqref{approxi-inv2-beta-ho} and \eqref{eq-homo-tildeh}, we have
\begin{eqnarray*}
|\bs h_1^\top W_{11}\bs h_1 | & \le & r^2 \| W_{11} \|_{\max} \|\bs h_1 \|_\infty^{2} \lesssim r^2 \cdot (b_n^6 \log n)^2 \cdot \frac{ b_n^3 }{ n^2c_n^2 }\lesssim  \frac{ b_n^{15} (\log n)^2}{ n^{2}c_n^2},
\\
|\bs h_1^\top W_{12}\bs h_2 | & \lesssim & r(2n-r) \cdot (b_n^6\log n)^2 \cdot \frac{ b_n^3 }{ n^2c_n^2 }
\lesssim \frac{ b_n^{15} (\log n)^2}{ nc_n^2 },  \\
\frac{ \tilde{h}_1^2 }{ \tilde{v}_{11}} & \lesssim & \frac{ r^2 ( b_n^6\log n/c_n^{5})^2 }{ rn/b_n} \lesssim \frac{ b_n^{13} (\log n)^2 }{ nc_n^{10}}, \\
|\tilde{w}_{11} \tilde{h}_1^2| & \lesssim &  ( b_n^6\log n/c_n^{5})^2 \cdot \frac{ b_n^3 }{ n^2 c_n^2 } \lesssim \frac{ b_n^{15} (\log n)^2 }{ n^2 c_n^{12}}, \\
 |\tilde{h}_1 \bs{\tilde{w}}_{12} \bs{\tilde{h}}_2| & \lesssim & 2n-r \cdot ( b_n^6\log n/c_n^{5})^2 \cdot \frac{ b_n^3 }{ n^2 c_n^2 }\lesssim \frac{ b_n^{15} (\log n)^2 }{ n c_n^{12}}.
\end{eqnarray*}
To evaluate the bound of $\bs h_2^\top W_{22} \bs h_2 - \widetilde{\bs h}_2^\top \widetilde{W}_{22} \widetilde{\bs h}_2$,
we divide it into three terms:
\begin{eqnarray}
\nonumber
& & \bs h_2^\top W_{22} \bs h_2 - \widetilde{\bs h}_2^\top \widetilde{W}_{22} \widetilde{\bs h}_2 \\
\label{eq-theorem2-whh}
& = & \underbrace{\bs h_2^\top W_{22} \bs h_2 - \bs h_2^\top \widetilde{W}_{22} \bs h_2}_{C_1} +
\underbrace{\bs h_2^\top \widetilde{W}_{22} \bs h_2
- \widetilde{ \bs h}_2^{\top} \widetilde{W}_{22} \bs h_2}_{C_2}
+
\underbrace{\widetilde{\bs h}_2^\top \widetilde W_{22} \bs h_2 - \widetilde{\bs h}_2^{\top} \widetilde{W}_{22} \widetilde{\bs h}_2}_{C_3}.
\end{eqnarray}
The first term $C_1$ is bounded as follows.
By \eqref{ineq-WW-diff-2b}, we have
\begin{eqnarray}
\nonumber
| \bs h_2^\top ( W_{22} - \widetilde{W}_{22} ) \bs h_2 | & \le & (2n-r)^2  \| \bs h_{2} \|_\infty^{2} \frac{ b_n^6 }{ n^3c_n^5 } \\
\nonumber
&+&\frac{1}{v_{2n,2n}}|H_{1}(-1)^{1{(i>n-r)}}(\sum_{i=r+1}^{n}h_{i}-\sum_{i=n+1}^{2n-1}h_{i})\mathbf{1}_{2n-r-1}\bs h_{2}|\\
\label{ineq-upper-B1}
& \lesssim & \frac{ b_n^{18}(\log n)^2}{ n c_n^{15}}+(2n-r)\cdot  \frac{b_{n}}{n}\cdot \frac{b_n^3}{n^2c_n^3}\cdot\left(\frac{ b_n^6 \log n}{c_n^{5}} \right)^2\\
\nonumber
& \lesssim & \frac{ b_n^{18}(\log n)^2}{ n c_n^{15}}.
\end{eqnarray}
In view of \eqref{approxi-inv2-beta-ho} and \eqref{ineq-h-h-r1}, the upper bounds of $C_2$ and $C_3$ are derived as follows:
\begin{eqnarray}
\nonumber
|C_2|& = 
&   |( \bs{h}_2- \widetilde{\bs{h} }_2) ^\top \widetilde{W}_{22} \bs{h}_2 |
 \le  (2n-r)^2 \| \widetilde{W}_{22} \|_{\max} \| \bs{h}_2- \widetilde{\bs{h}}_2 \|_\infty \| \bs{h}_2 \|_\infty \\
\label{ineq-upper-B2}
& \lesssim & n^2 \cdot \frac{b_n^3}{n^2c_n^2} \cdot \frac{ b_n^{27} (\log n)^{2} }{ n^{1/2}c_n^{9} } \cdot b_n^6 \log n
 \lesssim \frac{ b_n^{36} (\log n)^{3} }{ n^{1/2}c_n^{11} }.
\end{eqnarray}
and
\begin{eqnarray}
\nonumber
|C_3|& = & |\widetilde{\bs{h}}_2^\top \widetilde{W}_{22} (\bs{h}_2 -  \widetilde{\bs{h}}_2) |
 \le  (2n-r)^2 \|\widetilde{\bs{h}}_2\|_\infty  \| \widetilde{W}_{22} \|_{\max} \| \bs{h}_2 -  \widetilde{\bs{h}}_2 \|_\infty \\
\label{ineq-upper-B3}
& \lesssim & \frac{ b_n^{36} (\log n)^{3} }{ n^{1/2}c_n^{11} }.
\end{eqnarray}
By combining \eqref{eq-theorem2-hhd}--\eqref{ineq-upper-B3}, it yields
\begin{equation*}
|\mathbf{h}^\top V^{-1} \mathbf{h} -
\mathbf{\widetilde{h}}^\top \widetilde{V}^{-1} \mathbf{\widetilde{h}}|
\lesssim \frac{ b_n^{36} (\log n)^{3} }{ n^{1/2}c_n^{11} }.
\end{equation*}
This completes the proof of the first step.

Step 2. To prove that this term does go to zero, we did a careful analysis on the difference
$B_2-B_2^0$ by using asymptotic representations of $\widehat{\theta}_i - \theta_i$ and $\widehat{\theta}_i^0 - \theta_i$.
With the use of Lemma \ref{lemma-hat-beta-diff-2b}, we have
\begin{equation}\label{ineq-B2-B20}
  B_2-B_2^0 = O_p\left(  \frac{ b_n^{21} (\log n)^{5/2}}{ n^{1/2}c_{n}^{17}}\right),
\end{equation}
whose detailed proofs are given in Section \ref{subsection:B2B20} in Supplementary Material.

Step 3. With the use of asymptotic representations of $\widehat{\theta}_i - \theta_i$ and $\widehat{\theta}_i^0 -\theta_i$ and Lemma \ref{lemma-hat-beta-diff-2b}, we can show
\begin{equation}\label{ineq-B3-B20}
  B_3-B_3^0 = O_p\left(  \frac{ b_n^{18} (\log n)^{5/2}}{ n^{1/2}c_{n}^{14}}\right),
\end{equation}
whose detailed proofs are given in Section \ref{subsection-B3B30} in Supplementary Material.

\end{proof}

\subsection{Proofs for Theorem  \ref{theorem-ratio-p-fixed}}
\label{section:theorem1-a}
To prove Theorem \ref{theorem-ratio-p-fixed}, we need four lemmas below.

Define one matrices $S_{22}=(s_{i,j})_{(2n-1-r)\times (2n-1-r)}$ with
\begin{equation}\label{definition-S221}
  s_{i,j}=
   \begin{cases}
  \displaystyle \frac{\delta_{i,j}}{v_{i,i}}+\frac{1}{\widetilde{v}_{2n,2n}}, &i,j=r+1,\ldots,n,\\
    \displaystyle -\frac{1}{\widetilde{v}_{2n,2n}}, &i=r+1,\ldots,n,j=n+1,\ldots,2n-1,\\
     \displaystyle -\frac{1}{\widetilde{v}_{2n,2n}}, &i=n+1,\ldots,2n-1,j=r+1,\ldots,n,\\
      \displaystyle\frac{\delta_{i,j}}{v_{i,i}}+\frac{1}{\widetilde{v}_{2n,2n}}, &i,j=n+1,\ldots,2n-1,\\
   \end{cases}
  \end{equation}
where $\widetilde{v}_{2n,2n}=\sum_{i=r+1}^{2n-1}\widetilde{v}_{i,2n}=\sum_{i=r+1}^{n}{v}_{i,2n}+\sum_{i=n+1}^{2n-1}\sum_{j=1}^{r}{v}_{i,j}$ and $\delta_{i,j}$ is the Kronecker delta function.

\begin{lemma}\label{lemma-appro-beta-VS}
For $V\in \mathcal{L}_n(1/b_n, 1/c_n)$ with $n\ge 2$ and its bottom right $(2n-1-r)\times (2n-1-r)$ block $V_{22}$ with $r\in\{1,\ldots, n-1\}$, if
\[
\frac{(n-r)c_{n} }{ (n-1)rb_{n} } = o(1),
\]
 we have
\begin{equation}\label{ineq-V-S-appro-upper-b}
\max\{ \|V^{-1} - S \|_{\max}, \| V_{22}^{-1} - S_{22} \|_{\max} \}  \le \frac{c_0 b_{n}^{3}}{c_{n}^{2}(n-1)^{2} },
\end{equation}
where $c_0$ is a fixed constant.
\end{lemma}

Lemma \ref{lemma-appro-beta-VS} is an extension of that of Proposition 1 in \cite{Yan:Leng:Zhu:2016} and therefore the proof of Lemma \ref{lemma-appro-beta-VS}  is omitted.

\begin{lemma} \label{lemma-clt-beta-W}
Recall that $V_{22}$ is the bottom right $(2n-1-r)\times (2n-1-r)$ block of $V$.
Let $W=V^{-1}-S$,
$\widetilde{W}_{22}=V_{22}^{-1}-S_{22}$ and $\bs{\bar{\mathbf{g}}}_2=(\bar{d}_{r+1}, \ldots, \bar{d}_n, \bar{b}_{1}, \ldots, \bar{b}_{n-1})^\top$. For any given $r\in \{0, \ldots, n-1\}$, we have
\[
\bs{\bar{\mathbf{g}}}_2^\top \widetilde{W}_{22} \bs{\bar{\mathbf{g}}}_2 = O_p\left( \frac{b_n^3}{c_n^3}(2-\frac{r}{n})^{3/2} \right),
\]
where $r=0$ implies $\bs{\bar{\mathbf{g}}}_2=\bs{\bar{\mathbf{g}}}$, $V_{22}=V$ and $\widetilde{W}_{22}=W$.
\end{lemma}

Lemma \ref{lemma-clt-beta-W}
states that the remainder terms $\bs{\bar{\mathbf{g}}}_2^\top \widetilde{W}_{22} \bs{\bar{\mathbf{g}}}_2$ and $\bs{\bar{\mathbf{g}}}^\top W \bs{\bar{\mathbf{g}}}$ in \eqref{eq-theorem2-B10}
 is in the order of $O_p\left( (b_n^3/c_n^3)(2-r/n)^{3/2} \right)$ for any given $r\ge 0$.

\begin{lemma}\label{lemma-consi-beta}
Under the null $H_0: (\alpha_1, \ldots, \alpha_r)=(\alpha_1^0, \ldots, \alpha_r^0)$ for any given $r\in\{0, \ldots, n-1\}$,
if
\begin{equation}\label{con-lemma3-beta-c}
\frac{2b_{n}^{6}(2n-1-r)^2}{(n-1)^{2}c_{n}^{5}}
=o\left(\sqrt{\frac{n}{\log n}}\right),
\end{equation}
 then with probability at least $1-4/n$, the restricted MLE $\widehat{\bs{\theta}}^0$ exists and satisfies
\begin{equation*}\label{ineq-En-beta}
\| \widehat{\bs{\theta}}^0 - \bs{\theta} \|_\infty  \le  \frac{b_{n}^{3}}{c_{n}^2}\sqrt{\frac{\log n}{n}},
\end{equation*}
where $r=0$ means there is no any restriction on $\bs{\theta}$ and implies $\widehat{\bs{\theta}}^0=\widehat{\bs{\theta}}$. 
Further, if the restricted MLE exists, it must be unique.
\end{lemma}


From Lemma \ref{lemma-consi-beta}, we can see that the consistency rate for the restricted MLE $\bs{\widehat{\theta}}^0$ in terms of the $L_\infty$-norm
is independent of $r$ while the condition depends on $r$. The larger $r$ is, the weaker the condition is. When $r=0$, the lemma gives the error bound for the MLE $\bs{\widehat{\theta}}$.

\begin{lemma}
\label{lemma:beta3:err}
If \eqref{con-lemma3-beta-c} holds,  then for an arbitrarily given $r\in\{0, \ldots, n-1\}$, 
\begin{eqnarray}
\sum_{i=r+1}^{n} (\widehat{\alpha}_i-\alpha_i)^3 \sum_{j\neq i} \mu^{\prime\prime}( \pi_{ij} )
\label{eq-S1-bound}
= O_p\left( \frac{ (1-r/n)b_n^{11} (\log n)^2 }{  c_n^8 } \right)+\frac{n(n-r)\bar b_{n}^3}{c_{n}v_{2n,2n}^3},\\
\sum_{j=1}^{n-1}  (\widehat{\beta}_j-\beta_j)^3 \sum_{j\neq i} \mu^{\prime\prime}( \pi_{ij} )
\label{eq-S2-bound}
= O_p\left(  \frac{ (1-r/n)b_n^{11} (\log n)^2 }{  c_n^8 }  \right)-\frac{n(n-r)\bar b_{n}^3}{c_{n}v_{2n,2n}^3},\\
\sum_{i=r+1 }^n \sum_{j=1,j\neq i}^{n-1} (\widehat{\alpha}_i-\alpha_i)^2(\widehat{\beta}_j-\beta_j)\mu^{\prime\prime}( \pi_{ij} )  \label{eq-S3-bound}
= O_p
\left(\frac{ (1-r/n)b_n^{11} (\log n)^2 }{  c_n^8 } \right)-\frac{n(n-r)\bar b_{n}^3}{c_{n}v_{2n,2n}^3},\\
\sum_{i=r+1 }^n \sum_{j=1,j\neq i}^{n-1} (\widehat{\alpha}_i-\alpha_i)(\widehat{\beta}_j-\beta_j)^2\mu^{\prime\prime}( \pi_{ij} ) \label{eq-S4-bound}
 = O_p
\left( \frac{ (1-r/n)b_n^{11} (\log n)^2 }{  c_n^8 } \right)+\frac{n(n-r)\bar b_{n}^3}{c_{n}v_{2n,2n}^3}.
\end{eqnarray}
If $\widehat{\alpha}_i$ is replaced with $\widehat{\alpha}_i^0$ for $i=r+1,\ldots,n$, then the above upper bound still holds.
\end{lemma}

\begin{proof}[Proof of Theorem \ref{theorem-ratio-p-fixed}]
Under the null $H_0: (\alpha_1, \ldots, \alpha_r)=(\alpha_1^0, \ldots, \alpha_r^0)$, the data generating parameter $\bs{\theta}$ is equal to $(\alpha_1^0, \ldots, \alpha_r^0, \alpha_{r+1}, \ldots, \alpha_n, \beta_{1}, \ldots, \beta_{n-1})^\top$.
For convenience, we suppress the superscript $0$ in $\alpha_i^0, i=1,\ldots,r$ when causing no confusion.
The following calculations are based on the event $E_n$ that $\bs{\widehat{\theta}}$ and $\bs{\widehat{\theta}}^0$ simultaneously exist and satisfy
\begin{equation}
\max\left\{ \| \widehat{\bs{\theta}} - \bs{\theta} \|_\infty, \| \widehat{\bs{\theta}}^0 - \bs{\theta} \|_\infty \right\}
\le \frac{b_{n}^{3}}{c_{n}^2}\sqrt{\frac{\log n}{n}}.
\end{equation}
By Lemma \ref{lemma-consi-beta}, $\P(E_n) \ge 1 - O(n^{-1})$ if $b_n^6/c_n^5=o\left\{ (n/\log n)^{1/2} \right\}$.

Similar to the proof of Theorem \ref{theorem-LRT-beta} (a),
it is sufficient to demonstrate:
 (1) $\{2( B_1 - B_1^0)-r\}/(2r)^{1/2}$ converges in distribution to the standard normal distribution as $r\to\infty$;
(2) $(B_2 - B_2^0)/r^{1/2}=o_p(1)$; (3) $(B_3-B_3^0)/{r^{1/2}}=o_p(1)$.
The second claim is a direct result of Lemma \ref{lemma:beta3:err}.
Note that $\widehat{\alpha}_i^0 = \alpha_i^{0}$, $i=1, \ldots, r$. So $B_3^0$ has less terms than $B_3$.
In view of \eqref{ineq-mu-deriv-bound} and \eqref{ineq-beta-beta0-upp}, if $b_n^{12}/c_n^{9}=o( r^{1/2}/(\log n)^2 )$, then
\begin{eqnarray}
\label{ineq-B3-upper}
\frac{ |B_3| }{ r^{1/2} } & \lesssim & \frac{1}{r^{1/2}} \cdot
\frac{n^2}{c_n} \cdot \| \bs{\widehat{\theta}} - \bs{\theta} \|_\infty^4 \lesssim \frac{ b_n^{12}(\log n)^2 }{ r^{1/2} c_n^{9} } = o(1), \\
\label{ineq-B30-upper}
\frac{ |B_3^0| }{ r^{1/2} } & \lesssim & \frac{1}{r^{1/2}} \cdot
\frac{n(n-r)}{c_n} \cdot \| \bs{\widehat{\theta}}^0 - \bs{\theta} \|_\infty^4 \lesssim \frac{ b_n^{12}(\log n)^2 }{ r^{1/2} c_n^{9} } = o(1),
\end{eqnarray}
which shows the third claim.
Therefore, the remainder of the proof is to verify claim (1).

Step 1. Similar to the proof of Theorem \ref{theorem-LRT-beta} (a), we have
\begin{equation*}
\bs{\mathbf{g}} - \E( \bs{\mathbf{g}} ) = V ( \widehat{\boldsymbol{\theta}} - \boldsymbol{\theta} ) + \bs{h}.
\end{equation*}
It yields that
\begin{equation}\label{eq-expansion-hatbeta-beta}
\boldsymbol{\widehat{\theta}} - \boldsymbol{\theta} = V^{-1} \bs{\bar{\mathbf{g}}} - V^{-1}\mathbf{h}.
\end{equation}
By Lemma \ref{lemma-appro-beta-VS}, we have
\[
|(W\textbf{h})_{i}|\leq \|W\|_{\max}\times \left[(2n-1)\max_{i}|h_{i}|\right]\leq O(\frac{b_{n}^{9}\log n}{c_{n}^{7}n}).
\]
Note that $(S\textbf{h})_{i}=\cfrac{h_{i}}{v_{i,i}}+(-1)^{1{(i>n)}}\cfrac{h_{2n}}{v_{2n,2n}}$, $h_{2n}:=\sum_{i=1}^{n}h_{i}-\sum_{i=n+1}^{2n-1}h_{i},$ we have
\[
|(S\textbf{h})_{i}|\leq \cfrac{|h_{i}|}{v_{i,i}}+\cfrac{|h_{2n}|}{v_{2n,2n}}\leq O(\frac{b_{n}^{7}\log n}{c_{n}^{5}n}).
 \]
Therefore, we have
\begin{equation}\label{eqA4}
|(V^{-1}\textbf{h})_{i}|\leq |(S\textbf{h})_{i}|+|(W\textbf{h})_{i}|\leq O(\frac{b_{n}^{9}\log n}{c_{n}^{7}n}).
    \end{equation}
Recall  $\bs{\bar{\mathbf{g}}}_2=(\bar{d}_{r+1}, \ldots, \bar{d}_n, \bar{b}_{1}, \ldots, \bar{b}_{n-1})^\top$. Let $\boldsymbol{\widehat{\theta}}_2^0= (\widehat{\alpha}_{r+1}^0, \ldots, \widehat{\alpha}_n^0, \widehat{\beta}_{1}^{0}, \ldots, \widehat{\beta}_{n-1}^{0})^\top$ and
$\boldsymbol{{\theta}}_2= ({\alpha}_{r+1}, \ldots, {\alpha}_n,$\\
$ {\beta}_{1}, \ldots,{\beta}_{n-1})^\top$.
Similar to \eqref{eq-expansion-hatbeta-beta}, we have
\begin{equation}\label{eq-beta0-exapnsion}
\boldsymbol{\widehat{\theta}}^0_2 - \boldsymbol{\theta}_2 = V_{22}^{-1}\bs{\bar{\mathbf{g}}}_2  - V_{22}^{-1}\widetilde{\mathbf{{h}}}_2,
\end{equation}
where $\mathbf{\widetilde{h}}_2 = (\tilde{h}_{r+1}, \ldots, \widetilde{h}_{2n-1})^\top$, $\mathbf{\widetilde{h}} = (\tilde{h}_{1}, \ldots, \tilde{h}_{2n-1})^\top$ and
\begin{equation}\label{defintion-tilde-h}
 \tilde h_{i,j}  = \frac{1}{2}\mu^{\prime\prime}( \tilde{\pi}_{ij}^{0} )(\widehat{\pi}_{ij}^{0} - \pi_{ij})^2, ~~
|\tilde h_{i}|\leq \frac{b_{n}^{6}\log n}{c_{n}^{5}}.
\end{equation}
In the above equation, $\tilde{\pi}_{ij}^0$ lies between $\pi_{ij}$ and $\widehat{\pi}_{ij}^0=\widehat{\alpha}_i^0 + \widehat{\beta}_j^0$.

Step 2. We derive the explicit expression of $B_1-B_1^0$.
Substituting \eqref{eq-expansion-hatbeta-beta} and \eqref{eq-beta0-exapnsion} into the expressions of $B_1$ in \eqref{lrt-a-beta-B1} and $B_1^0$ respectively, it yields
\begin{eqnarray*}\label{B1-expression}
2B_1 & = & \bs{\bar{\mathbf{g}}}^\top V^{-1} \bs{\bar{\mathbf{g}}} - \bs{h}^\top V^{-1} \bs{h},\\
\label{likelihood-beta-composite}
 2B_1^0 & = & \bs{\bar{\mathbf{g}}}_2^\top V_{22}^{-1} \bs{\bar{\mathbf{g}}}_2 -  \bs{\widetilde{h}}_2^\top V_{22}^{-1} \bs{\widetilde{h}}_2.
\end{eqnarray*}
By setting $V^{-1}=S+W$ and $V_{22}^{-1}=S_{22}+ \widetilde{W}_{22}$, we have
\begin{equation}\label{eq-theorem2-B10}
2(B_1 - B_1^0) = \sum_{i=1}^r \frac{ \bar{d}_i^{\,2} }{v_{i,i}}+\frac{ \bar{b}_n^{\,2} }{v_{2n,2n}}-\frac{(\bar{b}_n -\sum_{i=1}^{r}\bar{d}_i)^{2}}{\widetilde{v}_{2n,2n}}   + \bs{\bar{\mathbf{g}}}^\top W \bs{\bar{\mathbf{g}}} - \bs{\bar{\mathbf{g}}}_2^\top \widetilde{W}_{22} \bs{\bar{\mathbf{g}}}_2
-\bs{{h}}^\top V^{-1} \bs{{h}} + \bs{\widetilde{h}}_2^\top V_{22}^{-1} \bs{\widetilde{h}}_2.
\end{equation}

Step 3.  
We show four claims: (i) $(\sum_{i=1}^r  \bar{d}_i^{\,2}/v_{i,i}-r)/(2r)^{1/2}\stackrel{L}{\to} N(0,1)$;
(ii) $\frac{ \bar{b}_n^{\,2} }{r^{1/2}v_{2n,2n}}=o_p(1)$ and $\frac{(\bar{b}_n -\sum_{i=1}^{r}\bar{d}_i)^{2}}{r^{1/2}\widetilde{v}_{2n,2n}}=o_p(1)$;
(iii) $\bs{\bar{\mathbf{g}}}^\top W \bs{\bar{\mathbf{g}}}/r^{1/2}=o_p(1)$ and $\bs{\bar{\mathbf{g}}}_2^\top \widetilde{W}_{22} \bs{\bar{\mathbf{g}}}_2/r^{1/2}=o_p(1)$;
(iv) $\bs{{h}}^\top V^{-1} \bs{{h}}/r^{1/2}=o_p(1)$ and $\bs{\widetilde{h}}_2^\top V_{22}^{-1} \bs{\widetilde{h}}_2/r^{1/2}=o_p(1)$.
The first and three claims directly follows from Lemma \ref{lemma:weighte-degree-al} and Lemma \ref{lemma-clt-beta-W}, respectively.
By (\ref{ineq-union-d}), if $b_{n}=o(r^{1/2}/\log n)$, we have
\[
|\frac{ \bar{b}_n^{\,2} }{r^{1/2}v_{2n,2n}}|\leq\frac{1}{r^{1/2}}\cdot n\log n\cdot \frac{b_{n}}{n-1}=o(1).
\]
Let $v_{0,0}=\sum_{i=1}^{r}v_{i,i}$. Since $\sum_{i=1}^{r}\bar{d}_i=\sum_{i=1}^{r}\sum_{j=1}^{n-1}\bar{a}_{i,j}$, by the central limit theorem for bounded case [\cite{Loeve:1977} (p.289)], ${v}_{0,0}^{-1/2}{\sum_{i=1}^{r}\bar{d}_i}$ converges in distribution to the standard normal distribution if ${v}_{0,0}\rightarrow\infty$. Therefore, as $r\rightarrow\infty$,
\[
\frac{({\sum_{i=1}^{r}\bar{d}_i})^{2}}{v_{0,0}}\cdot\frac{v_{0,0}}{\widetilde{v}_{2n,2n}}\cdot\frac{1}{r^{1/2}}=o_{p}(1).
\]
This demonstrates claim (ii).
By \eqref{eqA4}, we have
\[
|\mathbf{h}^\top V^{-1} \mathbf{h}| \le (2n-1) \| \mathbf{h} \|_\infty \| V^{-1} \mathbf{h} \|_\infty
\lesssim
 (2n-1) \cdot \frac{b_n^6\log n}{c_n^{5}} \cdot \frac{ b_n^9 \log n}{nc_n^{7}} \lesssim \frac{b_n^{15} (\log n)^2}{c_n^{12}}.
\]
If $b_n^{15}/c_n^{12} =o( r^{1/2}/(\log n)^{2})$, then
\begin{equation}\label{eq-simi-aVh}
\frac{1}{r^{1/2}} |\mathbf{h}^\top V^{-1} \mathbf{h}| \lesssim \frac{ b_n^{15}(\log n)^2 }{r^{1/2}c_n^{12}}  = o(1).
\end{equation}
In view of \eqref{defintion-tilde-h}, with the same arguments as in the proof of the above inequality, we have
\begin{equation}\label{eq-simi-a}
\frac{1}{r^{1/2}} |\bs{\widetilde{h}}_2^\top V_{22}^{-1} \bs{\widetilde{h}}_2| \lesssim \frac{ (2n-1-r)b_n^{15}(\log n)^2 }{n r^{1/2}c_n^{12}}  = o(1).
\end{equation}
This demonstrates claim (iv).
It completes the proof.
\end{proof}

\section*{Conflict of Interest}

The authors declare no conflict of interest.

\section*{Acknowledgements}
We are grateful to two anonymous referees for their
insightful comments and suggestions.

\newpage
\begin{center}
{\Large Supplementary material for ``Testing degree heterogeneity in directed networks"} \\

\medskip
\end{center}

\input{supp.tex}

\end{document}

%% file: supp.tex
Supplementary Material contains the proofs of supported lemmas in the proofs of Theorems 1 and 2 as well as inequalities \eqref{eq-thereom1b-z43}, \eqref{ineq-B2-B20} and \eqref{ineq-B3-B20} in the main text, two additional algorithms for solving the restricted MLEs in Section 2,
and 7 additional figures in the simulation section.
This supplementary material is organized as follows.

Section \ref{section-th1b} presents proofs of supported lemmas in the proof of Theorem 1 (a).
This section is organized as follows.
Sections \ref{section-lemma7}, \ref{section-lemma8} and \ref{section-lemma9}
present the proofs of Lemmas \ref{lemma-tilde-W}, \ref{lemma-con-beta-b} and \ref{lemma-beta-homo-expan}, respectively.
Section \ref{section-proof-39-B20} presents the proof of \eqref{eq-thereom1b-z43} in the main text.

Section \ref{section-theorem2a} presents proofs of supported lemmas in the proof of Theorem 1 (b)
as well as  proof of \eqref{ineq-B2-B20} and \eqref{ineq-B3-B20} in the main text.
This section is organized as follows.
Sections \ref{section-proof-lemma1010}, \ref{le12} and \ref{section-proof-lemma13}
present the proofs of Lemmas \ref{w2-error-2b}, \ref{lemma-hat-beta-diff-2b} and \ref{lemma-gg0}, respectively.
Sections \ref{subsection:B2B20} and \ref{subsection-B3B30} presents the proofs of orders of two remainder terms
$B_2-B_2^0$ in \eqref{ineq-B2-B20} and $B_3-B_3^0$ in \eqref{ineq-B3-B20} in the main text, respectively.

Section \ref{section:beta-th1a} contains the proofs of supported lemmas in the proof of Theorem 2.
This section is organized as follows.
Sections \ref{subsection-proof-lemma2}, \ref{subsection-proof-lemma4} and \ref{subsection-proof-lemma5} present
the proofs of Lemmas \ref{lemma-clt-beta-W}, \ref{lemma-consi-beta} and \ref{lemma:beta3:err}, respectively.
Section \ref{section-bernstein} reproduces Bernstein's inequality for easy readability.
Section \ref{section:algorithms} presents the fixed-point iterative algorithms for solving the restricted MLE
under the specified null and the global MLE.
Section \ref{section:figure} presents 7 additional figures in the simulation section.

All notation is as defined in the main text unless explicitly noted otherwise. Equation and lemma numbering continue
in sequence with those established in the main text.

We first recall useful inequalities on the derivatives of $\mu(x)$, which will be used in the proofs repeatedly.
Recall that
\[
\mu(x) = \frac{ e^x }{ 1 + e^x}.
\]
A direct calculation gives that the derivative of $\mu(x)$ up to the third order are
\begin{eqnarray}\label{eq-derivative-mu-various}
\mu^\prime(x) = \frac{e^x}{ (1+e^x)^2 },~~  \mu^{\prime\prime}(x) = \frac{e^x(1-e^x)}{ (1+e^x)^3 },~~ \mu^{\prime\prime\prime}(x) =  \frac{ e^x [ (1-e^x)^2 - 2e^x] }{ (1 + e^x)^4 }.
\end{eqnarray}
Note that $\bs{\theta}=(\alpha_1, \ldots, \alpha_n, \beta_1, \ldots, \beta_{n-1})^{\top}$ denotes the data generating parameter, under which the data are generated.
Recall that
\[
\pi_{ij} = \alpha_i + \beta_j, ~~\widehat{\pi}_{ij}=\widehat{\alpha}_i + \widehat{\beta}_j,~~
\widehat{\pi}_{ij}^0 = \widehat{\alpha}_i^0 + \widehat{\beta}_j^0.
\]
According to the definition of $c_n$, we have
\begin{equation}\label{ineq-mu-deriv-bound1}
|\mu^\prime(\pi_{ij})| \le \frac{1}{c_n}, ~~
|\mu^{\prime\prime}(\pi_{ij} )| \le \frac{1}{c_n},~~ |\mu^{\prime\prime\prime}(\pi_{ij})| \le \frac{1}{c_n}.
\end{equation}
For a $\bs{\widetilde{\theta} }$ satisfying $\| \bs{\widetilde{\theta} } - \bs{\theta} \|_\infty =o(1)$, we also have
\begin{equation}\label{ineq-mu-tilde}
|\mu^\prime(\tilde{\pi}_{ij} )| \lesssim \frac{1}{c_n}, ~~ |\mu^{\prime\prime}(\tilde{\pi}_{ij})| \lesssim \frac{1}{c_n},~~ |\mu^{\prime\prime\prime}(\tilde{\pi}_{ij})| \lesssim \frac{1}{c_n}.
\end{equation}
These facts will be used in the proofs repeatedly.
Recall that
\[
\bar{a}_{i,j}=a_{i,j}-\E(a_{i,j})
\]
is the centered random variable of $a_{i,j}$ and $\bar{a}_{i,i}=0$ for all $i=1, \ldots, n$.
Correspondingly, denote $\bar{d}_i = d_i - \E(d_i)=\sum_j \bar{a}_{i,j}$, $\bar{b}_j = b_j - \E(b_j)= \sum_i \bar{a}_{i,j}$ and
$\bs{\bar{\mathbf{g}}}=(\bar{d}_1, \ldots, \bar{d}_n, \bar{b}_1,\ldots, \bar{b}_{n-1})^\top$.

\section{Proofs of supported lemmas in the proof of Theorem 1 (a)}
\label{section-th1b}

This section is organized as follows.
Sections \ref{section-lemma7}, \ref{section-lemma8} and \ref{section-lemma9}
present the proofs of Lemmas \ref{lemma-tilde-W}, \ref{lemma-con-beta-b} and \ref{lemma-beta-homo-expan}, respectively.
Section \ref{section-proof-39-B20} presents the proof of \eqref{eq-thereom1b-z43} in the main text.
\subsection{Proof of Lemma \ref{lemma-tilde-W}}
\label{section-lemma7}
\begin{proof}
Recall that  $\bs{\widetilde{\mathbf{g}}}=(\sum_{i=1}^r d_i, d_{r+1},\ldots, d_n,b_{1},\ldots, b_{n-1})^\top$ and
$\widetilde{W} = \widetilde{V} - \widetilde{S}$.
It is sufficient to demonstrate:
\begin{equation}\label{wd-expectation2}
\E[  (\bs{\widetilde{\mathbf{g}}} - \E\bs{\widetilde{\mathbf{g}}} )^\top \widetilde{W}  (\bs{\widetilde{\mathbf{g}}} - \E\bs{\widetilde{\mathbf{g}}} ) ] =0,
\end{equation}
and
\begin{equation}\label{Wd-op2}
\mathrm{Var}( (\bs{\widetilde{\mathbf{g}}} - \E\bs{\widetilde{\mathbf{g}}} )^\top \widetilde{W}  (\bs{\widetilde{\mathbf{g}}} - \E\bs{\widetilde{\mathbf{g}}} ) )=
O\left( \frac{ b_n^3 }{ c_n^3 } \right).
\end{equation}
The claim of \eqref{wd-expectation2} is due to that
\begin{eqnarray*}
\E[  (\bs{\widetilde{\mathbf{g}}} - \E\bs{\widetilde{\mathbf{g}}} )^\top \widetilde{W} (\bs{\tilde{d}}\bs{\widetilde{\mathbf{g}}} - \E\bs{\widetilde{\mathbf{g}}} ) ] & = &
\mathrm{tr} (\E[  (\bs{\widetilde{\mathbf{g}}} - \E\bs{\widetilde{\mathbf{g}}} )^\top (\bs{\widetilde{\mathbf{g}}} - \E\bs{\widetilde{\mathbf{g}}} )  ] \widetilde{W} ) \\
& = & \mathrm{tr} ( \widetilde{V} \widetilde{W} ) = \mathrm{tr} ( I_{2n-r} - \widetilde{V} \widetilde{S} ) = 0.
\end{eqnarray*}
Let
\[
R = \begin{pmatrix} \overline{W}_{11} &  \overline{W}_{12} \\
\overline{W}_{21} & \widetilde{W}_{22}
\end{pmatrix},
\]
where $\widetilde{W}_{22}$ is the bottom right $(2n-1-r)\times (2n-1-r)$ block of $\widetilde{W}$,
$\overline{W}_{11}$ is the $r\times r$ matrix with all its elements being equal to $\tilde{w}_{11}$,
and $\overline{W}_{12}$ is the $r \times (2n-1-r)$ matrix with all its row being equal to
the vector $(\tilde{w}_{12}, \ldots, \tilde{w}_{1,2n-r})$, and $\overline{W}_{21}$ is the transpose of $\overline{W}_{12}$.
Therefore, we have
\[
(\bs{\widetilde{\mathbf{g}}} - \E\bs{\widetilde{\mathbf{g}}} )^\top \widetilde{W}  (\bs{\widetilde{\mathbf{g}}} - \E\bs{\widetilde{\mathbf{g}}} ) = \bs{\widetilde{\mathbf{g}}}^\top R \bs{\widetilde{\mathbf{g}}}.
\]
Because
\[
\| R \|_{\max}=\| \widetilde{W} \|_{\max}  \lesssim \frac{ b_n^3 }{ n^2c_n^2 },
\]
with the same arguments as in the proof of Lemma \ref{lemma-clt-beta-W}, we have \eqref{Wd-op2}.
This completes the proof.
\end{proof}

\subsection{Proof of Lemma \ref{lemma-con-beta-b}}
\label{section-lemma8}
Recall that $\bs{\widehat{\theta}}^0$ denotes the restricted MLE under the null $H_0: \alpha_1=\cdots =\alpha_r$.
In what follows, $\bs{\widehat{\theta}}^0$ and $\bs{\theta}$ denote
respective vectors $(\widehat{\alpha}_1^0, \widehat{\alpha}_{r+1}^0, \ldots,
\widehat{\alpha}_n^0,\widehat{\beta}_1^0, \ldots,
\widehat{\beta}_{n-1}^0)^\top$ and $(\alpha_1, \alpha_{r+1}, \ldots, \alpha_n,\beta_1,  \ldots, \beta_{n-1})^\top$ with some ambiguity of notations.
Define a system of score functions based on likelihood equations:
\begin{equation}
\label{eqn:def:F1}
\begin{split}
F_1(\bs{\theta}) & =  \sum\limits_{i=1}^r \sum\limits_{k=1, k\neq i}^n \mu(\alpha_i + \beta_k) - \sum\limits_{i=1}^r d_i, \\
F_i( \bs \theta )& =\sum_{k=1;k\neq i}^{n}\mu(\alpha_i + \beta_k)-d_{i},  ~~~  i=r+1, \ldots, n, \\
F_{n+j}( \bs \theta )& =  \sum_{k=1;k\neq j}^{n}\mu(\alpha_k + \beta_j)-b_{j},  ~~~  j=1, \ldots, n-1, \\
\end{split}
\end{equation}
and $F(\bs{\theta})=(F_1(\bs{\theta}),F_{r+1}(\bs{\theta}), \ldots, F_{2n-1}(\bs{\theta}))^\top$, where
$\alpha_1=\cdots=\alpha_r$.

The proof proceeds three steps. Step 1 is about verifying condition \eqref{eq-kantororich-a}.
Step 2 is about verifying \eqref{eq-kantororich-b}. Step 3 is a combining step.

Step 1. We claim that
\begin{equation}\label{eq-con-ho-a}
\| [F^\prime(\bs{\theta})]^{-1} \{ F^\prime(\mathbf{x}) - F^\prime(\mathbf{y}) \}\| \lesssim \left( \frac{(n-r+2)b_{n}}{c_{n}} + \frac{b_n^3}{c_n^3} \right) \|\mathbf{x} - \mathbf{y}\|.
\end{equation}
This is verified as follows.
The Jacobian matrix $F^\prime(\bs{\theta})$ of $F(\bs{\theta})$ can be calculated as follows.
For $i=1$, we have
\[
    \begin{split}
    \cfrac{\partial F_{1}}{\partial \alpha_{l}}=0,~~l=r+1,\ldots,n;~~
\cfrac{\partial F_{1}}{\partial \alpha_{1}}=\sum_{k\neq 1}\mu'(\alpha_{1}+\beta_{k})+\cdots+\sum_{k\neq r}\mu'(\alpha_{1}+\beta_{k}),\\
\cfrac{\partial F_{1}}{\partial\beta_{l}}=(r-1)\mu'(\alpha_{1}+\beta_{l}),~~l=1,\ldots,r;~~
\cfrac{\partial F_{1}}{\partial\beta_{l}}=r\mu'(\alpha_{1}+\beta_{l}),~~l=r+1,\ldots,n-1,
\end{split}
    \]
and for $i\in \{r+1, \ldots, n\}$, we have
\[
    \begin{split}
    \cfrac{\partial F_{i}}{\partial \alpha_{l}}&=0,~~l=1,\ldots,n,l\neq i;~~
\cfrac{\partial F_{i}}{\partial \alpha_{i}}=\sum_{k=1;k\neq i}^{n}\mu'(\alpha_{i}+\beta_{k}),\\
\cfrac{\partial F_{i}}{\partial\beta_{l}}&=\mu'(\alpha_{i}+\beta_{l}),~~l=1,\ldots,n-1,l\neq i;~~\cfrac{\partial F_{i}}{\partial\beta_{i}}=0,\\
\end{split}
    \]
and for $j\in\{1, \ldots, n-1\}$,\\
\[
    \begin{split}
\cfrac{\partial F_{n+j}}{\partial \alpha_{1}}&=(r-1)\mu'(\alpha_{1}+\beta_{j}),~~j=1,\ldots,r;
~~\cfrac{\partial F_{n+j}}{\partial \alpha_{1}}=r\mu'(\alpha_{1}+\beta_{j}),~~j=r+1,\ldots,n-1,\\
\cfrac{\partial F_{n+j}}{\partial \alpha_{l}}&=\mu'(\alpha_{l}+\beta_{j}),~~l=r+1,\ldots,n,l\neq j;
~~\cfrac{\partial F_{n+j}}{\partial \alpha_{j}}=0;\cfrac{\partial F_{n+j}}{\partial\beta_{l}}=0,~~l\neq j,\\
\cfrac{\partial F_{n+j}}{\partial\beta_{j}}&=\sum_{k=r+1;k\neq j}^{n}\mu'(\alpha_{k}+\beta_{j})+(r-1)\mu'(\alpha_{1}+\beta_{j}),~~j=1,\ldots,r,\\
\cfrac{\partial F_{n+j}}{\partial\beta_{j}}&=\sum_{k=r+1;k\neq j}^{n}\mu'(\alpha_{k}+\beta_{j})+r\mu'(\alpha_{1}+\beta_{j}),~~j=r+1,\ldots,n-1.
\end{split}
    \]
Let
\[
F_i'(\bs{\theta}) = (F_{i,1}'(\bs{\theta}),F_{i,r+1}'(\bs{\theta}), \ldots, F_{i,2n-1}'(\bs{\theta}))^{\top}:=
( \frac{\partial F_i}{\partial \alpha_{1} },\frac{\partial F_i}{\partial \alpha_{r+1} }, \ldots,
\frac{\partial F_i}{\partial \beta_{n-1} })^{\top}.
\]
For $i=1$, we have
\[
    \begin{split}
    \cfrac{\partial^2 F_{1}}{\partial \alpha_{l}\partial \alpha_{1}}&=0,~~l\neq 1;~~
\cfrac{\partial^2 F_{1}}{\partial \alpha_{1}^2}=\sum_{k\neq 1}\mu^{\prime\prime}(\alpha_{1}+\beta_{k})+\cdots+\sum_{k\neq r}\mu^{\prime\prime}(\alpha_{1}+\beta_{k}),\\
\cfrac{\partial^2 F_{1}}{\partial \alpha_{1}\partial\beta_{l}}&=(r-1)\mu^{\prime\prime}(\alpha_{1}+\beta_{l}),~~l=1,\ldots,r;~~
 \cfrac{\partial^2 F_{1}}{\partial \alpha_{s}\partial \alpha_{l}}=0,\\
\cfrac{\partial^2 F_{1}}{\partial \alpha_{1}\partial\beta_{l}}&=r\mu^{\prime\prime}(\alpha_{1}+\beta_{l}),~~l=r+1,\ldots,n-1;~~
 \cfrac{\partial^2 F_{1}}{\partial \beta_{s}\partial \alpha_{l}}=0,\\
 \cfrac{\partial^2 F_{1}}{\partial\beta_{l}^2}&=(r-1)\mu^{\prime\prime}(\alpha_{1}+\beta_{l}),~~l=1,\ldots,r;~~\cfrac{\partial^{2} F_{1}}{\partial\beta_{s}\partial\beta_{l}}=0, l\neq s,\\
 \cfrac{\partial^2 F_{1}}{\partial\beta_{l}^2}&=r\mu^{\prime\prime}(\alpha_{1}+\beta_{l}),~~l=r+1,\ldots,n-1;~~\cfrac{\partial^{2} F_{1}}{\partial\beta_{l}\partial\alpha_{s}}=0, s\neq 1,
\end{split}
    \]
and for $i\in\{r+1, \ldots, n\}$, we have
\[
    \begin{split}
&\cfrac{\partial^{2} F_{i}}{\partial \alpha_{s}\partial \alpha_{l}}=0, \   s\neq l; \ \cfrac{\partial^{2} F_{i}}{\partial\alpha^{2}_{i}}=\sum_{k\neq i}\mu^{\prime\prime}(\alpha_{i}+\beta_{k}),\\
&\cfrac{\partial^{2} F_{i}}{\partial \beta_{s}\partial \alpha_{i}}=\mu^{\prime\prime}(\alpha_{i}+\beta_{s}), \   s=1, \ldots, n-1, s\neq i; ~~\cfrac{\partial^{2} F_{i}}{\partial\beta_{i}\partial\alpha_{i}}=0,\\
&\cfrac{\partial^{2} F_{i}}{\partial \beta^{2}_{l}}=\mu^{\prime\prime}(\alpha_{i}+\beta_{l}) ,\   l=1, \ldots, n-1,l\neq i;;~~ \cfrac{\partial^{2} F_{i}}{\partial\beta_{s}\partial\beta_{l}}=0, s\neq l.
\end{split}
    \]
By the mean value theorem for vector-valued functions (\cite{Lang:1993}, p.341), for $\mathbf{x}, \mathbf{\mathbf{y}} \in R^{2n-r}$, we have
\[
F'_i(\mathbf{x}) - F'_i(\mathbf{y}) = J^{(i)}(\mathbf{x}-\mathbf{y}),
\]
where
\[
J^{(i)}_{s,l} = \int_0^1 \frac{ \partial F'_{i,s} }{\partial \theta_l}(t\mathbf{x}+(1-t)\mathbf{y})dt,~~ s,l=1,\ldots, 2n-r.
\]
By (\ref{ineq-mu-deriv-bound1}), we have
\[
\sum_{s,l} |J^{(i)}|\le
\begin{cases}
\frac{ 3r(n-1) }{ c_n } , & i=1, \\
\frac{ (n-1)(n-r) }{ c_n }+\frac{2(n-2)}{c_{n}} , & i=2, \ldots, n-r+1.
\end{cases}
\]
Similarly, we also have
$F'_i(\mathbf{x}) - F'_i(\mathbf{y}) = J^{(i)}(\mathbf{x}-\mathbf{y})$ and $\sum_{s,l}|J^{(i)}_{s,l}|\le  2(2n-r) /c_n $ for $i=n-r+2, \ldots, 2n-r$.
This gives that
\begin{equation}\label{eq-wf-a}
\parallel \widetilde{S}(F'(\textbf{x})-F'(\textbf{y}))\parallel
\le \begin{cases}
\frac{ 3b_{n} }{ c_n } \| \mathbf{x}-\mathbf{y} \|_\infty, & i=1, \\
\frac{(n-r+2)b_n }{ c_n } \| \mathbf{x}-\mathbf{y} \|_\infty, & i=2, \ldots, 2n-r.
\end{cases}
\end{equation}
and, by \eqref{approxi-inv2-beta-ho},
\begin{equation}\label{eq-wf-b}
\parallel \widetilde{W}(F'(\textbf{x})-F'(\textbf{y}))\parallel
\lesssim
\frac{ b_n^3 }{ n^2c_n^2 } \times \frac{ 3r(n-1)}{ c_n }\| \mathbf{x}-\mathbf{y} \|_\infty \lesssim \frac{ 4 b_n^3 }{ c_n^3 }\| \mathbf{x}-\mathbf{y} \|_\infty.
\end{equation}
Note that $\widetilde{W}=\widetilde{V}^{-1} - \widetilde{S}$
and $\widetilde{S}$ is defined at (\ref{definition-S222}).
By combining \eqref{eq-wf-a} and \eqref{eq-wf-b}, we have \eqref{eq-con-ho-a}.

Step 2. We claim that with probability at least $1-2(2n-r)/n^2$, we have
\begin{equation}\label{ineq-union-da}
\|\widetilde{V}^{-1}F(\bs{\theta})\|_\infty \lesssim
\left\{ b_n + \frac{ b_n^3 }{ c_n^2} \left( \frac{ r^{1/2} }{ n } + \frac{2n-r-1}{n} \right) \right\} \sqrt{\frac{\log n}{n}}.
\end{equation}
Recall that $a_{i,j}$, $1\le i \neq j \le n$, are independent Bernoulli random variables
and $\bar{d}_i = \sum_{j\neq i} \bar{a}_{i,j}$.
By \citeauthor{Hoeffding:1963}'s \citeyearpar{Hoeffding:1963} inequality, we have
\begin{equation*}
\P\left( |\bar{d}_i | \ge \sqrt{n\log n}  \right) \le 2\exp \left(- 2\frac{n\log n}{n} \right) \le  \frac{2}{n^2},~~i=1, \ldots, n.
\end{equation*}
By the union bound, we have
\begin{eqnarray}
\label{ineq-d-upper}
\P\Bigg( \max_{i=r+1, \ldots, n} | \bar{d}_i | \ge \sqrt{ n\log n} \Bigg)
\le  \sum_{i=r+1}^n \P\left(|\bar{d}_i| \geq \sqrt{n\log n} \right)
\le  \frac{2(n-r)}{n^2 }.
\end{eqnarray}
Similarly, we have
\begin{eqnarray}
\label{ineq-d-upper}
\P\Bigg( \max_{j=1, \ldots, n-1} | \bar{b}_j | \ge \sqrt{ n\log n} \Bigg)
\le  \sum_{j=1}^{n-1} \P\left(|\bar{b}_j| \geq \sqrt{n\log n} \right)
\le  \frac{2(n-1)}{n^2 }.
\end{eqnarray}
Note that $\sum_{i=1}^r \bar{d}_i = \sum_{i=1}^r \sum_{j\neq i} \bar{a}_{i,j}$
and the terms in the above summation are independent.
\citeauthor{Hoeffding:1963}'s \citeyearpar{Hoeffding:1963} inequality gives that
\begin{align*}
 & \P\left( |\sum_{i=1}^r \bar{d}_i | \ge \sqrt{r(n-1) \log n}  \right)
\le  2\exp \left(- 2\frac{r(n-1)\log n}{r(n-1)}\right)\le  \frac{2}{n^2}.
\end{align*}
The above arguments imply that with probability at least $1-2(2n-r)/n^2$,
\[
\|\widetilde{S} F(\bs{\theta})\| \le \begin{cases}
\frac{b_n}{rn} \times \sqrt{ r(n-1)\log n} \le \frac{b_n}{r^{1/2}} \sqrt{\frac{\log n}{n}}, & i=1, \\
b_n \sqrt{\frac{\log n}{n}}, & i=2,\ldots, 2n-r,
\end{cases}
\]
and, by \eqref{approxi-inv2-beta-ho},
\begin{eqnarray*}
\|\widetilde{W}F(\bs{\theta})\| & \lesssim &
\frac{b_n^3}{n^2c_n^2} \times (\sqrt{ r(n-1)\log n} + (2n-r-1)\sqrt{n\log n} )\\
& \lesssim &
\frac{ b_n^3 }{ c_n^2} \left( \frac{ r^{1/2} }{ n } + \frac{2n-r-1}{n} \right) \sqrt{\frac{\log n}{n}}.
\end{eqnarray*}

Step 3. This step is one combining step. By \eqref{eq-con-ho-a}, we can set
\[
K=O\left( \frac{b_n}{c_n} + \frac{b_n^3}{c_n^3}  \right),
\]
and
\[
\eta = O\left(  \left\{ b_n + \frac{ b_n^3 }{ c_n^2} \left( \frac{ r^{1/2} }{ n } + \frac{2n-r-1}{n} \right) \right\} \sqrt{\frac{\log n}{n}} \right),
\]
 in Lemma \ref{lemma:Newton:Kantovorich}.
If $b_n^6 /c_n^5=o( (n/\log n)^{1/2})$,
then
\[
h=K\eta \lesssim  \left( \frac{b_n^2}{c_n} + \frac{ b_n^4 }{ c_n^3 } \cdot \left( \frac{ r^{1/2} }{ n } + \frac{2n-r-1}{n} \right) + \frac{ b_n^4}{c_n^3 }
+ \frac{ b_n^6 }{ c_n^5 } \left( \frac{ r^{1/2} }{ n} + \frac{2n-r-1}{n} \right) \right)\sqrt{\frac{\log n}{n}}\to 0.
\]
This completes the proof.

\subsection{Proof of Lemma \ref{lemma-beta-homo-expan}}
\label{section-lemma9}

\begin{proof}
Since $\alpha_1=\cdots=\alpha_r$ and $\widehat{\alpha}_1^0 = \cdots = \widehat{\alpha}_r^0$ with $r\in\{1,\ldots,n-1\}$ under the null,
with some ambiguity of notations,
we still use $\bs{\widehat{\theta}}^0$ and $\bs{\theta}$
to denote vectors $(\widehat{\alpha}^0_1, \widehat{\alpha}_{r+1}^0, \ldots,  \widehat{\alpha}_{n}^0, \widehat{\beta}_{1}^0, \ldots, \widehat{\beta}_{n-1}^0)^\top$ and $(\alpha_1, \alpha_{r+1}, \ldots, \alpha_n, \beta_{1}, \ldots, \beta_{n-1})^\top$, respectively.
By Lemma \ref{lemma-con-beta-b}, if $b_n^6/c_n^5=o( (n/\log n)^{1/2} )$, then $\P(E_n) \ge 1 -4/n$, where
\[
E_n : = \left\{ \| \bs{\widehat{\theta}}^0 - \bs{\theta}  \|_\infty \lesssim  \frac{b_n^3}{c_n^2} \sqrt{ \frac{\log n}{n} } \right\}.
\]
The following calculations are based on the event $E_n$.

A second order Taylor expansion gives that
\begin{eqnarray*}
\mu( \widehat{\pi}_{ij}^0 ) & = & \mu( \pi_{ij} ) + \mu^{\prime}(\pi_{ij})( \widehat{\pi}_{ij}^0 - \pi_{ij})
+ \frac{1}{2} \mu^{\prime\prime}( \tilde{\pi}_{ij} ) ( \widehat{\pi}_{ij}^0 - \pi_{ij} )^2,
\end{eqnarray*}
where $\tilde{\pi}_{ij}$ lies between $\pi_{ij}$ and $\widehat{\pi}_{ij}^0$, and, for any $i,j$,
\[
\pi_{ij}=\alpha_i+\beta_j, ~~\widehat{\pi}_{ij}^0 = \widehat{\alpha}_i^0 + \widehat{\beta}_j^0, ~~
\tilde{\pi}_{ij}=\tilde{\alpha}_i + \tilde{\beta}_j.
\]
Let
\begin{align*}
 \tilde{h}_{ij}&=\frac{1}{2} \mu^{\prime\prime}( \tilde{\pi}_{ij} ) ( \widehat{\pi}_{ij}^0 - \pi_{ij} )^2, \\
\tilde h_{1} & =\sum_{j\neq1}\tilde h_{1j}+\sum_{j\neq2}\tilde h_{1j}+\cdots+\sum_{j\neq r}\tilde h_{1j}, \\
\tilde h_i& =\sum_{k\neq i}\tilde h_{ik} ,~~i=r+1, \ldots, n,\\
\tilde h_{n+i}&=\sum_{k\neq i}\tilde h_{ki} ,~~i=1, \ldots, n-1.
\end{align*}
Writing the above equations into a matrix form, we have
\begin{equation}\label{eq-expression-beta-star}
\bs{\widetilde{\mathbf{g}}} - \E \bs{\widetilde{\mathbf{g}}} = \widetilde{V}( \bs{\widehat{\theta}}^0 - \bs{\theta} ) + \bs{\widetilde{h}},
\end{equation}
where $\bs{\widetilde{h}}=(\tilde{h}_1, \tilde{h}_{r+1}, \ldots, \tilde{h}_{2n-1})^\top$.  It is equivalent to
\[
\bs{\widehat{\theta}}^0 - \bs{\theta} = \widetilde{V}^{-1} ( \bs{\widetilde{\mathbf{g}}} - \E \bs{\widetilde{\mathbf{g}}} ) - \widetilde{V}^{-1}\bs{\widetilde{h}}.
\]
In view of that $\max_{ij}|\mu^{\prime\prime}(\pi_{ij})| \le 1/c_n$ and the event $E_n$, we have
\begin{equation}\label{ineq-home-be-h1}
|\tilde h_1|  \lesssim  \frac{r(n-1)}{c_n} \|  \bs{\widehat{\theta}}^0 - \bs{\theta} \|_\infty^2
 \lesssim  \frac{ rb_n^6\log n }{c_n^5 },
\end{equation}
and, for $i=r+1, \ldots, 2n-1$,
\begin{equation}\label{ineq-home-be-hk}
|\tilde h_i| \lesssim  \frac{n-1}{c_n} \|  \bs{\widehat{\theta}}^0 - \bs{\theta} \|_\infty^2 \lesssim  \frac{ b_n^6\log n }{ c_n^5 }.
\end{equation}
By letting $\widetilde{V}^{-1} = \widetilde{S} + \widetilde{W}$, in view of \eqref{definition-S222}, \eqref{ineq-home-be-h1} and \eqref{ineq-home-be-hk}, we have
\begin{eqnarray}
\nonumber
\| \widetilde{V}^{-1}\bs{\tilde h} \|_\infty & \lesssim & \frac{ |\tilde h_1| }{ \tilde{v}_{11} } + \max_{i=r+1, \ldots, 2n-1} \frac{ |\tilde h_i| }{ v_{i,i} }+\frac{\tilde h_{2n}}{v_{2n,2n}}
+ \| \widetilde{W} \|_{\max} \left(\sum_{i=r+1}^{2n-1} |\tilde h_i| + |\tilde h_1| \right) \\
\nonumber
& \lesssim & \frac{ b_n^6 \log n}{c_n^5 } \cdot \frac{ b_n }{n } + \frac{ b_n^3 }{ n^2 c_n^2 } \cdot \left\{ \frac{rb_n^6\log n }{ c_n^5}
+ (2n-r)  \frac{ b_n^6 \log n}{c_n^5 } \right\} \\
\label{ineq-home-be-vh}
& \lesssim & \frac{ b_n^9 \log n}{ nc_n^7 }.
\end{eqnarray}
Now, we bound the error term $\| \widetilde{W} (\bs{\widetilde{\mathbf{g}}}- \E \bs{\widetilde{\mathbf{g}}}) \|_\infty$.
Note that
\begin{align*}
 &[\widetilde{W} (\bs{\widetilde{\mathbf{g}}}- \E \bs{\widetilde{\mathbf{g}}})]_i \\
 = & \tilde{w}_{i1} \sum_{j=1}^r \bar{d}_j + \sum_{j=r+1}^{n} \tilde{w}_{ij}\bar{d}_j+\sum_{j=n+1}^{2n-1} \tilde{w}_{ij}\bar{b}_j \\
 = & \tilde{w}_{i1} \sum_{j=1}^r \sum_{k\neq j} \bar{a}_{j,k}
 + \sum_{j=r+1}^{n}\tilde{w}_{ij}\sum_{k\neq j} \bar{a}_{j,k}+ \sum_{j=n+1}^{2n-1}\tilde{w}_{ij}\sum_{k\neq j} \bar{a}_{k,j}.
\end{align*}
Because $\E \bar{a}_{i,j}^2 \le 1/c_n$, we have
\begin{align*}
& \E \{[\widetilde{W} (\bs{\widetilde{\mathbf{g}}}- \E \bs{\widetilde{\mathbf{g}}})]_i \}^2 \\
\le & \left\{\frac{r(n-1)}{c_n} + \frac{(2n-1-r)(n-1)}{c_n} \right\}\|\widetilde{W} \|_{\max}^2 \\
 \lesssim & \frac{ n^2 }{c_n} \|\widetilde{W} \|_{\max}^2.
\end{align*}
It follows from Bernstein's inequality in Lemma \ref{lemma:bernstein}, with probability $1- 2n^{-2}$,  we have that
\begin{eqnarray*}
| [\widetilde{W} (\bs{\widetilde{\mathbf{g}}}- \E \bs{\widetilde{\mathbf{g}}})]_i | & \le & \sqrt{ 4\log n \left( \E \{[\widetilde{W} (\bs{\widetilde{\mathbf{g}}}- \E \bs{\widetilde{\mathbf{g}}})]_i \}^2   \right)} + \frac{4}{3} \cdot  \frac{ b_n }{ (n-1)c_n } \cdot \log n \\
& \lesssim & \frac{ b_n^3}{ n^2 c_n^2 } \cdot \frac{n (\log n)^{1/2}}{ c_n^{1/2}} \\
& \lesssim & \frac{ b_n^3(\log n)^{1/2} }{ n c_n^{5/2} }.
\end{eqnarray*}
By the uniform bound,  with probability at leas $1- 2/n$, we have
\begin{equation}\label{eq-hatbeta-bb}
\| \widetilde{W} (\bs{\widetilde{\mathbf{g}}}- \E \bs{\widetilde{\mathbf{g}}}) \|_\infty \lesssim \frac{ b_n^3(\log n)^{1/2} }{ n c_n^{5/2} }.
\end{equation}
By combining \eqref{eq-expression-beta-star}, \eqref{ineq-home-be-vh} and  \eqref{eq-hatbeta-bb}, with probability at least $1- O(n^{-1})$, we have
\begin{equation*}
\begin{array}{rcl}
\widehat{\alpha}_1^0 - \alpha_1 & = & \frac{ \sum_{i=1}^r \bar{d}_i }{ \tilde{v}_{11} }+\frac{\bar{b}_{n}}{v_{2n,2n}}+\gamma_1^{0}, \\
\widehat{\alpha}_i^0 - \alpha_i & = & \frac{ \bar{d}_i }{ v_{i,i} }+\frac{\bar{b}_{n}}{v_{2n,2n}} + \gamma_i^{0},~~ i=r+1, \ldots, n,\\
\widehat{\beta}_j^0 - \beta_j & = & \frac{ \bar{b}_j }{ v_{n+j,n+j} }-\frac{\bar {b}_{n}}{v_{2n,2n}} + \gamma_{n+j}^{0},~~ j=1, \ldots, n-1,
\end{array}
\end{equation*}
where $\gamma_1^{0}, \gamma_{r+1}^{0}, \ldots, \gamma_{2n-1}^{0}$ with probability at least $1- O(n^{-1})$ satisfy
\[
\gamma_i^{0} = (\widetilde{V}^{-1}\bs{\widetilde{h}})_i + [\widetilde{W} (\bs{\widetilde{\mathbf{g}}}- \E \bs{\widetilde{\mathbf{g}}})]_i = O\left( \frac{b_n^9 \log n}{ nc_n^7 } \right).
\]
\end{proof}

\subsection{Proof of \eqref{eq-thereom1b-z43}}
\label{section-proof-39-B20}
The expression of $B_2^0$ can be written as
\begin{eqnarray*}
-B_2^0  & = & \sum_{i=1}^{r}\sum_{j\neq i}\mu^{\prime\prime}(\pi_{ij}) (\widehat{\alpha}_1^0 - \alpha_1)^3
+\sum_{j=1}^{r} \sum_{i=r+1,i\neq j}^n (\mu^{\prime\prime}( \pi_{ij})+(r-1)\mu^{\prime\prime}(\pi_{1j}))( \widehat{\beta}_j^0 - \beta_j)^3  \\
&&+\sum_{j=r+1}^{n-1} \sum_{i=r+1,i\neq j}^n (\mu^{\prime\prime}( \pi_{ij})+r\mu^{\prime\prime}(\pi_{1j}))( \widehat{\beta}_j^0 - \beta_j)^3+ \sum_{i=r+1}^n \sum_{j\neq i}\mu^{\prime\prime}( \pi_{ij}) ( \widehat{\alpha}_i^0  - \alpha_i)^3 \\
&&+3(r -1)\sum_{j=1}^r \mu^{\prime\prime}(\pi_{1j}) ( \widehat{\alpha}_1^0  - \alpha_1)^2 ( \widehat{\beta}_j^0 - \beta_j)
+3r\sum_{j=r+1}^{n-1} \mu^{\prime\prime}(\pi_{1j}) ( \widehat{\alpha}_1^0  - \alpha_1)^2 ( \widehat{\beta}_j^0 - \beta_j)\\
&&+3(r -1)\sum_{j=1}^r \mu^{\prime\prime}(\pi_{1j}) ( \widehat{\alpha}_1^0  - \alpha_1) ( \widehat{\beta}_j^0 - \beta_j)^2
+3r\sum_{j=r+1}^{n-1} \mu^{\prime\prime}(\pi_{1j}) ( \widehat{\alpha}_1^0  - \alpha_1) ( \widehat{\beta}_j^0 - \beta_j)^2\\
&&+3\sum_{i=r+1}^n\sum_{j\neq i} \mu^{\prime\prime}(\pi_{ij}) ( \widehat{\alpha}_i^0  - \alpha_i)^2 ( \widehat{\beta}_j^0 - \beta_j)
+3\sum_{i=r+1}^n\sum_{j\neq i} \mu^{\prime\prime}(\pi_{ij}) ( \widehat{\alpha}_i^0  - \alpha_i) ( \widehat{\beta}_j^0 - \beta_j)^2.
\end{eqnarray*}
We shall in turn bound each term in the above summation.
To simplify notations, let
\begin{equation}
\begin{split}
  f_{ij} & = \mu^{\prime\prime}({\pi}_{ij} ), ~~~~f_{n+i}=\sum_{k=r+1;k\neq i}^{n}\mu^{\prime\prime}( \pi_{ki} ), ~~i=1,\ldots,n-1,\\
  f_{1} & = \sum_{i=1}^{r}\sum_{j\neq i}\mu^{\prime\prime}({\pi}_{ij} ),~~~~  f_{i} =\sum_{k=1;k\neq i}^{n-1} \mu^{\prime\prime}( \pi_{ik} ),~~ i=r+1,\ldots,n.
\end{split}
\end{equation}
In view of \eqref{ineq-mu-deriv-bound}, we have
\begin{equation}\label{ineq-fi-fij}
\max_{i=r+1,\ldots,n} |f_i| \le \frac{ n-1}{ c_n },~~\max_{i=1,\ldots,n-1} |f_{n+i}| \le \frac{ n-r}{ c_n }, ~~  |f_{1} | \le \frac{ r(n-1)}{c_n},~~ \max_{i,j} |f_{ij} | \le \frac{ 1}{c_n}.
\end{equation}
Because $\sum_{i=1}^r \bar{d}_i$ can be expressed as the sum of $r(n-1)$ independent and bounded random variables,
\[
\sum_{i=1}^r \bar{d}_i = \sum_{i=1}^r \sum_{j\neq i} \bar{a}_{i,j}
\]
and
\[
\E (\sum_{i=1}^r \sum_{j\neq i} \bar{a}_{i,j})^2
\le\frac{r(n-1)}{c_n},
\]
by Bernstern's inequality, with probability at least $ 1- 2(rn)^{-2}$, we have
\begin{equation}\label{eq-sum-hom-d}
\left| \sum_{i=1}^r \bar{d}_i \right| \lesssim \sqrt{ 2\log (rn) \times \frac{r(n-1)}{c_n} } + \frac{12}{3}\log n \lesssim
\sqrt{\frac{rn\log n}{c_n}}.
\end{equation}
By \eqref{alpha1} and \eqref{eq-sum-hom-d}, we have
\begin{eqnarray}
\nonumber
| f_1(\widehat{\alpha}_1^0 - \alpha_1)^3 | & \leq & \frac{r(n-1)}{c_n}  \left|\left( \frac{ \sum_{i=1}^r \bar{d}_i }{ \tilde{v}_{11} } +\frac{\bar{b}_{n}}{v_{2n,2n}}+ \gamma_1^{0} \right)^3 \right|\\
\nonumber
& =&  \frac{r(n-1)}{c_n}  \left|
 ( \frac{ \sum_{i=1}^r \bar{d}_i }{ \tilde{v}_{11} })^3 +(\frac{\bar{b}_{n}}{v_{2n,2n}})^{3}+(\gamma_1^{0})^{3} +3( \frac{ \sum_{i=1}^r \bar{d}_i }{ \tilde{v}_{11} })^2\frac{\bar{b}_{n}}{v_{2n,2n}} \right.\\
\nonumber
 &+&\left. 3 ( \frac{ \sum_{i=1}^r \bar{d}_i }{ \tilde{v}_{11} })^2 \gamma_1^{0}+3\frac{ \sum_{i=1}^r \bar{d}_i }{ \tilde{v}_{11} }(\frac{\bar{b}_{n}}{v_{2n,2n}})^2
 + 3  \frac{ \sum_{i=1}^r \bar{d}_i }{ \tilde{v}_{11} } (\gamma_1^{0})^2 +3\frac{\bar{b}_{n}}{v_{2n,2n}}(\gamma_1^{0})^2\right.\\
\nonumber
 &+&\left.3(\frac{\bar{b}_{n}}{v_{2n,2n}})^2\gamma_1^{0}+6 \frac{ \sum_{i=1}^r \bar{d}_i }{ \tilde{v}_{11} }\frac{\bar{b}_{n}}{v_{2n,2n}}\gamma_1^{0}
\right| \\
\nonumber
& \lesssim &  \frac{r(n-1)}{c_n} \left\{
\frac{b_n^3}{ (rn)^3}  \cdot (\frac{rn\log n}{c_n})^{3/2}+(\frac{\bar{b}_{n}}{v_{2n,2n}})^{3}+ (\frac{ b_n^9 \log n}{ nc_n^7})^3 \right. \\
\nonumber
&& \left.+
\frac{b_n^2}{ (rn)^2}\cdot\frac{rn\log n}{c_n}  \cdot \sqrt{n\log n} \cdot\frac{b_{n}}{n} +
\frac{b_n^2}{ (rn)^2} \cdot \frac{rn\log n}{c_n} \cdot \frac{ b_n^9 \log n}{ nc_n^7} \right. \\
\nonumber
&& \left.+
\frac{b_n}{ rn} \cdot \sqrt{\frac{rn\log n}{c_n}} \cdot n\log n  \cdot (\frac{b_{n}}{n})^2
+\frac{b_n}{ rn} \cdot \sqrt{\frac{rn\log n}{c_n}} \cdot (\frac{ b_n^9 \log n}{ nc_n^7})^2\right. \\
\nonumber
&& \left.+
 \sqrt{n\log n} \cdot\frac{b_{n}}{n}\cdot (\frac{ b_n^9 \log n}{ nc_n^7})^2+n\log n  \cdot (\frac{b_{n}}{n})^2 \cdot\frac{ b_n^9 \log n}{ nc_n^7}\right. \\
\nonumber
&& \left.+
 \sqrt{n\log n} \cdot\frac{b_{n}}{n}\cdot \frac{ b_n^9 \log n}{ nc_n^7}\cdot\frac{b_n}{ rn} \cdot \sqrt{\frac{rn\log n}{c_n}}
\right\} \\
\nonumber
& \lesssim &  \frac{ b_n^3 (\log n)^{3/2}}{ c_n^{5/2}(rn)^{1/2} }  + \frac{ b_n^{27}(\log n)^3r }{n^2 c_n^{22} } + \frac{ b_n^{3}(\log n)^{3/2}}{n^{1/2} c_n^{2} } + \frac{ b_n^{11}(\log n)^{2}}{n c_n^{9} } \\
\nonumber
&&+ \frac{ b_n^3 (\log n)^{3/2}(rn)^{1/2}}{ c_n^{3/2}n }
+ \frac{ b_n^{19} (\log n)^{5/2}(rn)^{1/2}}{ c_n^{31/2}n^2 }+ \frac{r b_n^{19}(\log n)^{5/2}}{n^{3/2} c_n^{15} }\\
\nonumber
&&+ \frac{ b_n^{11}(\log n)^{2}r}{n c_n^{8} }+ \frac{ b_n^{11}(\log n)^{2}r^{1/2}}{n c_n^{17/2} }+\frac{r(n-1)}{c_n}(\frac{\bar{b}_{n}}{v_{2n,2n}})^{3}.
\end{eqnarray}
Therefore, if $b_n^{11}/c_n^8=o( n/(\log n)^2r )$, then
\begin{equation}\label{ineq-1b-Q1}
| f_1(\widehat{\alpha}_1^0 - \alpha_1)^3 | = O\left(  \frac{ b_n^{11}(\log n)^{2}r}{n c_n^{8} } \right)+\frac{r(n-1)\bar{b}_{n}^3}{c_n v_{2n,2n}^3} = o(1)+\frac{r(n-1)\bar{b}_{n}^3}{c_n v_{2n,2n}^3} .
\end{equation}
With similar arguments as in the proof of Lemma \ref{lemma:beta3:err}, if $b_n^{11}/c_n^8=o( n/(\log n)^2(n-r) )$ we have
\begin{eqnarray}
\label{ineq-1b-Q2}
|\sum_{j=1}^{n-1} \sum_{i=r+1,i\neq j}^n \mu^{\prime\prime}( \pi_{ij})+r\mu^{\prime\prime}(\pi_{1j}))( \widehat{\beta}_j^0 - \beta_j)^3 |
 &=& o(1)-\frac{r(n-1)\bar{b}_{n}^3}{c_n v_{2n,2n}^3}-\frac{n(n-r)\bar{b}_{n}^3}{c_n v_{2n,2n}^3}.  \\
\label{ineq-1b-Q3}
|\sum_{i=r+1}^n \sum_{j\neq i}\mu^{\prime\prime}( \pi_{ij}) ( \widehat{\alpha}_i^0  - \alpha_i)^3| &=&   o(1)+\frac{n(n-r)\bar{b}_{n}^3}{c_n v_{2n,2n}^3}.\\
\label{ineq-1b-Q4}
|\sum_{i=r+1}^n\sum_{j\neq i} \mu^{\prime\prime}(\pi_{ij}) ( \widehat{\alpha}_i^0  - \alpha_i)^2 ( \widehat{\beta}_j^0 - \beta_j)|&=& o(1)-\frac{n(n-r)\bar{b}_{n}^3}{c_n v_{2n,2n}^3}.\\
\label{ineq-1b-Q5}
|\sum_{i=r+1}^n\sum_{j\neq i} \mu^{\prime\prime}(\pi_{ij}) ( \widehat{\alpha}_i^0  - \alpha_i) ( \widehat{\beta}_j^0 - \beta_j)^2|&=& o(1)+\frac{n(n-r)\bar{b}_{n}^3}{c_n v_{2n,2n}^3}.
\end{eqnarray}
We now consider the terms $(r -1)\sum_{j=1}^r \mu^{\prime\prime}(\pi_{1j}) ( \widehat{\alpha}_1^0  - \alpha_1)^2 ( \widehat{\beta}_j^0 - \beta_j)+r\sum_{j=r+1}^{n-1} \mu^{\prime\prime}(\pi_{1j}) ( \widehat{\alpha}_1^0  - \alpha_1)^2 ( \widehat{\beta}_j^0 - \beta_j)$.
\begin{eqnarray*}
&&r|\sum_{j=1}^{n-1} \mu^{\prime\prime}(\pi_{1j}) ( \widehat{\alpha}_1^0  - \alpha_1)^2 ( \widehat{\beta}_j^0 - \beta_j)|\\
&=& r |\sum_{j=1}^{n-1} \mu^{\prime\prime}(\pi_{1j}) (\frac{ \sum_{i=1}^r \bar{d}_i }{ \tilde{v}_{11} } +\frac{\bar{b}_{n}}{v_{2n,2n}}+ \gamma_1^{0})^2(\frac{ \bar{b}_j }{ v_{n+j,n+j} }-\frac{\bar {b}_{n}}{v_{2n,2n}} + \gamma_{n+j}^{0})|\\
& \lesssim &r \left|\sum_{j=1}^{n-1} \mu^{\prime\prime}(\pi_{1j})  \frac{( \sum_{i=1}^r \bar{d}_i)^2 }{ \tilde{v}_{11}^2 }\frac{ \bar{b}_j }{ v_{n+j,n+j} }- \sum_{j=1}^{n-1} \mu^{\prime\prime}(\pi_{1j})  \frac{( \sum_{i=1}^r \bar{d}_i)^2 }{ \tilde{v}_{11}^2 }\frac{\bar{b}_{n}}{v_{2n,2n}}\right. \\
& + &\left. \sum_{j=1}^{n-1} \mu^{\prime\prime}(\pi_{1j})  \frac{( \sum_{i=1}^r \bar{d}_i)^2 }{ \tilde{v}_{11}^2 }\gamma_{n+j}^{0}+\sum_{j=1}^{n-1} \mu^{\prime\prime}(\pi_{1j})  (\frac{\bar{b}_{n}}{v_{2n,2n}})^2\frac{ \bar{b}_j }{ v_{n+j,n+j} } -\sum_{j=1}^{n-1} \mu^{\prime\prime}(\pi_{1j})  (\frac{\bar{b}_{n}}{v_{2n,2n}})^3\right. \\
& + &\left.\sum_{j=1}^{n-1} \mu^{\prime\prime}(\pi_{1j})  (\frac{\bar{b}_{n}}{v_{2n,2n}})^2\gamma_{n+j}^{0}+ \sum_{j=1}^{n-1} \mu^{\prime\prime}(\pi_{1j})  (\gamma_1^{0})^2\frac{ \bar{b}_j }{ v_{n+j,n+j} }-\sum_{j=1}^{n-1} \mu^{\prime\prime}(\pi_{1j})  (\gamma_1^{0})^2\frac{ \bar{b}_n }{ v_{2n,2n} } \right. \\
& + &\left.\sum_{j=1}^{n-1} \mu^{\prime\prime}(\pi_{1j})  (\gamma_1^{0})^2\gamma_{n+j}^{0}\right|\\
&\lesssim&
\frac{ b_n^{11} (\log n)^2 }{ n c_n^9 }
+ \frac{ b_n^{2} (\log n)^{3/2}r }{ n^{3/2} c_n}+\frac{ b_n^{11} (\log n)^2 r}{ n c_n^8 }+\frac{ b_n^{18} (\log n)^2r }{ n c_n^{15} }-\frac{r(n-1)\bar{b}_{n}^{3}}{c_{n}v_{2n,2n}^{2}}.
\end{eqnarray*}
If $b_n^{11}/c_n^8=o( n/(\log n)^2r )$, then
\begin{equation}\label{ineq-1b-Q6}
  r|\sum_{j=1}^{n-1} \mu^{\prime\prime}(\pi_{1j}) ( \widehat{\alpha}_1^0  - \alpha_1)^2 ( \widehat{\beta}_j^0 - \beta_j)| =o(1)-\frac{r(n-1)\bar{b}_{n}^{3}}{c_{n}v_{2n,2n}^{2}}.
\end{equation}
Similarly, we have
\begin{equation}\label{ineq-1b-Q7}
  r|\sum_{j=1}^{n-1} \mu^{\prime\prime}(\pi_{1j}) ( \widehat{\alpha}_1^0  - \alpha_1) ( \widehat{\beta}_j^0 - \beta_j)^2| =o(1)+\frac{r(n-1)\bar{b}_{n}^{3}}{c_{n}v_{2n,2n}^{2}}.
\end{equation}
By combining the upper bounds of the above in \eqref{ineq-1b-Q1}, \eqref{ineq-1b-Q2}, \eqref{ineq-1b-Q3}, \eqref{ineq-1b-Q4}, \eqref{ineq-1b-Q5},  \eqref{ineq-1b-Q6} and \eqref{ineq-1b-Q7}, it shows \eqref{eq-thereom1b-z43} in the main text.

\section{Proofs of supported Lemmas in the proof of Theorem 1 (b)}
\label{section-theorem2a}
This section presents the proofs of supported Lemmas in the proof of Theorem 1 (b) and two vanishing remainder terms.
This section is organized as follows.
Sections \ref{section-proof-lemma1010}, \ref{le12} and \ref{section-proof-lemma13}
present the proofs of Lemmas \ref{w2-error-2b}, \ref{lemma-hat-beta-diff-2b} and \ref{lemma-gg0}, respectively.
Sections \ref{subsection:B2B20} and \ref{subsection-B3B30} present the proofs of orders of two remainder terms
$B_2-B_2^0$ in \eqref{ineq-B2-B20} and $B_3-B_3^0$ in \eqref{ineq-B3-B20} in the main text, respectively.

\subsection{Proof of Lemma \ref{w2-error-2b}}
\label{section-proof-lemma1010}
We partition $\widetilde{S}$ in \eqref{definition-S222} into four blocks:
\begin{equation*}
\widetilde{S} = \begin{pmatrix} \tilde{s}_{11} & \bs{\tilde{s}}_{12} \\
\bs{\tilde{s}}_{21} & \widetilde{S}_{22}
\end{pmatrix}=\begin{pmatrix} 1/\tilde{v}_{11} & \bs{0}^\top_{2n-1-r} \\
\bs{0}_{2n-1-r} &D_{22}
\end{pmatrix}+
\begin{array}{ccc}
\begin{pmatrix} \frac{1}{v_{2n,2n}}& \ldots& \frac{1}{v_{2n,2n}}&  -\frac{1}{v_{2n,2n}}&\ldots& -\frac{1}{v_{2n,2n}} \\
\vdots&&\vdots&\vdots& &\vdots\\
\frac{1}{v_{2n,2n}}& \ldots& \frac{1}{v_{2n,2n}}&  -\frac{1}{v_{2n,2n}}&\ldots& -\frac{1}{v_{2n,2n}} \\
-\frac{1}{v_{2n,2n}}& \ldots& -\frac{1}{v_{2n,2n}}&  \frac{1}{v_{2n,2n}}&\ldots& \frac{1}{v_{2n,2n}} \\
\vdots&&\vdots&\vdots& &\vdots\\
-\frac{1}{v_{2n,2n}}& \ldots& -\frac{1}{v_{2n,2n}}&  \frac{1}{v_{2n,2n}}&\ldots& \frac{1}{v_{2n,2n}} \\
\end{pmatrix}
&
\begin{tiny}
\begin{array}{l}
\left.\rule{-5mm}{11mm}\right\}{n-r+1}\\
\left.\rule{-5mm}{11mm}\right\}{n-1}
\end{array}
\end{tiny}
\\
\begin{tiny}
\begin{array}{cc}
\underbrace{\rule{37mm}{0mm}}_{n-r+1}
\underbrace{\rule{37mm}{0mm}}_{n-1}
\end{array}
\end{tiny}
\end{array}
\end{equation*}
Note that
\[
\left[
\begin{pmatrix} \tilde{s}_{11} & \bs{\tilde{s}}_{12} \\
\bs{\tilde{s}}_{21} & \widetilde{S}_{22}
\end{pmatrix}
+
\begin{pmatrix} \tilde{w}_{11} & \bs{\tilde{w}}_{12} \\
\bs{\tilde{w}}_{21} & \widetilde{W}_{22}
\end{pmatrix}
\right]
\begin{pmatrix} \tilde{v}_{11} & \bs{\tilde{v}}_{12} \\ \bs{\tilde{v}}_{21} & V_{22} \end{pmatrix}
=I_{(2n-r)\times (2n-r)},
\]
and
\[
\left[
\begin{pmatrix} S_{11} & S_{12} \\
S_{21} & S_{22}
\end{pmatrix}
+
\begin{pmatrix} W_{11} & W_{12} \\
W_{21} & W_{22}
\end{pmatrix}
\right]
\begin{pmatrix}
V_{11} & V_{12} \\
V_{21} & V_{22}
\end{pmatrix}
=I_{(2n-1)\times(2n-1)}.
\]
It follows that
\[
D_{22} V_{22} + \bs{\tilde{w}}_{21}\bs{\tilde{v}}_{12} + \widetilde{W}_{22}V_{22} =D_{22} V_{22} + {{W}}_{21}{{V}}_{12} +{W}_{22}V_{22},
\]
we have
\begin{align*}
  W_{22} - \widetilde{W}_{22} & = V_{22}^{-1}(\bs{\tilde{w}}_{21}\bs{\tilde{v}}_{12}-W_{21}V_{12}) \\
   & =H_{2}+D_{22}(\bs{\tilde{w}}_{21}\bs{\tilde{v}}_{12}-W_{21}V_{12})+\frac{1}{v_{2n,2n}}H_{1}\begin{array}{ccc}
\begin{pmatrix} \mathbf{1}& \mathbf{-1}\\
\mathbf{-1}& \mathbf{1} \\
\end{pmatrix}
&
\begin{tiny}
\begin{array}{l}
\left.\rule{-5mm}{4mm}\right\}{n-r}\\
\left.\rule{-5mm}{4mm}\right\}{n-1}
\end{array}
\end{tiny}
\\
\begin{tiny}
\begin{array}{cc}
\underbrace{\rule{8mm}{-15mm}}_{n-r}
\underbrace{\rule{8mm}{-15mm}}_{n-1},
\end{array}
\end{tiny}
\end{array}
\end{align*}
where
\[
H_{1}=\bs{\tilde{w}}_{21}\bs{\tilde{v}}_{12}-W_{21}V_{12},~~H_{2}=(V_{22}^{-1}-{S}_{22})(\bs{\tilde{w}}_{21}\bs{\tilde{v}}_{12}-W_{21}V_{12}).
\]
Note that $r$ is a fixed positive integer. In view of \eqref{ineq-V-S-appro-upper-b}, we have
\begin{eqnarray*}
&&|(V_{22}^{-1}-{S}_{22})(\bs{\tilde{w}}_{21}\bs{\tilde{v}}_{12}-W_{21}V_{12})| \\
&  = & |\sum_{k=1}^{2n-1-r} (V_{22}^{-1}-{S}_{22})_{ik} (\bs{\tilde{w}}_{21})_{k1}(\bs{\tilde{v}}_{12})_{1j}-\sum_{k=1}^{2n-1-r}\sum_{h=1}^r (V_{22}^{-1}-{S}_{22})_{ik}  (W_{21})_{k h} (V_{12})_{hj}| \\
& \lesssim & (2n-1-r)\cdot \| W \|_{\max}^2 \cdot \frac{1}{c_n}-(2n-1-r)\cdot r \cdot \| W \|_{\max}^2 \cdot \frac{1}{c_n} \lesssim \frac{ b_n^6 }{ n^3c_n^5 },
\end{eqnarray*}
and
\begin{eqnarray*}
&&|( D_{22}\bs{\tilde{w}}_{21}\bs{\tilde{v}}_{12})_{ij}-( D_{22} W_{21} V_{12})_{ij}|\\
 & = & |\sum_{k=1}^{2n-1-r}(D_{22})_{ik} (\bs{\tilde{w}}_{21})_{k1}(\bs{\tilde{v}}_{12})_{1j}-\sum_{k=1}^{2n-1-r}\sum_{h=1}^r ( D_{22})_{ik} (W_{21})_{k h} (V_{12})_{hj}| \\
& \le & \frac{ b_n }{n} \cdot \frac{ b_n^3 }{ n^2c_n^2 } \cdot \frac{1}{c_n} -r \cdot \frac{ b_n }{n} \cdot \frac{ b_n^3 }{ n^2c_n^2 } \cdot \frac{1}{c_n}  \lesssim \frac{b_n^4}{ n^2c_n^3}.
\end{eqnarray*}
This leads to \eqref{ineq-WW-diff-2b}.

\subsection{Proof of Lemma \ref{lemma-hat-beta-diff-2b}}
\label{le12}
By Theorem 2 in \cite{Yan:Leng:Zhu:2016}, we have
\begin{equation}\label{eq-a-zcc1}
\begin{split}
\widehat{\alpha}_i - \alpha_i& = \frac{ \bar{d}_i }{ v_{i,i} }+ \frac{\bar b_{n} }{  v_{2n,2n} }+ \gamma_i, ~~i=r+1, \ldots, n \\
\widehat{\beta}_j -\beta_j& = \frac{ \bar{b}_j }{ v_{n+j,n+j} }- \frac{\bar b_{n} }{  v_{2n,2n} }+ \gamma_{n+j}, ~~j=1, \ldots, n-1,
\end{split}
\end{equation}
where $\max\{ |\gamma_i|,|\gamma_{n+j}|\}\leq \frac{b_{n}^{9}\log n}{c_{n}^{7}n}$.
And by Lemma \ref{lemma-beta-homo-expan}, we have
\begin{eqnarray}
\label{alphai1}
 \widehat{\alpha}_i^0 - \alpha_i & = & \frac{ \bar{d}_i }{ v_{i,i} }+\frac{\bar{b}_{n}}{v_{2n,2n}} + \gamma_i,~~ i=r+1, \ldots, n,\\
\label{betaj1}
\widehat{\beta}_j^0 - \beta_j & = & \frac{ \bar{b}_j }{ v_{n+j,n+j} }-\frac{\bar {b}_{n}}{v_{2n,2n}} + \gamma_{n+j},~~ j=1, \ldots, n-1,
\end{eqnarray}
where $\gamma_{r+1}, \ldots, \gamma_{2n-1}$ with probability at least $1- O(n^{-1})$ satisfy
\[
\gamma_i = (\widetilde{V}^{-1}\bs{\widetilde{h}})_i + [\widetilde{W} (\bs{\widetilde{\mathbf{g}}}- \E \bs{\widetilde{\mathbf{g}}})]_i = O\left( \frac{b_n^9 \log n}{ nc_n^7 } \right).
\]
By combining \eqref{eq-a-zcc1}, \eqref{alphai1} and \eqref{betaj1}, this completes the proof.

\subsection{Proof of Lemma \ref{lemma-gg0}}
\label{section-proof-lemma13}
By Lemma \ref{lemma-beta-homo-expan} and $\widetilde{V}^{-1}=\widetilde{S}+\widetilde{W}$, we have
\[
\gamma_i^0  = \frac{ \tilde h_i }{ v_{i,i} }+(-1)^{1{(i>n-r)}}\frac{\tilde h_{2n}}{v_{2n,2n}}+ (\widetilde{W}\bs{\widetilde{h}})_i + [\widetilde{W} (\bs{\widetilde{\mathbf{g}}}- \E \bs{\widetilde{\mathbf{g}}})]_i,~~i=r+1, \ldots, 2n-1.
\]
Because $r$ is a fixed constant, by \eqref{eq-homo-tildeh}, we have
\[
(\widetilde{W}\bs{\widetilde{h}})_i =(\widetilde{W}_{22}\bs{\widetilde{h}}_{2})_i +O_{p}(\frac{b_{n}^{9}\log n }{n^{2}}).
\]
By \eqref{approxi-inv2-beta-ho}, we have
\[
[\widetilde{W} (\bs{\widetilde{\mathbf{g}}}- \E \bs{\widetilde{\mathbf{g}}})]_i=[\widetilde{W}_{22} \bs{\mathbf{\bar g}}_{2}]_i+O_{p}(\frac{b_{n}^{3}(\log n)^{1/2} }{n^{3/2}}).
\]
It follows that for $i=r+1, \ldots, 2n-1$, we have
\begin{equation}\label{gamma0}
 \gamma_i^0  = \frac{ \tilde h_i }{ v_{i,i} }+(-1)^{1{(i>n-r)}}\frac{\tilde h_{2n}}{v_{2n,2n}}+ [\widetilde{W}_{22} \bs{\mathbf{\bar g}_{2}}]_i+(\widetilde{W}_{22}\bs{\widetilde{h}}_{2})_i+O_{p}(\frac{b_{n}^{3}(\log n)^{1/2} }{n^{3/2}}).
\end{equation}
Similarly, for $i=r+1, \ldots, 2n-1$, we have
\begin{equation}\label{gamma}
 \gamma_i  = \frac{  h_i }{ v_{i,i} }+(-1)^{1{(i>n-r)}}\frac{ h_{2n}}{v_{2n,2n}}+ [{W}_{22} \bs{\mathbf{\bar g}}_{2}]_i+({W}_{22}\bs{{h}}_{2})_i+O_{p}(\frac{b_{n}^{3}(\log n)^{1/2} }{n^{3/2}}).
\end{equation}
Combining \eqref{gamma0} and \eqref{gamma}  yields
 \begin{eqnarray}
  \nonumber
  \gamma_i-\gamma_i^0 &=&\underbrace{\frac{  h_i }{ v_{i,i} } - \frac{ \tilde h_i }{ v_{i,i} }}_{T_{1}}+ \underbrace{(-1)^{1{(i>n-r)}}\frac{ h_{2n}}{v_{2n,2n}}-(-1)^{1{(i>n-r)}}\frac{\tilde h_{2n}}{v_{2n,2n}}}_{T_{2}}\\
    \label{gamma-gamma0}
     &&+ \underbrace{[{W}_{22} \bs{\mathbf{\bar g}}_{2}]_i-[\widetilde{W}_{22} \bs{\mathbf{\bar g}_{2}}]_i}_{T_{3}}+\underbrace{({W}_{22}\bs{{h}}_{2})_i-(\widetilde{W}_{22}\bs{\widetilde{h}}_{2})_i}_{T_{4}}.
  \end{eqnarray}

We bound $T_{i}, i = 1, 2, 3, 4$ in turns. We first bound $T_{1}$. For $i=r+1, \ldots, n$, we have
\[
h_i- \tilde h_i=\frac{1}{2}\sum_{j=1,j\neq i}^{n} \left\{\mu^{\prime\prime}( \tilde{\pi}_{ij} ) (\widehat{\pi}_{ij} - \pi_{ij})^2- \mu^{\prime\prime}( \tilde{\pi}_{ij}^{0} ) (\widehat{\pi}_{ij}^{0} - \pi_{ij})^2\right\}.
\]
With the use of the mean value theorem and Lemmas \ref{lemma-con-beta-b} and \ref{lemma-hat-beta-diff-2b}, we have
\begin{align*}
& | \mu^{\prime\prime} (\tilde{\pi}_{ij})(\widehat{\pi}_{ij} - \pi_{ij} )^2
- \mu^{\prime\prime} (\tilde{\pi}_{ij}^0)(\widehat{\pi}_{ij}^0 - \pi_{ij} )^2 | \\
=&  | [\mu^{\prime\prime} (\tilde{\pi}_{ij}) - \mu^{\prime\prime} (\tilde{\pi}_{ij}^0)] (\widehat{\pi}_{ij} - \pi_{ij} )^2
+ \mu^{\prime\prime} (\tilde{\pi}_{ij}^0)[(\widehat{\pi}_{ij} - \pi_{ij} )^2-(\widehat{\pi}_{ij}^0 - \pi_{ij} )^2] |
\\
\le &  \mu^{\prime\prime\prime}( \dot{\pi}_{ij})|\tilde{\pi}_{ij}-\tilde{\pi}_{ij}^0|(\widehat{\pi}_{ij} - \pi_{ij} )^2 +  \mu^{\prime\prime} (\tilde{\pi}_{ij}^0)|\widehat{\pi}_{ij} - \widehat{\pi}_{ij}^0 |
\times ( | \widehat{\pi}_{ij}^0 - \pi_{ij}| +  |\widehat{\pi}_{ij} - \pi_{ij}| ) \\
\leq &\frac{1}{c_{n}}\cdot\frac{ b_n^{9} \log n }{ nc_n^7}\cdot\frac{b_{n}^{3}}{c_{n}^{2}}\sqrt{\frac{\log n}{n}}+\frac{1}{c_{n}}\left(\frac{b_{n}^{3}}{c_{n}^{2}}\sqrt{\frac{\log n}{n}}\right)^3\\
\leq &\frac{ b_n^{12} (\log n)^{3/2} }{ n^{3/2}c_n^{10}}.
\end{align*}
This gives that
\begin{equation}\label{hi}
  |h_i- \tilde h_i|\lesssim \frac{ b_n^{12} (\log n)^{3/2} }{ n^{1/2}c_n^{10}}.
\end{equation}
It follows that
\begin{equation}\label{T1}
  |T_{1}|=\frac{  |h_i- \tilde h_i|}{v_{i,i}}\lesssim \frac{ b_n^{13} (\log n)^{3/2} }{ n^{3/2}c_n^{10}}.
\end{equation}
Similarly, for $i=n+1, \ldots, 2n-1$, we also have $ |T_{1}|\lesssim \frac{ b_n^{13} (\log n)^{3/2} }{ n^{3/2}c_n^{10}}.$
With the same arguments, we have
\begin{equation}\label{T2}
  |T_{2}|\lesssim \frac{ b_n^{13} (\log n)^{3/2} }{ n^{3/2}c_n^{10}}.
\end{equation}

Next, we bound $T_{3}$. By \eqref{ineq-WW-diff-2b}, we have
\[
(W_{22} - \widetilde{W}_{22}) {\mathbf{\bar g}_{2}}=\frac{1}{v_{2n,2n}}H_{1}(-1)^{1{(i>n-r)}}(\bar b_{n}-\sum_{i=1}^{r}\bar d_{i})\mathbf{1}_{2n-r-1}+H_{2}{\mathbf{\bar g}_{2}}.
\]
By \eqref{H1h2}, we have
\begin{equation}\label{first}
\begin{split}
  & \|\frac{1}{v_{2n,2n}}H_{1}(-1)^{1{(i>n-r)}}(\bar b_{n}-\sum_{i=1}^{r}\bar d_{i})\mathbf{1}_{2n-r-1}\|\\
   \lesssim& \frac{b_{n}(n\log n)^{1/2}}{n}\cdot(2n-r-1)\cdot\frac{b_{n}^{4}}{n^{2}c_{n}^{3}}
    \lesssim\frac{b_{n}^{5}(\log n)^{1/2}}{n^{3/2}c_{n}^{3}}.
\end{split}
\end{equation}
We use Benstern's inequality to bound $H_{2}{\mathbf{\bar g}_{2}}$. Note that
\begin{align*}
    [H_{2} {\mathbf{\bar g}_{2}}]_{i}
  = \sum_{j=r+1}^{n} (H_{2})_{ij}\sum_{k\neq j}\bar a_{j,k}+\sum_{j=n+1}^{2n-1} (H_{2})_{ij}\sum_{k\neq j-n}\bar a_{k,j-n}.
\end{align*}
Because the terms involved in the sum are independent and bounded,
Benstern's inequality in Lemma \ref{lemma:bernstein} gives that with probability at least $1- O(n^{-2})$, we have
\begin{eqnarray}
\nonumber
|[H_{2} \bs{\bar{\mathbf{g}}}_2]_i | & \lesssim & \|H_{2}\|_{\max} \sqrt{ 4 \log n \cdot \frac{n(n-r)}{c_n}}
+ \frac{4}{3} \|H_{2}\|_{\max} \log n \\
\nonumber
& \lesssim & \frac{b_n^6}{n^3 c_n^5} \cdot \frac{n(\log n)^{1/2}}{ c_n^{1/2} } \\
\label{ew-WWd}
& \lesssim & \frac{ b_n^6 (\log n)^{1/2} }{ n^2 c_n^{11/2} },
\end{eqnarray}
where
\[
\E \left( [H_{2} \bs{\bar{\mathbf{g}}}_2]_i  \right)^2 \lesssim \frac{ n(n-r) \|H_{2}\|_{\max}^{2}}{ c_n}.
\]
By combining \eqref{first} and \eqref{ew-WWd}, with probability at least $1- O(n^{-2})$, we have
\begin{equation}\label{TT3}
  \max_{i=r+1, \ldots, 2n-1}|\{(W_{22} - \widetilde{W}_{22}) {\mathbf{\bar g}_{2}}\}_{i}|\lesssim\frac{b_{n}^{5}(\log n)^{1/2}}{n^{3/2}c_{n}^{3}}+\frac{ b_n^6 (\log n)^{1/2} }{ n^2 c_n^{11/2} }.
\end{equation}
Last, we bound $T_{4}$. By Lemma \ref{w2-error-2b}, we have
\begin{eqnarray}
\nonumber
{W}_{22}\bs{{h}}_{2}-\widetilde{W}_{22}\bs{\widetilde{h}}_{2} &=& {W}_{22}\bs{{h}}_{2}-\widetilde{W}_{22}\bs{{h}}_{2}+\widetilde{W}_{22}\bs{{h}}_{2}-\widetilde{W}_{22}\bs{\widetilde{h}}_{2} \\
\nonumber
 & =&\frac{1}{v_{2n,2n}}H_{1}(-1)^{1{(i>n-r)}}(\sum_{i=r+1}^{n}h_{i}-\sum_{i=n+1}^{2n-1}h_{i})\mathbf{1}_{2n-r-1}+H_{2}\bs{{h}}_{2}+
  \widetilde{W}_{22}(\bs{{h}}_{2}-\bs{\widetilde{h}}_{2}).
 \end{eqnarray}
Therefore, we have
\begin{align*}
  |{W}_{22}\bs{{h}}_{2}-\widetilde{W}_{22}\bs{{h}}_{2}| & \lesssim \frac{b_{n}b_{n}^{6}\log n}{n}\cdot(2n-r-1)\cdot\frac{b_{n}^{4}}{n^{2}c_{n}^{3}}+\frac{ b_n^6 }{ n^3 c_n^5 } \cdot b_n^6 \log n\\
   &\lesssim\frac{b_{n}^{11}\log n}{n^{2}c_{n}^{3}}.
\end{align*}
By \eqref{hi}, we have
\[
|\widetilde{W}_{22}(\bs{{h}}_{2}-\bs{\widetilde{h}}_{2})|\lesssim\frac{ b_n^6 }{ n^2 c_n^2 }\cdot (2n-r)\cdot\frac{ b_n^{12} (\log n)^{3/2} }{ n^{1/2}c_n^{10}}\lesssim \frac{ b_n^{15} (\log n)^{3/2} }{ n^{3/2}c_n^{12}}.
\]
It follows that
\begin{eqnarray}
|{W}_{22}\bs{{h}}_{2}-\widetilde{W}_{22}\bs{\widetilde{h}}_{2}|
\label{T4}
 & \lesssim & \frac{ b_n^{15} (\log n)^{3/2} }{ n^{3/2}c_n^{12}}.
 \end{eqnarray}
Combining \eqref{T1}, \eqref{T2}, \eqref{TT3} and \eqref{T4} yields
\[
\max_{i=r+1, \ldots, 2n-1} | \gamma_i - \gamma_i^0 | \lesssim \frac{ b_n^{15} (\log n)^{3/2} }{ n^{3/2}c_n^{12}}.
\]

\subsection{The proof of \eqref{ineq-B2-B20} for bounding $B_2-B_2^0$}
\label{subsection:B2B20}
In this section, we present the proof of the error bound of $B_2-B_2^0$ in \eqref{ineq-B2-B20} in the main text,
reproduced below:
\begin{equation}\label{ineq-B2-B200}
  B_2-B_2^0 = O_p\left(  \frac{ b_n^{21} (\log n)^{5/2}}{ n^{1/2}c_{n}^{17}}\right).
\end{equation}
\begin{proof}
Because $r$ is a fixed constant, $B_2 - B_2^0$ can be written as
\begin{eqnarray}
\nonumber
B_2^0 - B_2 & = &  \underbrace{ \sum_{i=r+1}^n  \{ (\widehat{\alpha}_i-\alpha_i)^3 - (\widehat{\alpha}_i^0-\alpha_i)^3 \} \sum_{j\neq i} \mu^{\prime\prime}( \pi_{ij}) }_{T_1}\\
\nonumber
&&+ \underbrace{ \sum_{i=r+1,i\neq j}^n\sum_{j}  \{ (\widehat{\beta}_j-\beta_j)^3 - (\widehat{\beta}_j^0-\beta_j)^3 \} \mu^{\prime\prime}( \pi_{ij} ) }_{T_2} \\
\nonumber
&& + \underbrace{ \sum_{i=r+1}^n \sum_{ j\neq i}  \left\{(\widehat{\alpha}_i-\alpha_i)^2(\widehat{\beta}_j-\beta_j)-(\widehat{\alpha}_i^0-\alpha_i)^2
(\widehat{\beta}_j^0-\beta_j)\right\} \mu^{\prime\prime}( \pi_{ij} ) }_{T_3}\\
\nonumber
&&+ \underbrace{ \sum_{i=r+1}^n \sum_{ j\neq i}  \left\{(\widehat{\alpha}_i-\alpha_i)(\widehat{\beta}_j-\beta_j)^2-(\widehat{\alpha}_i^0-\alpha_i)
(\widehat{\beta}_j^0-\beta_j)^2\right\} \mu^{\prime\prime}( \pi_{ij} ) }_{T_4}+o_{p}(1).
\end{eqnarray}
The terms $T_{1}$, $T_{2}$, $T_{3}$ and $T_{4}$ are bounded as follows. By \eqref{eq-a-zcc} and Lemma \ref{lemma-beta-homo-expan}, for $i=r+1, \ldots, n$, we have
\begin{eqnarray*}
& & | (\widehat{\alpha}_i-\alpha_i)^3 - (\widehat{\alpha}_i^0 -\alpha_i)^3 | \\
& \le & 2 |\gamma_i-\gamma_i^0| | (\widehat{\alpha}_i-\alpha_i)^2 + (\widehat{\alpha}_i^0 -\alpha_i)^2 | \\
& \lesssim &  \frac{ b_n^{15} (\log n)^{3/2}}{ n^{3/2}c_{n}^{12}}\cdot\left(\frac{ b_n^{3}}{c_{n}^{2}}\sqrt{\frac{\log n}{n}}  \right)^{2}\lesssim \frac{ b_n^{21} (\log n)^{5/2}}{ n^{5/2}c_{n}^{16}},
\end{eqnarray*}
where $\gamma_{i}$ and $\gamma_i^0$ are defined in Lemma \ref{lemma-gg0}. It follows that
\[
T_1 \lesssim \frac{ b_n^{21} (\log n)^{5/2}}{ n^{1/2}c_{n}^{17}}.
\]
Similarly, we have
\[
T_2 \lesssim \frac{ b_n^{21} (\log n)^{5/2}}{ n^{1/2}c_{n}^{17}}.
\]
Again, by \eqref{eq-a-zcc} and Lemmas \ref{lemma-beta-homo-expan}, \ref{lemma-gg0}. For $i=r+1, \ldots, n$, $j=1, \ldots, n-1$, we have
\begin{eqnarray*}
&&|(\widehat{\alpha}_i-\alpha_i)^2(\widehat{\beta}_j-\beta_j) -  (\widehat{\alpha}_i^0-\alpha_i)^2(\widehat{\beta}_j^0-\beta_j)| \\
& \le & |(\widehat{\alpha}_i-\alpha_i)^2(\widehat{\beta}_j-\beta_j) -  (\widehat{\alpha}_i-\alpha_i)^2(\widehat{\beta}_j^0-\beta_j)|
\\
& &+ |(\widehat{\alpha}_i-\alpha_i)^2(\widehat{\beta}_j^0-\beta_j) -  (\widehat{\alpha}_i^0-\alpha_i)^2(\widehat{\beta}_j^0-\beta_j)| \\
& \le & |\gamma_{n+j}-\gamma^{0}_{n+j}|\cdot |(\widehat{\alpha}_i-\alpha_i)^2|+|\widehat{\beta}_j^0-\beta_j| \cdot |\gamma_{i}-\gamma^{0}_{i}|\cdot(|\widehat{\alpha}_i-\alpha_i|+|\widehat{\alpha}_i^0-\alpha_i|)\\
& \lesssim &  \frac{ b_n^{15} (\log n)^{3/2}}{ n^{3/2}c_{n}^{12}}\cdot\left(\frac{ b_n^{3}}{c_{n}^{2}}\sqrt{\frac{\log n}{n}}  \right)^{2}\lesssim \frac{ b_n^{21} (\log n)^{5/2}}{ n^{5/2}c_{n}^{16}}.
\end{eqnarray*}
It follows that
\[
T_3 \lesssim \frac{ b_n^{21} (\log n)^{5/2}}{ n^{1/2}c_{n}^{17}}.
\]
Similarly, we have
\[
T_4 \lesssim \frac{ b_n^{21} (\log n)^{5/2}}{ n^{1/2}c_{n}^{17}}.
\]
By combining inequalities for $T_i, i=1,\ldots, 4$, it yields \eqref{ineq-B2-B200}.
\end{proof}

\subsection{The upper bound of $B_3 - B_3^0$ in the proof of Theorem 1 (b)}
\label{subsection-B3B30}

In this section, we present the proof of the error bound of $B_3-B_3^0$ in \eqref{ineq-B3-B20} in the main text,
reproduced below:
\begin{equation}\label{ineq-2a-B3300}
|B_3 - B_3^0| \lesssim \frac{b_n^{18} (\log n)^{5/2} }{ n^{1/2}c_n^{14}}.
\end{equation}

\begin{proof}
Recall that $B_{3}$ in \eqref{lrt-a-beta-B3} is
\begin{eqnarray*}
 -B_3 & = &  \sum_{i}  (\widehat{\alpha}_i-\alpha_i)^4 \sum_{j\neq i} \mu^{\prime\prime\prime}( \bar{\pi}_{ij} )
 +\sum_{j}  (\widehat{\beta}_j-\beta_j)^4 \sum_{j\neq i} \mu^{\prime\prime\prime}( \bar{\pi}_{ij} )\\
\nonumber
&+&4\sum_{i}\sum_{j,j\neq i}  (\widehat{\alpha}_i-\alpha_i)(\widehat{\beta}_j-\beta_j)^3\mu^{\prime\prime\prime}( \bar{\pi}_{ij} )
+6\sum_{i}\sum_{j,j\neq i}  (\widehat{\alpha}_i-\alpha_i)^2(\widehat{\beta}_j-\beta_j)^2\mu^{\prime\prime\prime}( \bar{\pi}_{ij}) \\
\nonumber
&+&4\sum_{i}\sum_{j, j\neq i}  (\widehat{\alpha}_i-\alpha_i)^3(\widehat{\beta}_j-\beta_j)\mu^{\prime\prime\prime}( \bar{\pi}_{ij} ),
\end{eqnarray*}
where $\bar{\pi}_{ij}$ lies between $\widehat{\pi}_{ij}$ and ${\pi}_{ij}$. Because $r$ is a fixed constant, in view of Lemma \ref{lemma-consi-beta}, we have
\begin{eqnarray*}
 -B_3 & = &  \sum_{i=r+1}^{n}  (\widehat{\alpha}_i-\alpha_i)^4 \sum_{j\neq i} \mu^{\prime\prime\prime}( \bar{\pi}_{ij} )
 +\sum_{j}  (\widehat{\beta}_j-\beta_j)^4 \sum_{i\neq j,i=r+1 }^{n} \mu^{\prime\prime\prime}( \bar{\pi}_{ij} )\\
\nonumber
&+&4\sum_{i=r+1}^{n}\sum_{j,j\neq i}  (\widehat{\alpha}_i-\alpha_i)(\widehat{\beta}_j-\beta_j)^3\mu^{\prime\prime\prime}( \bar{\pi}_{ij} )
+6\sum_{i=r+1}^{n}\sum_{j,j\neq i}  (\widehat{\alpha}_i-\alpha_i)^2(\widehat{\beta}_j-\beta_j)^2\mu^{\prime\prime\prime}( \bar{\pi}_{ij}) \\
\nonumber
&+&4\sum_{i=r+1}^{n}\sum_{j, j\neq i}  (\widehat{\alpha}_i-\alpha_i)^3(\widehat{\beta}_j-\beta_j)\mu^{\prime\prime\prime}( \bar{\pi}_{ij} )+O_{p}\left( \frac{b_{n}^{12}(\log n)^2}{nc_{n}^{8}}\right).
\end{eqnarray*}
Similar to \eqref{lrt-a-beta-B3}, we have
\begin{eqnarray*}
 -B_3^{0} & = &  \sum_{i=r+1}^{n}  (\widehat{\alpha}_i^{0}-\alpha_i)^4 \sum_{j\neq i} \mu^{\prime\prime\prime}( \bar{\pi}_{ij}^{0} )
 +\sum_{j}  (\widehat{\beta}_j^{0}-\beta_j)^4 \sum_{i\neq j,i=r+1 }^{n} \mu^{\prime\prime\prime}( \bar{\pi}_{ij}^{0} )\\
\nonumber
&+&4\sum_{i=r+1}^{n}\sum_{j,j\neq i}  (\widehat{\alpha}_i^{0}-\alpha_i)(\widehat{\beta}_j^{0}-\beta_j)^3\mu^{\prime\prime\prime}( \bar{\pi}_{ij}^{0} )
+6\sum_{i=r+1}^{n}\sum_{j,j\neq i}  (\widehat{\alpha}_i^{0}-\alpha_i)^2(\widehat{\beta}_j^{0}-\beta_j)^2\mu^{\prime\prime\prime}( \bar{\pi}_{ij}^{0}) \\
\nonumber
&+&4\sum_{i=r+1}^{n}\sum_{j, j\neq i}  (\widehat{\alpha}_i^{0}-\alpha_i)^3(\widehat{\beta}_j^{0}-\beta_j)\mu^{\prime\prime\prime}( \bar{\pi}_{ij}^{0} ),
\end{eqnarray*}
where $\bar{\pi}_{ij}^{0}$ lies between $\widehat{\pi}_{ij}^{0}$ and ${\pi}_{ij}$.  To show \eqref{ineq-B3-B20}, it is sufficient to bound the following five terms:
\begin{eqnarray*}
  C_{1}:&=& \sum_{i=r+1}^{n}\sum_{j,j\neq i} \left\{ (\widehat{\alpha}_i-\alpha_i)^4  \mu^{\prime\prime\prime}( \bar{\pi}_{ij} )-  (\widehat{\alpha}_i^{0}-\alpha_i)^4 \mu^{\prime\prime\prime}( \bar{\pi}_{ij}^{0} )\right\}\\
  C_{2} :&=&\sum_{j} \sum_{i\neq j,i=r+1 }^{n}\left\{ (\widehat{\beta}_j-\beta_j)^4 \mu^{\prime\prime\prime}( \bar{\pi}_{ij} )- (\widehat{\beta}_j^{0}-\beta_j)^4  \mu^{\prime\prime\prime}( \bar{\pi}_{ij}^{0} ) \right\}\\
  C_{3} :&=&  \sum_{i=r+1}^{n}\sum_{j,j\neq i}  \left\{(\widehat{\alpha}_i-\alpha_i)(\widehat{\beta}_j-\beta_j)^3\mu^{\prime\prime\prime}( \bar{\pi}_{ij} )- (\widehat{\alpha}_i^{0}-\alpha_i)(\widehat{\beta}_j^{0}-\beta_j)^3\mu^{\prime\prime\prime}( \bar{\pi}_{ij}^{0} )\right\}\\
  C_{4} :&=& \sum_{i=r+1}^{n}\sum_{j,j\neq i} \left\{ (\widehat{\alpha}_i-\alpha_i)^2(\widehat{\beta}_j-\beta_j)^2\mu^{\prime\prime\prime}( \bar{\pi}_{ij})-  (\widehat{\alpha}_i^{0}-\alpha_i)^2(\widehat{\beta}_j^{0}-\beta_j)^2\mu^{\prime\prime\prime}( \bar{\pi}_{ij}^{0})\right\} \\
\nonumber
  C_{5}:&=& \sum_{i=r+1}^{n}\sum_{j, j\neq i} \left\{ (\widehat{\alpha}_i-\alpha_i)^3(\widehat{\beta}_j-\beta_j)\mu^{\prime\prime\prime}( \bar{\pi}_{ij} )-  (\widehat{\alpha}_i^{0}-\alpha_i)^3(\widehat{\beta}_j^{0}-\beta_j)\mu^{\prime\prime\prime}( \bar{\pi}_{ij}^{0} )\right\}
\end{eqnarray*}
Before bounding $C_1$, $C_2$, $C_3$, $C_4$ and $C_5$,
we drive one useful inequality. By finding the fourth derivative of $\mu(x)$ with respect to $x$, we have
\begin{eqnarray*}
\mu^{\prime\prime\prime\prime}(x) & = & \frac{ e^x( 1- 8 e^x + 3e^{2x} ) }{ ( 1 + e^x )^4 } - \frac{ 4 e^{2x} ( 1 - 4 e^x + 4e^{2x} )}{ ( 1+ e^x)^5 } \\
& = & \frac{ e^x ( 1 -11e^x + 11 e^{2x} - e^{3x})  }{ ( 1 + e^x )^5 } \\
& = & \frac{ e^x }{ ( 1+e^x)^2 } \cdot \frac{ 1 -11e^x + 11 e^{2x} - e^{3x} }{ (1+e^x)^3 }.
\end{eqnarray*}
It is easy to see that
\[
\frac{ 11 }{ 3} (1+e^x)^3 \ge 1  + 11e^x + 11e^{2x} + e^{3x} \ge | 1 -11e^x + 11 e^{2x} - e^{3x} |.
\]
It follows that
\[
|\mu^{\prime\prime\prime\prime}(x)| \le \frac{ 11e^x }{ 3( 1+e^x)^2 }.
\]
Therefore, for any $\dot{\pi}_{ij}$ satisfying $|\dot{\pi}_{ij} - \pi_{ij}|\to 0$, we have
\begin{equation}\label{eq-fourth-mu}
| \mu^{\prime\prime\prime\prime}(\dot{\pi}_{ij} )| \le \frac{ 11\mu^{\prime}( \dot{\pi}_{ij} ) }{ 3 } \lesssim \mu^{\prime}( \pi_{ij} ) \lesssim \frac{1}{c_n}.
\end{equation}
It follows from the mean value theorem that for any $\dot{\pi}_{ij}$ satisfying $|\dot{\pi}_{ij} - \pi_{ij}|\to 0$,
\begin{equation}\label{ineq-mu-thr-diff}
| \mu^{\prime\prime\prime}( \dot{\pi}_{ij} ) - \mu^{\prime\prime\prime}( \pi_{ij} ) | \lesssim \frac{1}{c_n} |\dot{\pi}_{ij}- \pi_{ij} |.
\end{equation}

By Lemmas \ref{lemma-consi-beta}, \ref{lemma-con-beta-b} and \ref{lemma-hat-beta-diff-2b}, for $i=r+1, \ldots, n$,  we have
\begin{eqnarray}
\nonumber
&&\left|(\widehat{\alpha}_i -\alpha_i)^4\mu^{\prime\prime\prime}( \bar{\pi}_{ij} )-
(\widehat{\alpha}_i^0 -\alpha_i)^4\mu^{\prime\prime\prime}( \bar{\pi}_{ij}^0) \right| \\
\nonumber
& \le & \left| (\widehat{\alpha}_i -\alpha_i)^4 \left\{\mu^{\prime\prime\prime}( \bar{\pi}_{ij} )
-  \mu^{\prime\prime\prime}( \bar{\pi}_{ij}^0 )\right\} \right| + \left| \left\{(\widehat{\alpha}_i -\alpha_i)^4
- (\widehat{\alpha}_i^0 -\alpha_i)^4 \right\} \mu^{\prime\prime\prime}( \bar{\pi}_{ij}^0 ) \right| \\
\nonumber
& \lesssim & \frac{1}{c_n} \left( | \widehat{\alpha}_i - \alpha_i |^4 \cdot ( | \widehat{\alpha}_i - \alpha_i | + |\widehat{\beta}_j - \beta_j | )
+ | \widehat{\alpha}_i - \widehat{\alpha}_i^0|^2  \cdot ( | \widehat{\alpha}_i - \alpha_i |^2 + | \widehat{\alpha}_i^0 - \alpha_i|^2 ) \right)\\
\nonumber
& \lesssim & \frac{ 1 }{ c_n }\cdot \left( \frac{b_n^{3}}{c_n^{2}} \sqrt{\frac{\log n}{n}} \right)^5 + \frac{1}{c_n}\cdot\left(\frac{ b_n^9 \log n }{ n c_n^{7} }\right)^2 \cdot \left(  \frac{b_n^{3}}{c_n^{2}} \sqrt{\frac{\log n}{n}} \right)^2 \\
\label{ineq-fourth-C1}
& \lesssim &  \frac{b_n^{15}(\log n)^{5/2} }{ n^{5/2}c_n^{11}} + \frac{ b_n^{24} (\log n)^3 }{ n^3 c_n^{19}},
\end{eqnarray}
where the second inequality is due to \eqref{ineq-mu-thr-diff}.
Similarly,
\begin{eqnarray}\label{ineq-fourth-C2}
\left|(\widehat{\beta}_j -\beta_j)^4\mu^{\prime\prime\prime}( \bar{\pi}_{ij} )-
(\widehat{\beta}_j^0 -\beta_j)^4\mu^{\prime\prime\prime}( \bar{\pi}_{ij}^0) \right|
 \lesssim   \frac{b_n^{15}(\log n)^{5/2} }{ n^{5/2}c_n^{11}} + \frac{ b_n^{24} (\log n)^3 }{ n^3 c_n^{19}}.
\end{eqnarray}
For $i=r+1, \ldots, n$, $j=1, \ldots, n-1$, $i\neq j$, we have
\begin{eqnarray}
\nonumber
&&\left| (\widehat{\alpha}_i -\alpha_i)(\widehat{\beta}_j -\beta_j)^3\mu^{\prime\prime\prime}( \bar{\pi}_{ij} )
 -(\widehat{\alpha}_i^0 -\alpha_i)(\widehat{\beta}_j^0 -\beta_j)^3\mu^{\prime\prime\prime}( \bar{\pi}_{ij}^0 ) \right| \\
\nonumber
& \le & \underbrace{\left| (\widehat{\alpha}_i -\alpha_i)(\widehat{\beta}_j -\beta_j)^3\mu^{\prime\prime\prime}( \bar{\pi}_{ij} )
-  (\widehat{\alpha}_i^0 -\alpha_i)(\widehat{\beta}_j^0 -\beta_j)^3\mu^{\prime\prime\prime}( \bar{\pi}_{ij} ) \right|}_{ E_1} \\
\nonumber
&&+  \left|(\widehat{\alpha}_i^0 -\alpha_i)(\widehat{\beta}_j^0 -\beta_j)^3\mu^{\prime\prime\prime}( \bar{\pi}_{ij} )
 -(\widehat{\alpha}_i^0 -\alpha_i)(\widehat{\beta}_j^0 -\beta_j)^3\mu^{\prime\prime\prime}( \bar{\pi}_{ij}^0 ) \right|\\
\label{ineq-fourth-C3}
 & \lesssim & \frac{ 1 }{ c_n }\cdot \left( \frac{b_n^{3}}{c_n^{2}} \sqrt{\frac{\log n}{n}} \right)^5 + \frac{1}{c_n}\cdot\left(\frac{ b_n^9 \log n }{ n c_n^{7} }\right) \cdot \left(  \frac{b_n^{3}}{c_n^{2}} \sqrt{\frac{\log n}{n}} \right)^3,
\end{eqnarray}
where the second inequality for $E_1$ follows from
\begin{eqnarray*}
&&|(\widehat{\alpha}_i -\alpha_i)(\widehat{\beta}_j -\beta_j)^3 - (\widehat{\alpha}_i^0 -\alpha_i)(\widehat{\beta}_j^0 -\beta_j)^3| \\
& \le & |(\widehat{\alpha}_i -\alpha_i)(\widehat{\beta}_j -\beta_j)^3 - (\widehat{\alpha}_i^0 -\alpha_i)(\widehat{\beta}_j -\beta_j)^3| \\
&& + |(\widehat{\alpha}_i^0 -\alpha_i)(\widehat{\beta}_j -\beta_j)^3 - (\widehat{\alpha}_i^0 -\alpha_i)(\widehat{\beta}_j^0 -\beta_j)^3| \\
& \lesssim & | \widehat{\beta}_j - \widehat{\beta}_j^0 |\{ (\widehat{\beta}_j -\beta_j)^2 + (\widehat{\beta}_j^0 -\beta_j)^2 \} | \widehat{\alpha}_i^{0} -\alpha_i | \\
&& + |(\widehat{\beta}_j -\beta_j)^3|\cdot |\widehat{\alpha}_i - \widehat{\alpha}_i^0| \\
& \lesssim & \left(\frac{ b_n^9 \log n }{ n c_n^{7} }\right) \cdot \left(  \frac{b_n^{3}}{c_n^{2}} \sqrt{\frac{\log n}{n}} \right)^3.
\end{eqnarray*}
Similarly,
\begin{eqnarray}
\nonumber
&&\left| (\widehat{\alpha}_i -\alpha_i)^3(\widehat{\beta}_j -\beta_j)\mu^{\prime\prime\prime}( \bar{\pi}_{ij} )
 -(\widehat{\alpha}_i^0 -\alpha_i)^3(\widehat{\beta}_j^0 -\beta_j)\mu^{\prime\prime\prime}( \bar{\pi}_{ij}^0 ) \right| \\
\label{ineq-fourth-C5}
 & \lesssim & \frac{ 1 }{ c_n }\cdot \left( \frac{b_n^{3}}{c_n^{2}} \sqrt{\frac{\log n}{n}} \right)^5 + \frac{1}{c_n}\cdot\left(\frac{ b_n^9 \log n }{ n c_n^{7} }\right) \cdot \left(  \frac{b_n^{3}}{c_n^{2}} \sqrt{\frac{\log n}{n}} \right)^3.
\end{eqnarray}
Again, for $i\neq j, i=r+1,\ldots, n, j=1, \ldots, n-1$, we have
\begin{eqnarray}
\nonumber
&&  (\widehat{\alpha}_i^0-\alpha_i)^2(\widehat{\beta}_j^0-\beta_j)^2\mu^{\prime\prime\prime}( \bar{\pi}_{ij}^0)
-(\widehat{\alpha}_i - \alpha_i)^2(\widehat{\beta}_j-\beta_j)^2\mu^{\prime\prime\prime}( \bar{\pi}_{ij})  \\
\nonumber
& \le & \underbrace{| (\widehat{\alpha}_i^0-\alpha_i)^2(\widehat{\beta}_j^0-\beta_j)^2\mu^{\prime\prime\prime}( \bar{\pi}_{ij}^0)
-(\widehat{\alpha}_i - \alpha_i)^2(\widehat{\beta}_j-\beta_j)^2\mu^{\prime\prime\prime}( \bar{\pi}_{ij}^0) |}_{E_2} \\
\nonumber
&& + |(\widehat{\alpha}_i-\alpha_i)^2(\widehat{\beta}_j-\beta_j)^2\mu^{\prime\prime\prime}( \bar{\pi}_{ij}^0)
-(\widehat{\alpha}_i - \alpha_i)^2(\widehat{\beta}_j-\beta_j)^2\mu^{\prime\prime\prime}( \bar{\pi}_{ij})| \\
\label{ineq-fourth-C4}
& \lesssim & \frac{ 1 }{ c_n }\cdot \left( \frac{b_n^{3}}{c_n^{2}} \sqrt{\frac{\log n}{n}} \right)^5 + \frac{1}{c_n}\cdot\left(\frac{ b_n^9 \log n }{ n c_n^{7} }\right) \cdot \left(  \frac{b_n^{3}}{c_n^{2}} \sqrt{\frac{\log n}{n}} \right)^3,
\end{eqnarray}
where the inequality for $E_2$ follows from
\begin{eqnarray*}
 & & | (\widehat{\alpha}_i^0-\alpha_i)^2(\widehat{\beta}_j^0-\beta_j)^2 - (\widehat{\alpha}_i - \alpha_i)^2(\widehat{\beta}_j-\beta_j)^2 | \\
& \le & | (\widehat{\alpha}_i^0-\alpha_i)^2(\widehat{\beta}_j^0-\beta_j)^2 - (\widehat{\alpha}_i - \alpha_i)^2(\widehat{\beta}_j^0-\beta_j)^2 | \\
& & + |(\widehat{\alpha}_i-\alpha_i)^2(\widehat{\beta}_j^0-\beta_j)^2 - (\widehat{\alpha}_i - \alpha_i)^2(\widehat{\beta}_j-\beta_j)^2 | \\
& \lesssim & | \widehat{\beta}_j^0-\beta_j |^2 | \widehat{\alpha}_i -\widehat{\alpha}_i^0 | ( |\widehat{\alpha}_i^0-\alpha_i| + | \widehat{\alpha}_i-\alpha_i | ) \\
&& + (\widehat{\alpha}_i-\alpha_i)^2 |\widehat{\beta}_j^0 - \widehat{\beta}_j| (|\widehat{\beta}_j^0-\beta_j|+ |\widehat{\beta}_j-\beta_j | ) \\
& \lesssim & \left(\frac{ b_n^9 \log n }{ n c_n^{7} }\right) \cdot \left(  \frac{b_n^{3}}{c_n^{2}} \sqrt{\frac{\log n}{n}} \right)^3.
\end{eqnarray*}
By \eqref{ineq-fourth-C1}, we have
\[
|C_1| \lesssim (n-r)\cdot n\cdot   \left( \frac{b_n^{15}(\log n)^{5/2} }{ n^{5/2}c_n^{11}} + \frac{ b_n^{24} (\log n)^3 }{ n^3 c_n^{19}} \right)  \lesssim \frac{b_n^{15}(\log n)^{5/2} }{ n^{1/2}c_n^{11}}.
\]
By \eqref{ineq-fourth-C2}, we have
\[
|C_2|\lesssim (n-r)\cdot n\cdot \left( \frac{b_n^{15}(\log n)^{5/2} }{ n^{5/2}c_n^{11}} + \frac{ b_n^{24} (\log n)^3 }{ n^3 c_n^{19}} \right)  \lesssim \frac{b_n^{15}(\log n)^{5/2} }{ n^{1/2}c_n^{11}}.
\]
By \eqref{ineq-fourth-C3}, we have
\[
|C_3| \lesssim (n-r)\cdot n\cdot  \left( \frac{b_n^{15}(\log n)^{5/2} }{ n^{5/2}c_n^{11}} + \frac{ b_n^{18} (\log n)^{5/2} }{ n^{5/2} c_n^{14}} \right)  \lesssim \frac{b_n^{18}(\log n)^{5/2} }{ n^{1/2}c_n^{14}}.
\]
By \eqref{ineq-fourth-C5}, we have
\[
|C_5| \lesssim (n-r)\cdot n\cdot   \left( \frac{b_n^{15}(\log n)^{5/2} }{ n^{5/2}c_n^{11}} + \frac{ b_n^{18} (\log n)^{5/2} }{ n^{5/2} c_n^{14}} \right)  \lesssim \frac{b_n^{18}(\log n)^{5/2} }{ n^{1/2}c_n^{14}}.
\]
By \eqref{ineq-fourth-C4}, we have
\[
|C_4| \lesssim (n-r)\cdot n\cdot   \left( \frac{b_n^{15}(\log n)^{5/2} }{ n^{5/2}c_n^{11}} + \frac{ b_n^{18} (\log n)^{5/2} }{ n^{5/2} c_n^{14}} \right)  \lesssim \frac{b_n^{18}(\log n)^{5/2} }{ n^{1/2}c_n^{14}}.
\]
By combining the above five inequalities, it yields \eqref{ineq-2a-B3300}.
This completes the proof.
\end{proof}

\section{Proofs of supported lemmas in the proof of Theorem 2} 
\label{section:beta-th1a}
This section is organized as follows.
Sections \ref{subsection-proof-lemma2}, \ref{subsection-proof-lemma4} and \ref{subsection-proof-lemma5} present
the proofs of Lemmas \ref{lemma-clt-beta-W}, \ref{lemma-consi-beta} and \ref{lemma:beta3:err}, respectively.

\subsection{Proof of Lemma \ref{lemma-clt-beta-W}}
\label{subsection-proof-lemma2}
\begin{proof}
We first have
\begin{equation}\label{wd-expectation}
\E[ \bs{\bar{\mathbf{g}}}_2^\top \widetilde{W}_{22} \bs{\bar{\mathbf{g}}}_2 ] =0,
\end{equation}
which is due to that
\[
\E[  \mathbf{\bar{g}_2}^\top \widetilde{W}_{22}  \mathbf{\bar{g}_2} ]
= \mathrm{tr} (\E[  \mathbf{\bar{g}_2}^\top \mathbf{\bar{g}_2}  ] \widetilde{W}_{22} )
= \mathrm{tr} (V_{22}\widetilde{W}_{22}) = \mathrm{tr} ( I_{2n-1-r} - V_{22}S_{22} ) = 0.
\]

Let $\widetilde{W}_{22}=(\tilde{w}_{ij})_{(2n-1-r)\times (2n-1-r) }$.
Next, we bound the variance of $\sum_{i,j=r+1}^{2n-1} \bar{g}_i \tilde{w}_{(i-r)(j-r)} \bar{g}_j$.
Recall that $v_{i,j}=\mathrm{Var}(\bar{a}_{i,j})=\mu^\prime(\alpha_i + \beta_j)$.
There are four cases for calculating the covariance
\[
g_{ij\zeta\eta}=\mathrm{Cov}\big( \bar{g}_i  \tilde{w}_{(i-r)(j-r)} \bar{g}_j, \bar{g}_\zeta \tilde{w}_{(\zeta-r)(\eta-r)} \bar{g}_\eta ).
\]
Case 1: $i=j=\zeta=\eta$. For $i=r+1,\ldots,n,$
\begin{eqnarray*}
\mathrm{Var}( \bar{g}_i^{\,2}) = \mathrm{Cov} ( (\sum_{\alpha=1 }^n \bar{a}_{i,\alpha})^{\,2}, (\sum_{h=1 }^n \bar{a}_{i,h} )^{\,2} )
= \mathrm{Cov}( \sum_{\alpha=1}^n \sum_{\beta=1}^n \bar{a}_{i,\alpha} \bar{a}_{i,\beta}, \sum_{h=1}^n \sum_{g=1}^n \bar{a}_{i,h} \bar{a}_{i,g}).
\end{eqnarray*}
Note that the random variables $\bar{a}_{i,j}$ for $1\le i \neq j \le n$ are mutually independent.
There are only two cases in terms of $(\alpha, \beta, h, g)$ for which
$\mathrm{Cov}( \bar{a}_{i,\alpha} \bar{a}_{i,\beta}, \bar{a}_{i,h} \bar{a}_{i,g} )$ is not equal to zero:
(Case A) $\alpha=\beta=h=g \neq i$;
(Case B) $\alpha=h, \beta=g$ or $\alpha=g, \beta=h$.
By respectively considering Case A and Case B,  a direct calculation gives that
\begin{equation}\label{eq:variance:2}
\mathrm{Var}( \bar{g}_i^{\,2} )= 2 v_{i,i}^2 + \sum_{j} \mathrm{Var}( \bar{a}_{i,j}^2 ).
\end{equation}
Similarly, for $i=n+1,\ldots,2n-1$, we have
\begin{equation}\label{equ:varii}
\mathrm{Var}( \bar{g}_i^{\,2} ) = \sum_{i}\mathrm{Var}( \bar{a}_{i,j}^2 ) + 2 v_{i,i}^2.
\end{equation}
Let $p_{ij}=\mu(\alpha_i+\beta_j)$ and $q_{ij}=1-p_{ij}$. By \eqref{ineq-mu-deriv-bound},
\[
\max_{i,j} p_{ij}(1-p_{ij}) \le \frac{1}{c_n}.
\]
Note that
\[
\mathrm{Var}( \bar{a}_{i,j}^{\,2} ) = p_{ij}q_{ij}\{ p_{ij}^3 + q_{ij}^3 - p_{ij}q_{ij} \}
= p_{ij}q_{ij}
\le \frac{1}{c_n}.
\]
Thus, we have
\begin{equation*}
|g_{iiii}|\le \tilde{w}_{(i-r)(i-r)}^2\cdot \left( \frac{2(n-1)^2}{c_n^2} + \frac{n-1}{c_n} \right).
\end{equation*}
Case 2: Three indices among the four indices are the same. Without loss of generality, we assume that
$j=\zeta=\eta$ and $i\neq j$. For $i=r+1,\ldots,n$,
we have
\[
\mathrm{Cov}( \bar{g}_i \bar{g}_j, \bar{g}_j^{\,2}) = \sum_{k,h,\alpha,\gamma} ( \E \bar{a}_{i,k}\bar{a}_{j,h} \bar{a}_{j,\alpha}
\bar{a}_{j,\gamma} - \E \bar{a}_{i,k}\bar{a}_{j,h} \E \bar{a}_{j,\alpha} \bar{a}_{j,\gamma} )
\]
and, for distinct $k,h,\alpha,\gamma$,
\begin{eqnarray*}
\E \bar{a}_{i,k}\bar{a}_{j,h} \bar{a}_{j,\alpha}
\bar{a}_{j,\gamma} & = & 0, \\
\E \bar{a}_{i,j}\bar{a}_{j,i} \bar{a}_{j,\alpha}
\bar{a}_{j,\gamma} & = & 0, \\
\E \bar{a}_{i,h}\bar{a}_{j,h} \bar{a}_{j,\alpha}
\bar{a}_{j,\gamma} & = & 0.
\end{eqnarray*}
It follows that
\[
\mathrm{Cov}( \bar{g}_i \bar{g}_j, \bar{g}_j^{\,2}) = \E \bar{a}_{i,j}\bar{a}_{j,i}^3 - \E \bar{a}_{i,j}\bar{a}_{j,i}\E \bar{a}_{j,i}^2
+ 2\sum_{h\neq i,j} \E \bar{a}_{i,j}\bar{a}_{j,i} \E \bar{a}_{j,h}^{\,2}.
\]
Therefore,  by \eqref{ineq-mu-deriv-bound},
\begin{eqnarray*}
|g_{ijjj}|&\le & |\tilde{w}_{(i-r)(j-r)}\tilde{w}_{(j-r)(j-r)}|\cdot \frac{n}{c_n^2}.
\end{eqnarray*}
For $i=n+1,\ldots,2n-1$, we also have $|g_{ijjj}|\le  |\tilde{w}_{(i-r)(j-r)}\tilde{w}_{(j-r)(j-r)}|\cdot \frac{n}{c_n^2}$.
Similarly, we have the upper bounds in other cases.\\
Case 3. Two indices among the four are the same (e.g. $i=j$ or $j=\zeta$):
\begin{eqnarray*}
|g_{ii\eta\zeta}|&=&|\tilde{w}_{(i-r)(i-r)}\tilde{w}_{(\zeta-r)(\eta-r)}(2v_{i\zeta}v_{i\eta}+v_{ii}v_{\zeta\eta})|
\le |\tilde{w}_{(i-r)(i-r)}\tilde{w}_{(\zeta-r)(\eta-r)}|\cdot \frac{n}{c_n^2};\\
|g_{ijj\eta}|&=&|\tilde{w}_{(i-r)(i-r)}\tilde{w}_{(j-r)(\eta-r)}(2v_{ji}v_{j\eta}+v_{ij}v_{j\eta})|
\le 3|\tilde{w}_{(i-r)(i-r)}\tilde{w}_{(j-r)(\eta-r)}|\cdot \frac{1}{c_n^2}.
\end{eqnarray*}
Case 4: All four indices are different
\begin{eqnarray*}
|g_{ij\zeta\eta}|&=& |\tilde{w}_{(i-r)(j-r)}\tilde{w}_{(\zeta-r)(\eta-r)}(v_{i\zeta}v_{j\eta}+v_{i\eta}v_{j\zeta})|\le 2|\tilde{w}_{(i-r)(j-r)}\tilde{w}_{(\zeta-r)(\eta-r)}|\frac{1}{c_n^2}.
\end{eqnarray*}
Consequently, by \eqref{ineq-V-S-appro-upper-b}, we have
\begin{align*}
\mathrm{Var} ( \bs{\bar{\mathbf{g}}}_2^{\top} \widetilde{W}_{22} \bs{\bar{\mathbf{g}}}_2 )
& =  \sum_{i,j,\zeta,\eta={r+1}}^{2n-1} \mathrm{Cov}\big( \bar{g}_i  \tilde{w}_{(i-r)(j-r)} \bar{g}_j,
\bar{g}_\zeta \tilde{w}_{(\zeta-r)(\eta-r)} \bar{g}_\eta ) \\
& \lesssim   \left( \frac{b_n^3}{n^2 c_n^2 } \right)^2 \times \left( (2n-1-r)\cdot \frac{ n^2 }{c_n^2 }
+ (2n-1-r)^2 \cdot \frac{n}{c_n^2} \right.\\
&\left.+ (2n-1-r)^3 \cdot \frac{n}{c_n^2} + (2n-1-r)^4 \cdot \frac{1}{c_n^2 } \right) \\
& \lesssim \frac{b_n^6(2-r/n)^3}{c_n^6}.
\end{align*}
It follows that from Chebyshev's inequality and \eqref{wd-expectation}, we have
\begin{align*}
 & \P( \left( |\bs{\bar{\mathbf{g}}}_2^\top \widetilde{W}_{22}  \bs{\bar{\mathbf{g}}}_2  |
\ge  \rho_n\frac{ b_n^3(2-r/n)^{3/2} }{c_n^3} \right) \\
\le & \frac{ c_n^6 }{ b_n^6(2-r/n)^3\rho_n^2 }\times \mathrm{Var} ( \mathbf{\bar{\mathbf{g}}_2}^{\top} W_{22} \mathbf{\bar{\mathbf{g}}_2} ) \\
\lesssim & \frac{1}{\rho_n^2} \to 0,
\end{align*}
where $\{\rho_n\}_{n=1}^\infty$ is any positive sequence tending to infinity.
This completes the proof.
\end{proof}

\subsection{Proof of Lemma \ref{lemma-consi-beta}}
\label{subsection-proof-lemma4}

Before beginning the proof of Lemma \ref{lemma-consi-beta}, we introduce one useful lemma.
Let $F(\mathbf{x}): \R^n \to \R^n$ be a function vector on $\mathbf{x}\in\R^n$.
We say that a Jacobian matrix $F^\prime(\mathbf{x})$ with $\mathbf{x}\in \R^n$ is Lipschitz continuous on a convex set $D\subset\R^n$ if
for any $\mathbf{x},\mathbf{y}\in D$, there exists a constant $\lambda>0$ such that
for any vector $\mathbf{v}\in \R^n$ the inequality
\begin{equation*}
\| [F^\prime (\mathbf{x})] \mathbf{v} - [F^\prime (\mathbf{y})] \mathbf{v} \|_\infty \le \lambda \| \mathbf{x} - \mathbf{y} \|_\infty \|\mathbf{v}\|_\infty
\end{equation*}
holds.
The following Newton's iterative theorem of \cite{Kantorovich-Akilov1964} to prove the existence and consistency of the MLE.

\begin{lemma}[Theorem 6 in \cite{Kantorovich-Akilov1964}]\label{lemma:Newton:Kantovorich}
Let $X$ and $Y$ be Banach spaces, $D$ be an open convex subset of $X$ and
$F:D \subseteq X \to Y$ be Fr\'{e}chet differentiable.
Assume that, at some $\mathbf{x}_0 \in D$, $F^\prime(\mathbf{x}_0)$ is invertible and that
\begin{eqnarray}
\label{eq-kantororich-a}
  \|(F'(\mathbf{x}_0))^{-1}(F'(\mathbf{x})-F'(\mathbf{y}))\| \le K\|\mathbf{x}-\mathbf{y}\|,\mathbf{x},\mathbf{y} \in D, \\
\label{eq-kantororich-b}
\|(F'(\mathbf{x}_0))^{-1}F(\mathbf{x}_0)\|\leq \eta, \quad h = K\eta\leq 1/2,\\
\nonumber
\bar{S}(\bs{x}_0, t^*) \subseteq D, t^*=2\eta/( 1+ \sqrt{ 1-2h}).
\end{eqnarray}
Then:
  \begin{enumerate}
        \item The Newton iterates $\mathbf{x}_{n+1} = \mathbf{x}_n-(F'(\mathbf{x}_n))^{-1}F(\mathbf{x}_n), n\geq 0$ are well-defined, lie in $\bar{S}(\mathbf{x}_0,t^*)$ and converge to a solution $\mathbf{x}^*$ of $F(\mathbf{x})=0$.
        \item The solution $\mathbf{x}^*$ is unique in $S(\mathbf{x}_0,t^{**})\cap {D}, t^{**}=(1+\sqrt{1-2h})/K$ if $2h<1$ and in $\bar{S}(\mathbf{x}_0,t^{**})$ if $2h=1$.
        \item $ \| \mathbf{x}^* - \mathbf{x}_0 \| \le t^*$ if $n=0$ and $ \| \mathbf{x}^* - \mathbf{x}_{n}\| = 2^{1-n} (2h)^{ 2^n -1} \eta $ if $n\ge 1$.
         \end{enumerate}
   \end{lemma}

\begin{proof}
Under the null $H_0: (\alpha_1, \ldots, \alpha_r)^\top=(\alpha_1^0, \ldots, \alpha_r^0)^\top$, $\alpha_1, \ldots, \alpha_r$ are known and $\alpha_{r+1}, \ldots, \alpha_n, \beta_{1}, \ldots, \beta_{n-1}$ are unknown.
Recall that $\bs{\widehat{\theta}}^0$ denotes the restricted MLE under the null space,
where $\widehat{\alpha}_i = \alpha_i^0$, $i=1,\ldots,r$.
For convenience, we will use $\bs{\theta}$ and $\bs{\widehat{\theta}}^0$ to denote
the vectors $(\alpha_{r+1}, \ldots, \alpha_n, \beta_{1}, \ldots, \beta_{n-1})^\top$ and $(\widehat{\alpha}_{r+1}^{0}, \ldots, \widehat{\alpha}_n^{0}, \widehat{\beta}_{1}^{0}, \ldots, \widehat{\beta}_{n-1}^{0})^\top$ in this proof, respectively.
Note that $\bs{\widehat{\theta}}^0=\bs{\widehat{\theta}}$ when $r=0$.

Define a system of score functions based on likelihood equations:
\begin{equation}
\label{eqn:def:F}
\begin{split}
F_i( \bs \theta )& =\sum_{k=1;k\neq i}^{n}\mu(\alpha_i + \beta_k)-d_{i},  ~~~  i=1, \ldots, n, \\
F_{n+j}( \bs \theta )& =  \sum_{k=1;k\neq j}^{n}\mu(\alpha_k + \beta_j)-b_{j},  ~~~  j=1, \ldots, n-1, \\
\end{split}
\end{equation}
and $F(\bs{\theta})=(F_{r+1}(\bs{\theta}), \ldots, F_{2n-1}(\bs{\theta}))^\top$.

We will derive the error bound between $\bs{\widehat{\theta}}^0$ and $\bs{\theta}$ through
obtaining the convergence rate of the Newton iterative sequence $\bs{\theta}^{(n+1)}= \bs{\theta}^{(n)} - [F^\prime (\bs{\theta}^{(n)})]^{-1}
F (\bs{\theta}^{(n)})$,
where we choose the true parameter $\bs{\theta}$ as the starting point $\bs{\theta}^{(0)}:=\bs{\theta}$.
To this end, it is sufficient to demonstrate the conditions in Lemma \ref{lemma:Newton:Kantovorich}.
The proof proceeds three steps. Step 1 is about the Lipschitz continuous property of the Jacobian matrix $F'( \bs{\theta} )$.
Step 2 is about the tail probability of $F(\bs{\theta})$. Step 3 is a combining step.

Step 1. We claim that
the Jacobian matrix $F'( \bs{\theta} )$ of $F(\bs{\theta})$ on $\bs{\theta}$ is Lipschitz continuous with the Lipschitz coefficient  $2(2n-r-1)/c_{n}$.
This is verified as follows.
Let $(\alpha_1, \ldots, \alpha_r)=(\alpha_1^0, \ldots, \alpha_r^0)$.
The Jacobian matrix $F^\prime(\bs{\theta})$ of $F(\bs{\theta})$ can be calculated as follows.
For $i\in \{r+1, \ldots, n\}$, we have
\[
    \begin{split}
    \cfrac{\partial F_{i}}{\partial \alpha_{l}}&=0,~~l=r+1,\ldots,n,l\neq i;~~
\cfrac{\partial F_{i}}{\partial \alpha_{i}}=\sum_{k=1;k\neq i}^{n}\mu'(\alpha_{i}+\beta_{k}),\\
\cfrac{\partial F_{i}}{\partial\beta_{l}}&=\mu'(\alpha_{i}+\beta_{l}),~~l=1,\ldots,n-1,l\neq i;~~\cfrac{\partial F_{i}}{\partial\beta_{i}}=0,\\
\end{split}
    \]
and for $j\in\{1, \ldots, n-1\}$,\\
\[
    \begin{split}
\cfrac{\partial F_{n+j}}{\partial \alpha_{l}}&=\mu'(\alpha_{l}+\beta_{j}),~~l=r+1,\ldots,n,l\neq j;
~~\cfrac{\partial F_{n+j}}{\partial \alpha_{j}}=0,\\
\cfrac{\partial F_{n+j}}{\partial\beta_{l}}&=0,~~l=1,\ldots,n-1,l\neq j;~~\cfrac{\partial F_{n+j}}{\partial\beta_{j}}=\sum_{k=1;k\neq j}^{n}\mu'(\alpha_{k}+\beta_{j}).
\end{split}
    \]
Let
\[
F_i'(\bs{\theta}) = (F_{i,r+1}'(\bs{\theta}), \ldots, F_{i,2n-1}'(\bs{\theta}))^{\top}:=
(\frac{\partial F_i}{\partial \alpha_{r+1} }, \ldots, \frac{\partial F_i}{\partial \alpha_{n} }, \frac{\partial F_i}{\partial \beta_1 }, \ldots,
\frac{\partial F_i}{\partial \beta_{n-1} })^{\top}.
\]
Then, for $i\in\{r+1, \ldots, n\}$, we have
\[
    \begin{split}
&\cfrac{\partial^{2} F_{i}}{\partial \alpha_{s}\partial \alpha_{l}}=0, \   s\neq l, \ \cfrac{\partial^{2} F_{i}}{\partial\alpha^{2}_{i}}=\sum_{k\neq i}\mu^{\prime\prime}(\alpha_{i}+\beta_{k}),\\
&\cfrac{\partial^{2} F_{i}}{\partial \beta_{s}\partial \alpha_{i}}=\mu^{\prime\prime}(\alpha_{i}+\beta_{s}), \   s=1, \ldots, n-1, s\neq i, \cfrac{\partial^{2} F_{i}}{\partial\beta_{i}\partial\alpha_{i}}=0,\\
&\cfrac{\partial^{2} F_{i}}{\partial \beta^{2}_{l}}=\mu^{\prime\prime}(\alpha_{i}+\beta_{l}) ,\  l=1, \ldots, n-1,l\neq i  \cfrac{\partial^{2} F_{i}}{\partial\beta_{s}\partial\beta_{l}}=0, s\neq l.
\end{split}
    \]
By the mean value theorem for vector-valued functions \cite[][p.341]{Lang:1993}, 
for $\mathbf{x}, \mathbf{\mathbf{y}} \in R^{2n-r-1}$, we have
\[
F'_i(\mathbf{x}) - F'_i(\mathbf{y}) = J^{(i)}(\mathbf{x}-\mathbf{y}),
\]
where
\[
J^{(i)}_{s,l} = \int_0^1 \frac{ \partial F'_{i,s} }{\partial \theta_l}(t\mathbf{x}+(1-t)\mathbf{y})dt,~~ s,l=r+1,\ldots, 2n-1.
\]
By (\ref{ineq-mu-deriv-bound}), we have
\[
\max_s \sum_{l=r+1}^{2n-1} |J^{(i)}|\le (2n-r-1)/c_{n}, ~~\sum_{s,l}|J^{(i)}_{s,l}|\le 2(2n-r-1)/c_{n}.
\]
Similarly, we also have
$F'_i(\mathbf{x}) - F'_i(\mathbf{y}) = J^{(i)}(\mathbf{x}-\mathbf{y})$ and $\sum_{s,l}|J^{(i)}_{s,l}|\le 2(2n-r-1)/c_{n}$ for $i\in\{n+1, \ldots, 2n-1\}$.
Consequently,
\[
\| F_i'(\mathbf{x}) - F_i'(\mathbf{y}) \|_\infty \le \| J^{(i)} \|_\infty \| \mathbf{x} - \mathbf{y}\|_\infty \le \frac{2n-r-1}{c_{n}}\| \mathbf{x} - \mathbf{y}\|_\infty,~~~i=r+1,\ldots, 2n-1,
\]
and for any vector $\mathbf{v}\in R^{2n-r-1}$,
\begin{eqnarray*}
\| [F'(\mathbf{x}) - F'(\mathbf{y})]\mathbf{v} \|_\infty  & = & \max_{i=r+1,\ldots,2n-1} |\sum_{j=r+1}^{2n-1} ( F_{i,j}'(\mathbf{x}) - F_{i,j}'(\mathbf{y}) ) v_j | \\
& = & \max_{i=r+1,\ldots,2n-1} |(\mathbf{x}-\mathbf{y})J^{(i)} \mathbf{v} | \\
& \le & \|\mathbf{x}-\mathbf{y}\|_\infty \| \mathbf{v}\|_\infty \sum_{s,l}|J^{(i)}_{s,l}|\\
&\le& \frac{2(2n-r-1)}{c_{n}}\|\mathbf{x}-\mathbf{y}\|_\infty \| \mathbf{v}\|_\infty.
\end{eqnarray*}

Step 2. We give the tail probability of $\|F(\bs{\theta}) \|_\infty$ satisfying
\begin{equation}\label{ineq-union-d}
\P\Bigg( \parallel F(\bs{\theta})\parallel_{\infty} \le \sqrt{n\log n} \Bigg)  \ge  1 - \frac{ 4 }{ n }.
\end{equation}
This is verified as follows.
Recall that $a_{i,j}, 1\leq i\neq j\leq n$ are independent Bernoulli random variables
and $F_i(\bs{\theta}) = \sum_{k\neq i} (\E a_{i,k} - a_{i,k})$ for $i=1,\ldots,n$, $F_{n+j}(\bs{\theta}) = \sum_{k\neq j} (\E a_{k,j} - a_{k,j})$ for $j=1,\ldots,n-1$.
By \citeauthor{Hoeffding:1963}'s \citeyearpar{Hoeffding:1963} inequality, we have
\begin{equation*}
\P\left( |F_i(\bs{\theta}) | \ge \sqrt{n\log n}  \right) \le 2\exp (- \frac{2n\log n}{n} )  =  \frac{2}{n^2},~~i=1, \ldots, n.
\end{equation*}
By the union bound, we have
\begin{eqnarray*}
\P\Bigg( \max_{i=1, \ldots, n} |F_i(\bs{\theta})| \ge \sqrt{n\log n} \Bigg)
\le  \sum_{i=1}^n \P\left(|F_i(\bs{\theta})| \geq \sqrt{n\log n} \right)
\le  \frac{2}{n }.
\end{eqnarray*}
Similarly, we have
\begin{eqnarray*}
\P\Bigg( \max_{j=1, \ldots, n-1} |F_{n+j}(\bs{\theta})| \ge \sqrt{n\log n} \Bigg)
\le  \frac{2}{n }.
\end{eqnarray*}
Consequently,
\begin{align*}
&\P\Bigg( \max\{\max_{i=1, \ldots, n} |F_i(\bs{\theta})|, \max_{j=1, \ldots, n-1} |F_{n+j}(\bs{\theta})|\} \ge \sqrt{n\log n} \Bigg) \\
\leq ~&\P\Bigg( \max_{i=1, \ldots, n} |F_i(\bs{\theta})| \ge \sqrt{n\log n} \Bigg)+\P\Bigg( \max_{j=1, \ldots, n-1} |F_{n+j}(\bs{\theta})| \ge \sqrt{n\log n} \Bigg)\\
\leq ~& \frac{4}{n},
\end{align*}
such that
\begin{align*}
   & \P\Bigg( \max\{\max_{i=r+1, \ldots, n} |F_i(\bs{\theta}), \max_{j=1, \ldots, n-1} |F_{n+j}(\bs{\theta})|\} \leq \sqrt{n\log n} \Bigg) \\
  \geq~ & \P\Bigg( \max\{\max_{i=1, \ldots, n} |F_i(\bs{\theta})|, \max_{j=1, \ldots, n-1} |F_{n+j}(\bs{\theta})|\} \leq\sqrt{n\log n} \Bigg)
\geq 1- \frac{4}{n}.
\end{align*}

Step 3. This step is one combining step.
The following calculations are based on the event $E_n$:
\[
E_n = \{ \max_{i=r+1,\ldots, 2n-1} |F_i(\bs{\theta})| \le   (n\log n)^{1/2}  \}.
\]
First, we calculate $K$. Recall that $V_{22}=(v_{i,j})=  F^\prime(\bs{\theta})$ and $\widetilde{W}_{22}=V_{22}^{-1}-S_{22}$, where $S_{22}$ is defined at (\ref {definition-S221}).
Then, we have
\begin{equation*}\label{seq5}
\begin{split}
&\parallel F^\prime(\boldsymbol{\theta})^{-1} ( F^\prime(\textbf{x}) - F^\prime(\textbf{y}))\parallel_{\infty}\\
\leq ~&\parallel \widetilde{W}_{22}(F'(\textbf{x})-F'(\textbf{y}))\parallel_{\infty}
+\parallel S_{22}(F'(\textbf{x})-F'(\textbf{y}))\parallel_{\infty}\\
\leq~&\frac{b_{n}^{3}(2n-1-r)}{(n-1)^{2}c_{n}^{2}}\times \frac{2(2n-r-1)}{c_{n}}\parallel\mathbf{x-y}\parallel_{\infty}
+\frac{2b_{n}(2n-r-1)}{c_{n}(n-1)}\parallel\mathbf{x-y}\parallel_{\infty},
\end{split}
 \end{equation*}
where the last equations is due to Lemma \ref{lemma-appro-beta-VS} and Step 1.
By the event $E_n$, we have $\| F(\bs{\theta}) \|_\infty  \le (n\log n)^{1/2}.$
Next, we calculate $\eta$. Note that
\[
    \begin{split}
\|[F'(\boldsymbol{\theta})]^{-1}F(\boldsymbol{\theta})\|_{\infty}
&\leq\parallel \widetilde{W}_{22}F(\boldsymbol{\theta})\parallel_{\infty}
+\parallel S_{22}F(\boldsymbol{\theta})\parallel_{\infty}\\
&\leq \frac{b_{n}^{3}(2n-1-r)\|F(\boldsymbol{\theta})\|_{\infty}}{(n-1)^{2}c_{n}^2}
+\frac{|F_{2n}(\boldsymbol{\theta})|}{\widetilde v_{2n,2n}}
+\max_{i=r+1,\cdots,2n-1}\frac{|F_{i}(\boldsymbol{\theta})|}{v_{i,i}}\\
&\leq  \frac{b_{n}^{3}(2n-1-r)}{(n-1)^{2}c_{n}^2}\times\sqrt{n\log n}
+\frac{b_{n}}{n-1}\times\sqrt{n\log n}\\
&\leq \frac{b_{n}^{3}}{c_{n}^2}\sqrt{\frac{\log n}{n}}.
\end{split}
    \]
Note that for any $r\in [0,n-1]$, $2n-1-r\le 2n-1$.
Therefore, if $ b_n^6/c_n^{5}=o( (n/\log n)^{1/2} )$, then
\begin{eqnarray*}
K\eta  \le  \frac{2b_{n}^{3}(2n-1-r)^2}{(n-1)^{2}c_{n}^{3}}
\times   \frac{b_{n}^{3}}{c_{n}^2}\sqrt{\frac{\log n}{n}}
 = \frac{ 2 b_{n}^{6}(2n-1-r)^{2}}{(n-1)^{2}c_{n}^5}\sqrt{\frac{\log n}{n}} =o(1).
\end{eqnarray*}
The above arguments verify the conditions.
By Lemma \ref{lemma:Newton:Kantovorich}, it yields that
\begin{equation}
\label{eq-hatbeta-upper}
\| \bs{\widehat{\theta}}^0 - \bs{\theta} \|_\infty \le \frac{b_{n}^{3}}{c_{n}^2}\sqrt{\frac{\log n}{n}}.
\end{equation}
Step 2 implies $\P (E_n) \geq 1 - 4/n$. This completes the proof.
\end{proof}

\subsection{Proof of Lemma \ref{lemma:beta3:err}}
\label{subsection-proof-lemma5}
\begin{proof}
Note that $\alpha_1, \ldots, \alpha_r$ are known and $\alpha_{r+1}, \ldots, \alpha_n,\beta_{1},\ldots,\beta_{n-1}$ are unknown.
Following Theorem 2 in \cite{Yan:Leng:Zhu:2016},
let $E_{n1}$ be the event that
\begin{equation}\label{eq-a-zcc}
\begin{split}
E_{n1} = \{ \widehat{\alpha}_i - \alpha_i& = \frac{ \bar{d}_i }{ v_{i,i} }+ \frac{\bar b_{n} }{  v_{2n,2n} }+ \gamma_i, ~~i=r+1, \ldots, n, \\
\widehat{\beta}_j -\beta_j& = \frac{ \bar{b}_j }{ v_{n+j,n+j} }- \frac{\bar b_{n} }{  v_{2n,2n} }+ \gamma_{n+j}, ~~j=1, \ldots, n-1 \},
\end{split}
\end{equation}
where $\max\{ |\gamma_i|,|\gamma_{n+j}|\}\leq \frac{b_{n}^{9}\log n}{c_{n}^{7}n}.$ Let $E_{n2}$ be the event
\begin{equation}\label{eq-upp-bard}
E_{n2} =\{ \| \bs{\bar{\mathbf{g}} } \|_\infty \le (n\log n)^{1/2}\}.
\end{equation}
By \eqref{ineq-union-d},
$E_{n1}\bigcap E_{n2}$ holds with probability at least $1-8/n$.
The following calculations are based on $E_{n1}\bigcap E_{n2}$.
Let
\begin{equation}
\begin{split}
  f_{ij} & = \mu^{\prime\prime}({\pi}_{ij} ), ~~~~
    f_{i} =\sum_{k=1;k\neq i}^{n-1} \mu^{\prime\prime}( \pi_{ik} ),~~ i=r+1,\ldots,n,\\
    f_{n+i}&=\sum_{k=r+1;k\neq i}^{n}\mu^{\prime\prime}( \pi_{ki} ), ~~i=1,\ldots,n-1.
\end{split}
\end{equation}
In view of \eqref{ineq-mu-deriv-bound}, we have
\begin{equation}\label{ineq-fi-fij}
\max_{i=r+1,\ldots,n} |f_i| \le \frac{ n-1}{ c_n },~~\max_{i=1,\ldots,n-1} |f_{n+i}| \le \frac{ n-r}{ c_n }, ~~ \max_{i,j} |f_{ij} | \le \frac{ 1}{c_n}.
\end{equation}
By substituting the expression of $\widehat{\alpha}_i - \alpha_i$ in \eqref{eq-a-zcc} into $f_i(\widehat{\alpha}_i-\alpha_i)^3$, we get
\begin{equation}\label{eq-a-za}
\begin{split}
\sum_{i=r+1}^{n}  f_i(\widehat{\alpha}_i-\alpha_i)^3= &  \sum_{i=r+1}^{n}  f_i(\frac{ \bar{d}_i }{ v_{i,i} }+ \frac{\bar b_{n} }{ v_{2n,2n} }+ \gamma_i)^3 \\
 =&  \sum_{i=r+1}^{n}  f_i(\frac{ \bar{d}_i }{ v_{i,i} })^3+\sum_{i=r+1}^{n}  f_i(\frac{\bar b_{n} }{ v_{2n,2n} })^3+\sum_{i=r+1}^{n}  f_i\gamma_i^{3}+3\sum_{i=r+1}^{n}  f_i(\frac{ \bar{d}_i }{ v_{i,i} })^2\frac{\bar b_{n} }{ v_{2n,2n} }\\
 +&3\sum_{i=r+1}^{n}  f_i(\frac{ \bar{d}_i }{ v_{i,i} })^2\gamma_i+3\sum_{i=r+1}^{n}  f_i\frac{ \bar{d}_i }{ v_{i,i} }(\frac{\bar b_{n} }{ v_{2n,2n} })^2+3\sum_{i=r+1}^{n}  f_i\frac{ \bar{d}_i }{ v_{i,i} }\gamma_i^2\\
 +&3\sum_{i=r+1}^{n}  f_i\frac{\bar b_{n} }{ v_{2n,2n} }\gamma_i^2+3\sum_{i=r+1}^{n}  f_i(\frac{\bar b_{n} }{ v_{2n,2n} })^2\gamma_i+6\sum_{i=r+1}^{n}  f_i\frac{ \bar{d}_i }{ v_{i,i} }\frac{\bar b_{n} }{ v_{2n,2n} }\gamma_i.
\end{split}
\end{equation}
We bound the nine terms in the above right hand in an inverse order.
\begin{eqnarray}
\nonumber
|\sum_{i=r+1}^{n}  f_i\frac{ \bar{d}_i }{ v_{i,i} }\frac{\bar b_{n} }{ v_{2n,2n} }\gamma_i| & \le & (n-r) \cdot \max_i \frac{1}{v_{i,i}} \cdot \| \bs{\bar{d}} \|_\infty \cdot\max_{i=r+1,\ldots,n} |f_i|\cdot \frac{\bar b_{n} }{ v_{2n,2n} }\cdot\max_i |\gamma_i| \\
\nonumber
& \lesssim & (n-r)\cdot (\frac{b_n}{n})^2 \cdot (n\log n) \cdot \frac{ n }{c_n } \cdot (\frac{ b_n^9 \log n }{ nc_n^7}) \\
\label{eq-a-ten}
& = &  \frac{ (n-r)b_n^{11} (\log n)^2 }{ n c_n^8 }.
\end{eqnarray}
\begin{eqnarray}
\nonumber
|\sum_{i=r+1}^{n}  f_i(\frac{\bar b_{n} }{ v_{2n,2n} })^2\gamma_i| & \le & (n-r) \cdot\max_{i=r+1,\ldots,n} |f_i|\cdot (\frac{\bar b_{n} }{ v_{2n,2n} })^2\cdot\max_i |\gamma_i| \\
\nonumber
& \lesssim & (n-r)\cdot (\frac{b_n}{n})^2 \cdot (n\log n) \cdot \frac{ n }{c_n } \cdot (\frac{ b_n^9 \log n }{ nc_n^7}) \\
\label{eq-a-nine}
& = &  \frac{ (n-r)b_n^{11} (\log n)^2 }{ n c_n^8 }.
\end{eqnarray}
\begin{eqnarray}
\nonumber
|\sum_{i=r+1}^{n}  f_i\frac{\bar b_{n} }{ v_{2n,2n} }\gamma_i^2| & \le &  (n-r) \cdot\max_{i=r+1,\ldots,n} |f_i|\cdot \frac{\bar b_{n} }{ v_{2n,2n} }\cdot\max_i |\gamma_i^2| \\
\nonumber
& \lesssim & (n-r)\cdot \frac{b_n}{n} \cdot \sqrt{n\log n} \cdot \frac{ n }{c_n } \cdot (\frac{ b_n^9 \log n }{ nc_n^7})^2 \\
\label{eq-a-eight}
&= &  \frac{ (n-r)b_n^{19} (\log n)^{5/2} }{ n^{3/2} c_n^{15} }.
\end{eqnarray}
\begin{eqnarray}
\nonumber
|\sum_{i=r+1}^{n}  f_i\frac{ \bar{d}_i }{ v_{i,i} }\gamma_i^2| & \le & (n-r) \cdot \max_i \frac{1}{v_{i,i}} \cdot \| \bs{\bar{d}} \|_\infty \cdot\max_{i=r+1,\ldots,n} |f_i|\cdot\max_i |\gamma_i^2| \\
\nonumber
& \lesssim & (n-r)\cdot \frac{b_n}{n} \cdot \sqrt{n\log n} \cdot \frac{ n }{c_n } \cdot (\frac{ b_n^9 \log n }{ nc_n^7})^2 \\
\label{eq-a-seven}
& = &  \frac{ (n-r)b_n^{19} (\log n)^{5/2} }{ n^{3/2} c_n^{15} }.
\end{eqnarray}
Similarly,
\begin{eqnarray}
|\sum_{i=r+1}^{n}  f_i\frac{ \bar{d}_i }{ v_{i,i} }(\frac{\bar b_{n} }{ v_{2n,2n} })^2|
\label{eq-a-six}
& \lesssim &  \frac{(n-r)^{1/2}b_n^{2} (\log n)^{3/2} }{ n^{1/2} c_n }.\\
|\sum_{i=r+1}^{n}  f_i(\frac{ \bar{d}_i }{ v_{i,i} })^2\gamma_i|
\label{eq-a-five}
& \lesssim &  \frac{ (n-r)^{1/2}b_n^{10} \log n }{ n c_n^8 }.\\
|\sum_{i=r+1}^{n}  f_i(\frac{ \bar{d}_i }{ v_{i,i} })^2\frac{\bar b_{n} }{ v_{2n,2n} }|
\label{eq-a-four}
& \lesssim &  \frac{ (n-r)^{1/2} b_n \log n }{ n^{1/2} c_n }.\\
|\sum_{i=r+1}^{n}  f_i\gamma_i^{3} |
\label{eq-a-three}
& \lesssim&  \frac{ (n-r) b_n^{27}(\log n)^3 }{ n^2 c_n^{22} }.
\end{eqnarray}
Now, we bound the first term.
By Lemma 17 in \cite{yan2025likelihood}, we have that
\begin{equation}\label{ineq-d-firstaa}
\mathrm{Var}( \sum_{i=r+1}^n \frac{f_i\bar{d}_i^3}{v_{i,i}^3} ) \le \max_i |f_i|^2 \mathrm{Var}( \sum_{i=r+1}^n \frac{\bar{d}_i^3}{v_{i,i}^3} )
\lesssim \frac{ n^2}{c_n^2 } \cdot \frac{ b_n^6 }{ n^6 } \cdot \frac{ n^3(n-r) }{ c_n^3}= \frac{ (n-r)b_n^6 }{ n c_n^5 }.
\end{equation}
Because
\[
\E \bar{d}_{i}^{\,3} = \sum_{\alpha, \gamma, \zeta} \E \bar{a}_{i,\alpha}\bar{a}_{i,\gamma}\bar{a}_{i,\zeta} = \sum_{\alpha\neq i} \E \bar{a}_{i,\alpha}^{\,3},
\]
we have
\[
|\E \bar{a}_{i,j}^{\,3}|  = | p_{ij}q_{ij}( p_{ij}^2 - q_{ij}^2)|\le\frac{1}{c_n},
\]
such that
\begin{equation}\label{eq-ed33}
| \sum_{i=r+1}^n  \left (  \frac{\E \bar{d}_i^3 }{ v_{i,i}^3 } \right)f_i | \lesssim (n-r) \cdot \frac{n}{c_n} \cdot \frac{b_n^3}{n^3} \cdot \frac{n-r}{c_n} = \frac{(n-r)^2b_n^3}{n^2c_n^2}.
\end{equation}
In view of \eqref{ineq-d-firstaa} and \eqref{eq-ed33}, we have
\begin{equation}\label{ineq-d-first}
|\sum_{i=r+1}^n  \left (  \frac{ \bar{d}_i^3 }{ v_{i,i}^3 } \right)f_i | = O_p\left( \frac{ b_n^3 }{c_n^{5/2}} \left( \frac{ n-r }{ n} \right)^{1/2} \right).
\end{equation}
By combining the upper bounds of the above nine terms in \eqref{eq-a-ten}, \eqref{eq-a-nine}, \eqref{eq-a-eight}, \eqref{eq-a-seven}, \eqref{eq-a-six}, \eqref{eq-a-five}, \eqref{eq-a-four}, \eqref{eq-a-three} and \eqref{ineq-d-first}, it yields that

\begin{align*}
  \sum_{i=r+1}^{n}  f_i(\widehat{\alpha}_i-\alpha_i)^3= &O_p(  \frac{ (n-r)b_n^{11} (\log n)^2 }{ n c_n^8 }
+ \frac{ (n-r)b_n^{19} (\log n)^{5/2} }{ n^{3/2} c_n^{15} } +
 \frac{(n-r)^{1/2}b_n^{2} (\log n)^{3/2} }{ n^{1/2} c_n } \\
  + & \frac{ (n-r)^{1/2}b_n^{10} \log n }{ n c_n^8 }+ \frac{ (n-r)^{1/2} b_n \log n }{ n^{1/2} c_n }+ \frac{ (n-r) b_n^{27}(\log n)^3 }{ n^2 c_n^{22} }\\
  +&\frac{ b_n^3 }{c_n^{5/2}} \left( \frac{ n-r }{ n} \right)^{1/2})+\sum_{i=r+1}^{n}  f_i\frac{\bar b_{n}^3 }{ v_{2n,2n}^3 }.
\end{align*}
This leads to \eqref{eq-S1-bound}. The proof for \eqref{eq-S2-bound} is similar and we omit it.

By substituting the expression of $\widehat{\alpha}_i - \alpha_i$ and $\widehat{\beta}_j - \beta_j$ in \eqref{eq-a-zcc} into $\sum_{i=r+1 }^n \sum_{j=1,j\neq i}^{n-1} (\widehat{\alpha}_i-\alpha_i)^2(\widehat{\beta}_j-\beta_j)\mu^{\prime\prime}( \pi_{ij} )$, we get
\begin{equation}\label{eq-a-za}
\begin{split}
&\sum_{i=r+1 }^n \sum_{j=1,j\neq i}^{n-1} (\widehat{\alpha}_i-\alpha_i)^2(\widehat{\beta}_j-\beta_j)\mu^{\prime\prime}( \pi_{ij} )\\= &  \sum_{i=r+1 }^n \sum_{j=1,j\neq i}^{n-1}(\frac{ \bar{d}_i }{ v_{i,i} }+ \frac{\bar b_{n} }{ v_{2n,2n} }+ \gamma_i)^2(\frac{ \bar{b}_j }{ v_{n+j,n+j} }- \frac{\bar b_{n} }{ v_{2n,2n} }+ \gamma_{n+j}) f_{ij} \\
 =& \sum_{i=r+1 }^n \sum_{j=1,j\neq i}^{n-1} f_{ij}\frac{ \bar{d}_i^2 }{ v_{i,i}^2 }\frac{ \bar{b}_j }{ v_{n+j,n+j}}-\sum_{i=r+1 }^n \sum_{j=1,j\neq i}^{n-1} f_{ij}\frac{ \bar{d}_i^2 }{ v_{i,i}^2 }\frac{\bar b_{n} }{ v_{2n,2n} }+\sum_{i=r+1 }^n \sum_{j=1,j\neq i}^{n-1} f_{ij}\frac{ \bar{d}_i^2 }{ v_{i,i}^2 }\gamma_{n+j}\\
+&\sum_{i=r+1 }^n \sum_{j=1,j\neq i}^{n-1} f_{ij}\frac{ \bar{b}_n^2 }{ v_{2n,2n}^2 }\frac{ \bar{b}_j }{ v_{n+j,n+j}}-\sum_{i=r+1 }^n \sum_{j=1,j\neq i}^{n-1} f_{ij}(\frac{\bar b_{n} }{ v_{2n,2n} })^3+\sum_{i=r+1 }^n \sum_{j=1,j\neq i}^{n-1} f_{ij}\frac{ \bar{b}_n^2 }{ v_{2n,2n}^2 }\gamma_{n+j}\\
+&\sum_{i=r+1 }^n \sum_{j=1,j\neq i}^{n-1} f_{ij}\gamma_{i}^2\frac{ \bar{b}_j }{ v_{n+j,n+j}}-\sum_{i=r+1 }^n \sum_{j=1,j\neq i}^{n-1} f_{ij}\frac{\bar b_{n} }{ v_{2n,2n} }\gamma_{i}^2+\sum_{i=r+1 }^n \sum_{j=1,j\neq i}^{n-1} f_{ij}\gamma_{i}^2\gamma_{n+j}\\
+&2\sum_{i=r+1 }^n \sum_{j=1,j\neq i}^{n-1} f_{ij}\frac{ \bar{d}_i }{ v_{i,i} }\frac{\bar b_{n} }{ v_{2n,2n} }\frac{ \bar{b}_j }{ v_{n+j,n+j}}-2\sum_{i=r+1 }^n \sum_{j=1,j\neq i}^{n-1} f_{ij}\frac{ \bar{d}_i }{ v_{i,i} }\frac{\bar b_{n}^2 }{ v_{2n,2n}^2 }+2\sum_{i=r+1 }^n \sum_{j=1,j\neq i}^{n-1} f_{ij}\frac{ \bar{d}_i }{ v_{i,i} }\frac{\bar b_{n} }{ v_{2n,2n} }\gamma_{n+j}\\
+&2\sum_{i=r+1 }^n \sum_{j=1,j\neq i}^{n-1} f_{ij}\frac{ \bar{d}_i }{ v_{i,i} }\gamma_{i}\frac{ \bar{b}_j }{ v_{n+j,n+j}}-2\sum_{i=r+1 }^n \sum_{j=1,j\neq i}^{n-1} f_{ij}\frac{ \bar{d}_i }{ v_{i,i} }\gamma_{i}\frac{\bar b_{n} }{ v_{2n,2n} }+2\sum_{i=r+1 }^n \sum_{j=1,j\neq i}^{n-1} f_{ij}\frac{ \bar{d}_i }{ v_{i,i} }\gamma_{i}\gamma_{n+j}\\
+&2\sum_{i=r+1 }^n \sum_{j=1,j\neq i}^{n-1} f_{ij}\frac{\bar b_{n} }{ v_{2n,2n} }\gamma_{i}\frac{ \bar{b}_j }{ v_{n+j,n+j}}-2\sum_{i=r+1 }^n \sum_{j=1,j\neq i}^{n-1} f_{ij}\frac{\bar b_{n}^2 }{ v_{2n,2n}^2 }\gamma_{i}+2\sum_{i=r+1 }^n \sum_{j=1,j\neq i}^{n-1} f_{ij}\frac{\bar b_{n} }{ v_{2n,2n} }\gamma_{i}\gamma_{n+j}.
\end{split}
\end{equation}
Now, we bound the first term.
By Lemma 18 in \cite{yan2025likelihood}, we have that
\begin{equation}\label{ineq-d-firsta}
\mathrm{Var}( \sum_{i=r+1 }^n \sum_{j=1,j\neq i}^{n-1} f_{ij}\frac{ \bar{d}_i^2 }{ v_{i,i}^2 }\frac{ \bar{b}_j }{ v_{n+j,n+j}})
\lesssim  \frac{ b_n^6 }{ n^6 c_n } \cdot \frac{ n^5(n-r) }{ c_n^3} = \frac{ (n-r)b_n^6 }{ n c_n^4 },
\end{equation}
and
\begin{equation}\label{eq-ed3}
| \sum_{i=r+1 }^n \sum_{j=1,j\neq i}^{n-1} f_{ij}\frac{ \E\bar{d}_i^2\bar{b}_j }{ v_{i,i}^2 v_{n+j,n+j}} | \lesssim n \cdot \frac{n}{c_n} \cdot \frac{b_n^3}{n^3} \cdot \frac{n-r}{c_n} = \frac{(n-r)b_n^3}{nc_n^2}.
\end{equation}
By Chebyshev's inequality, we have
\[
|\sum_{i=r+1 }^n \sum_{j=1,j\neq i}^{n-1} f_{ij}\frac{ \bar{d}_i^2 }{ v_{i,i}^2 }\frac{ \bar{b}_j }{ v_{n+j,n+j}}|=O_{p}\left(\frac{ (n-r)^{1/2}b_n^3 }{ n^{1/2} c_n^2}\right).
\]
The remainder terms are calculated similarly to formula \eqref{eq-S1-bound}. By directing calculation, we have
\begin{equation}\label{eq-a-za}
\begin{split}
&\sum_{i=r+1 }^n \sum_{j=1,j\neq i}^{n-1} (\widehat{\alpha}_i-\alpha_i)^2(\widehat{\beta}_j-\beta_j)\mu^{\prime\prime}( \pi_{ij} )\\
\leq & \frac{ (n-r)^{1/2}b_n^3 }{ n^{1/2} c_n^2}+\frac{ (n-r)^{1/2} b_n \log n }{ n^{1/2} c_n } +\frac{ (n-r)^{1/2}b_n^{10} \log n }{ n c_n^8 } +\frac{ (n-r) b_n^{27}(\log n)^3 }{ n^2 c_n^{22} }\\
 +&\frac{ (n-r)b_n^{11} (\log n)^2 }{ n c_n^8 }
+ \frac{ (n-r)b_n^{19} (\log n)^{5/2} }{ n^{3/2} c_n^{15} }
-\sum_{i=r+1 }^n \sum_{j=1,j\neq i}^{n-1} f_{ij}(\frac{\bar b_{n} }{ v_{2n,2n} })^3.
\end{split}
\end{equation}
This leads to \eqref{eq-S3-bound}. The proof for \eqref{eq-S4-bound} is similar and we omit it.
\end{proof}

\section{Bernstein's inequality}
\label{section-bernstein}

This section collects a user-friendly version of the Bernstein inequality on bounded random variables.
The following Bernstein inequality can be easily found in textbooks such as \cite{boucheron2013concentration}. The proof is omitted.

\begin{lemma}[Bernstein's inequality]\label{lemma:bernstein}
Suppose $n$ independent random variables $x_{i}$ ($1\le i \le n$)
each satisfying $\left|x_{i}\right|\leq B$. For any $a\geq2$, one
has
\[
\left|\sum_{i=1}^{n}x_{i}-\E\left[\sum_{i=1}^{n}x_{i}\right]\right|\leq\sqrt{2a\log n\sum_{i=1}^{n}\E\left[x_{i}^{2}\right]}+\frac{2a}{3}B\log n
\]
with probability at least $1-2n^{-a}$.
\end{lemma}

\newpage
\section{Algorithms}
\label{section:algorithms}

\begin{algorithm}
\caption{Restricted MLE under $H_0: \alpha_i = \alpha_i^0$ for $i=1,...,r$}
\label{algorithm:b}
\begin{algorithmic}[1]
\Require Adjacency matrix $A=(a_{ij})$, fixed $\alpha_1^0,...,\alpha_r^0$, tolerance $\epsilon$
\State Initialize: $\boldsymbol{\alpha}^{(0)} \gets (\alpha_1^0,...,\alpha_r^0,\alpha_{r+1},...,\alpha_n)$, $\boldsymbol{\beta}^{(0)}$, $k \gets 0$
\State Compute out-degrees $d_i \gets \sum_{j\neq i}a_{ij}$ and in-degrees $b_j \gets \sum_{i\neq j}a_{ij}$

\Repeat
    \For{$i = 1$ \textbf{to} $n$}
        \State $\alpha_i^{(k+1)} \gets \log d_i - \log\sum\limits_{j\neq i}\frac{e^{\beta_j^{(k)}}}{1+e^{\alpha_i^{(k)}+\beta_j^{(k)}}}$
    \EndFor

    \For{$j = 1$ \textbf{to} $n$}
        \State $\beta_j^{(k+1)} \gets \log b_j - \log\sum\limits_{i\neq j}\frac{e^{\alpha_i^{(k)}}}{1+e^{\alpha_i^{(k)}+\beta_j^{(k)}}}$
    \EndFor

    \State Let $\alpha_{i}^{(k+1)}=(\alpha_{1}^{0}, \ldots, \alpha_{r}^{0}, \alpha_{r+1}^{(k+1)}, \ldots, \alpha_{n}^{(k+1)})$
    \State Compute convergence:
    \State $y_i \gets \frac{1}{d_i}\sum\limits_{j\neq i}\frac{e^{\alpha_i^{(k+1)}+\beta_j^{(k+1)}}}{1+e^{\alpha_i^{(k+1)}+\beta_j^{(k+1)}}}$ for $i=r+1,...,n$
    \State $z_j \gets \frac{1}{b_j}\sum\limits_{i\neq j}\frac{e^{\alpha_i^{(k+1)}+\beta_j^{(k+1)}}}{1+e^{\alpha_i^{(k+1)}+\beta_j^{(k+1)}}}$ for $j=1,...,n$
    \State $k \gets k + 1$
\Until{$\max(\max_i|y_i-1|, \max_j|z_j-1|) < \epsilon$}

\State \Return $\hat{\boldsymbol{\alpha}} \gets \boldsymbol{\alpha}^{(k)}$, $\hat{\boldsymbol{\beta}} \gets \boldsymbol{\beta}^{(k)}$
\end{algorithmic}
\end{algorithm}
\begin{algorithm}
\caption{Unrestricted MLE}
\label{alg:unrestricted_mle}
\begin{algorithmic}[1]
\Require Adjacency matrix $A = (a_{ij})_{n \times n}$, tolerance $\epsilon$
\State Initialize: $\boldsymbol{\alpha}^{(0)}$, $\boldsymbol{\beta}^{(0)}$, $k \leftarrow 0$
\State Compute out-degrees $d_i \leftarrow \sum_{j\neq i}a_{ij}$ and in-degrees $b_j \leftarrow \sum_{i\neq j}a_{ij}$
\Repeat
    \For{$i = 1$ to $n$}
        \State $\alpha^{(k+1)}_i \leftarrow \log d_i - \log \sum_{j\neq i} \frac{e^{\beta^{(k)}_j}}{1+e^{\alpha^{(k)}_i+\beta^{(k)}_j}}$
    \EndFor
    \For{$j = 1$ to $n$}
        \State $\beta^{(k+1)}_j \leftarrow \log b_j - \log \sum_{i\neq j} \frac{e^{\alpha^{(k)}_i}}{1+e^{\alpha^{(k)}_i+\beta^{(k)}_j}}$
    \EndFor
    \State Normalize:
    \State $\boldsymbol{\alpha}^{(k+1)} \leftarrow \boldsymbol{\alpha}^{(k+1)} + \beta^{(k+1)}_n$
    \State $\boldsymbol{\beta}^{(k+1)} \leftarrow \boldsymbol{\beta}^{(k+1)} - \beta^{(k+1)}_n$
    \State Compute convergence:
    \State $y_i \leftarrow \frac{1}{d_i}\sum_{j\neq i} \frac{e^{\alpha^{(k+1)}_i+\beta^{(k+1)}_j}}{1+e^{\alpha^{(k+1)}_i+\beta^{(k+1)}_j}}$ for $i = 1,\ldots,n$
    \State $z_j \leftarrow \frac{1}{b_j}\sum_{i\neq j} \frac{e^{\alpha^{(k+1)}_i+\beta^{(k+1)}_j}}{1+e^{\alpha^{(k+1)}_i+\beta^{(k+1)}_j}}$ for $j = 1,\ldots,n$
    \State $k \leftarrow k + 1$
    \Until{$\max(\max_i|y_i-1|, \max_j|z_j-1|) < \epsilon$}
\State \Return $\hat{\boldsymbol{\alpha}} \leftarrow \boldsymbol{\alpha}^{(k)}$, $\hat{\boldsymbol{\beta}} \leftarrow \boldsymbol{\beta}^{(k)}$
\end{algorithmic}
\end{algorithm}

\newpage
\section{Figures}
\label{section:figure}

\begin{figure}
\centering
\subfigure[QQ-plot for normalized log-likelihood ratio statistic]{\includegraphics[width=0.9\textwidth]{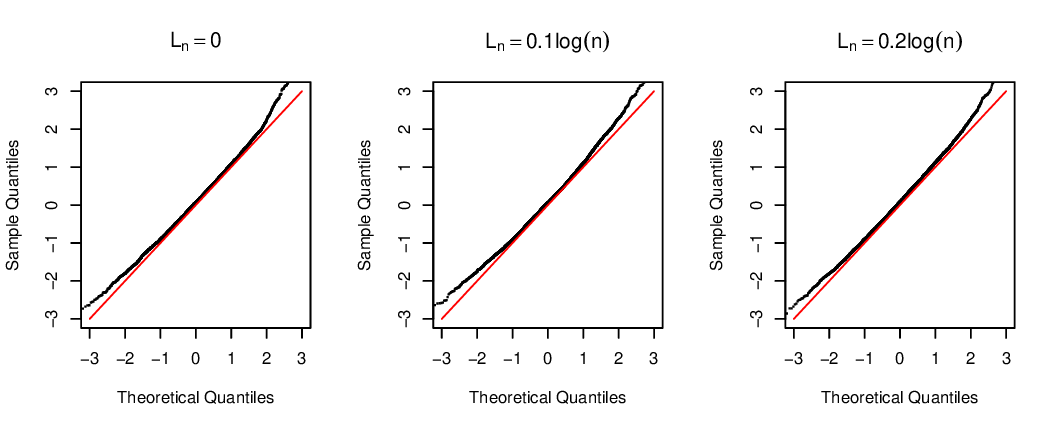}}
\subfigure[QQ-plot for log-likelihood ratio statistic]{\includegraphics[width=0.9\textwidth]{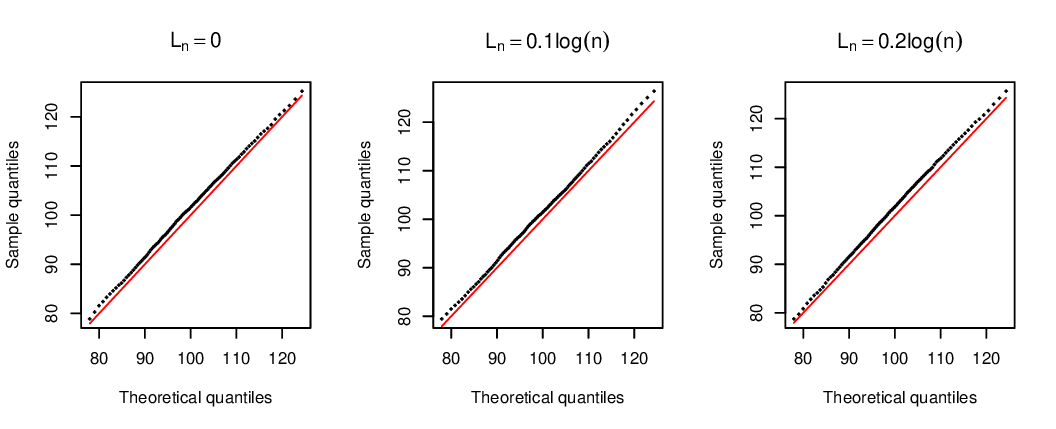}}
\caption{QQ plots for the $p_{0}$-model under $H_{01}$ (n=100). The horizontal and vertical axes in each QQ-plot are the respective theoretical and empirical quantiles.
The straight lines correspond to $y=x$.}
\label{fig:p11}
\end{figure}

\begin{figure}
\centering
\subfigure[QQ-plot for normalized log-likelihood ratio statistic]{\includegraphics[width=0.9\textwidth]{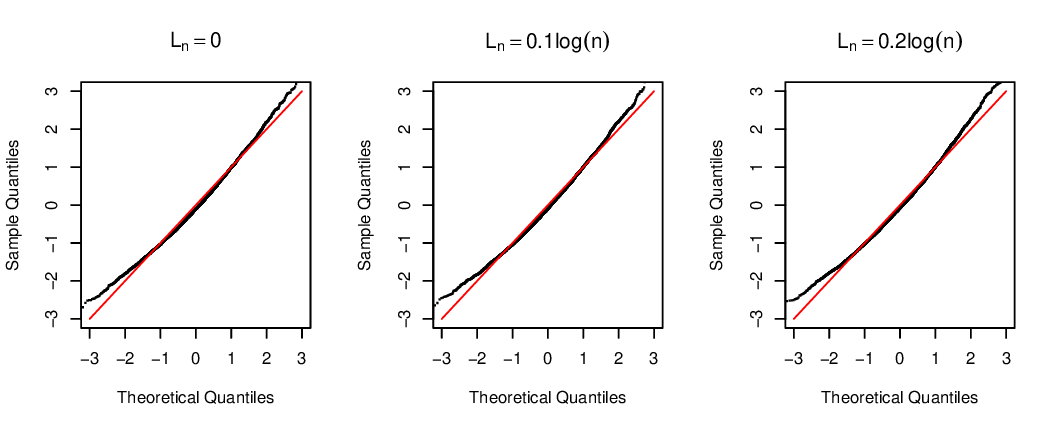}}
\subfigure[QQ-plot for log-likelihood ratio statistic]{\includegraphics[width=0.9\textwidth]{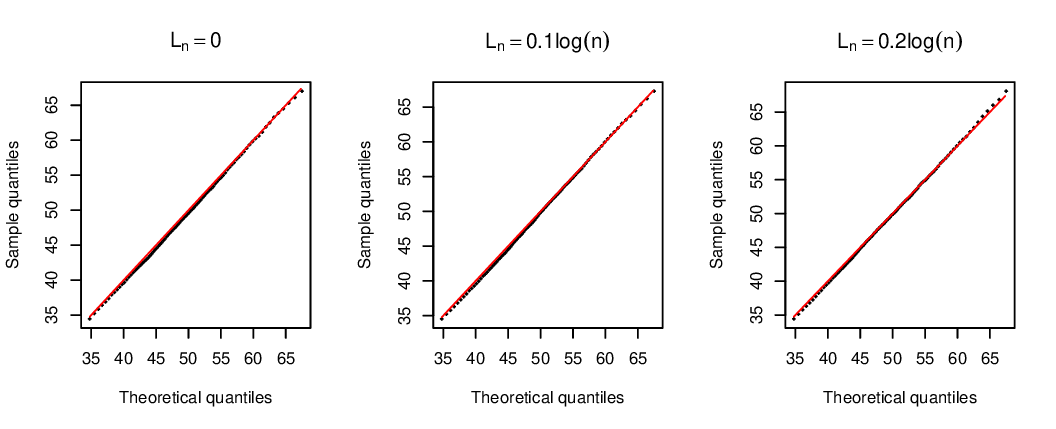}}
\caption{QQ plots for the $p_{0}$-model under $H_{02}$ (n=100). The horizontal and vertical axes in each QQ-plot are the respective theoretical and empirical quantiles.
The straight lines correspond to $y=x$.}
\label{fig:p21}
\end{figure}

\begin{figure}
\centering
{\includegraphics[width=0.9\textwidth]{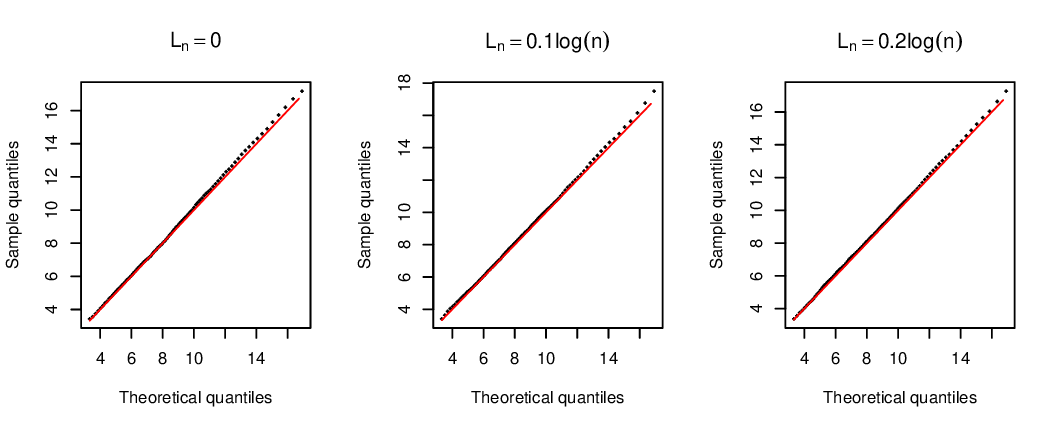}}
\caption{QQ plots for the $p_{0}$-model under $H_{03}$ (n=100). The horizontal and vertical axes in each QQ-plot are the respective theoretical and empirical quantiles.
The straight lines correspond to $y=x$.}
\label{fig:p31}
\end{figure}

\begin{figure}
\centering
{\includegraphics[width=0.9\textwidth]{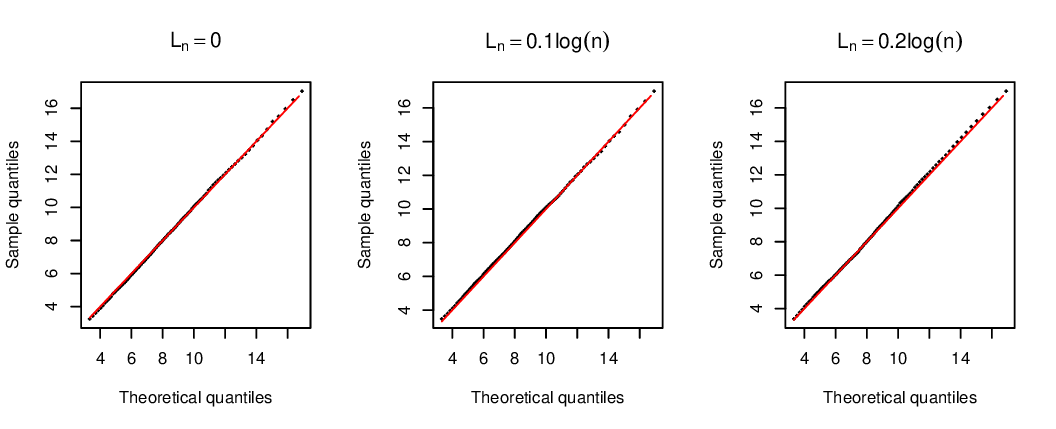}}
\caption{QQ plots for the $p_{0}$-model under $H_{03}$ (n=200). The horizontal and vertical axes in each QQ-plot are the respective theoretical and empirical quantiles.
The straight lines correspond to $y=x$.}
\label{fig:p32}
\end{figure}

\begin{figure}
\centering
\subfigure[QQ-plot for normalized log-likelihood ratio statistic]{\includegraphics[width=0.9\textwidth]{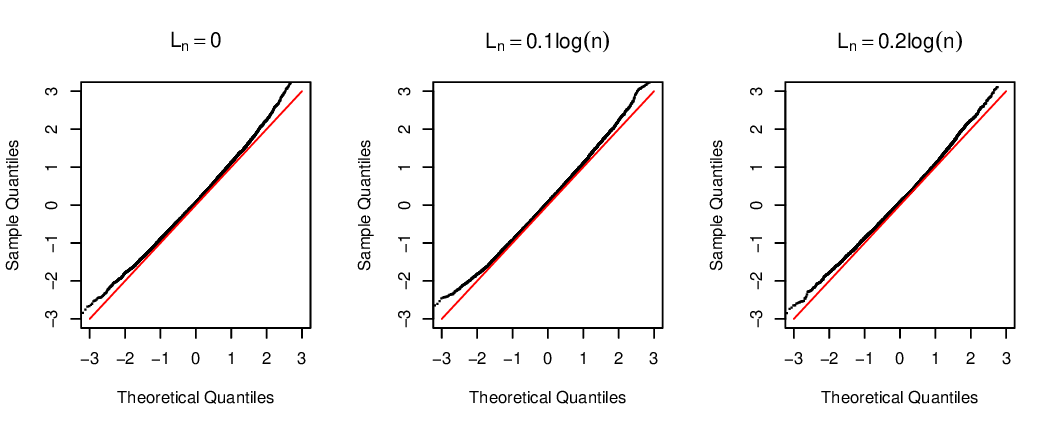}}
\subfigure[QQ-plot for log-likelihood ratio statistic]{\includegraphics[width=0.9\textwidth]{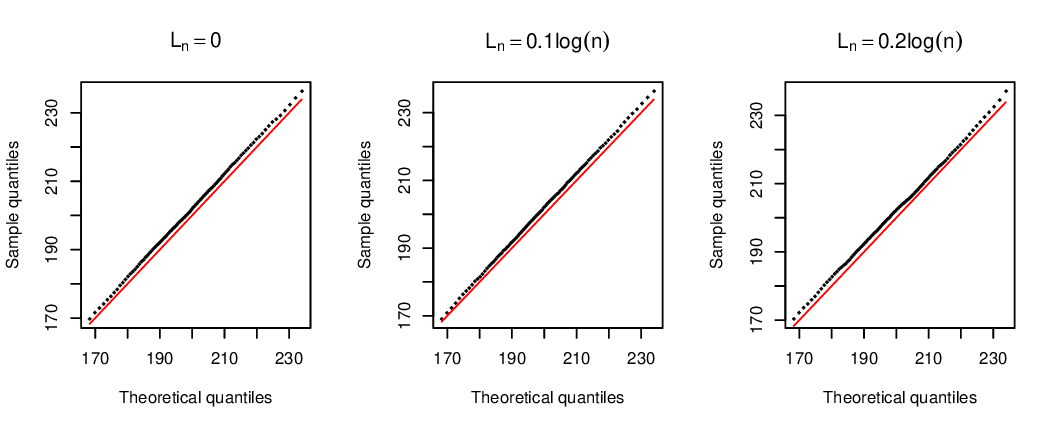}}
\caption{QQ plots for the $p_{0}$-model under $H_{01}$ (n=200). The horizontal and vertical axes in each QQ-plot are the respective theoretical and empirical quantiles.
The straight lines correspond to $y=x$.}
\label{fig:p12}
\end{figure}

\begin{figure}
\centering
\subfigure[QQ-plot for normalized log-likelihood ratio statistic]{\includegraphics[width=0.9\textwidth]{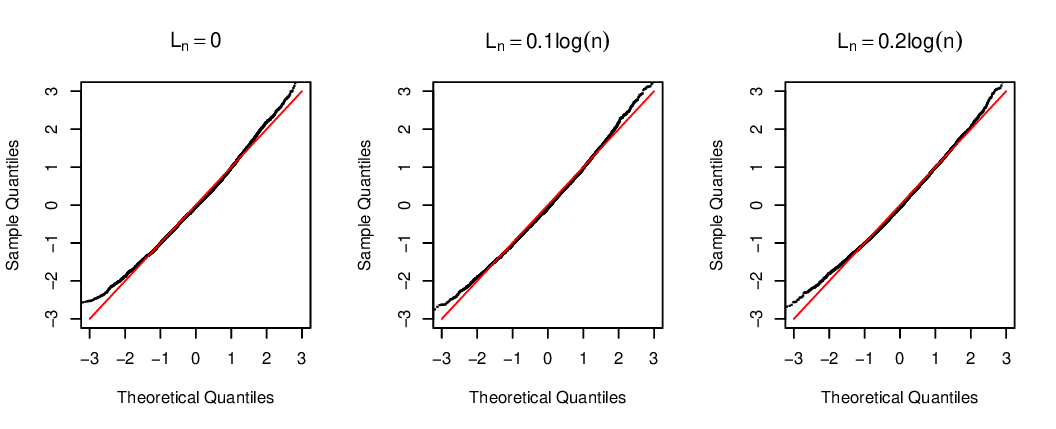}}
\subfigure[QQ-plot for log-likelihood ratio statistic]{\includegraphics[width=0.9\textwidth]{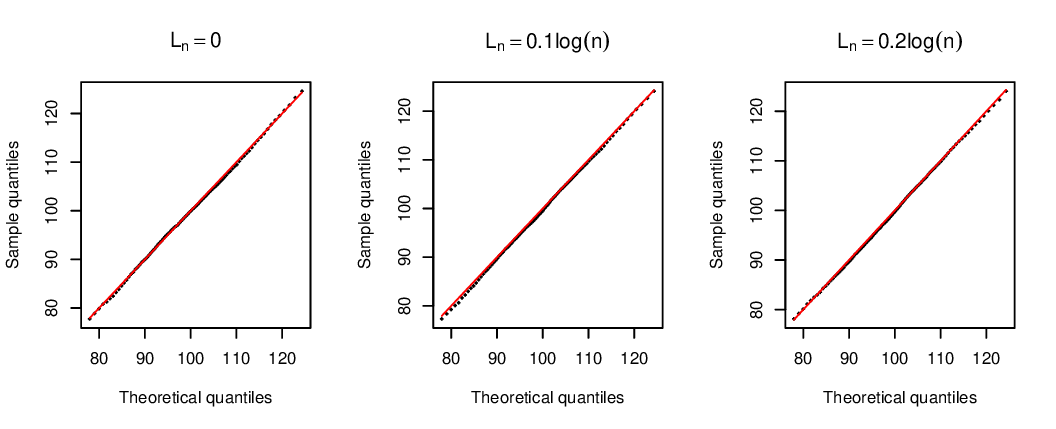}}
\caption{QQ plots for the $p_{0}$-model under $H_{02}$ (n=200). The horizontal and vertical axes in each QQ-plot are the respective theoretical and empirical quantiles.
The straight lines correspond to $y=x$.}
\label{fig:p22}
\end{figure}

\begin{figure}
\centering
{\includegraphics[width=0.9\textwidth]{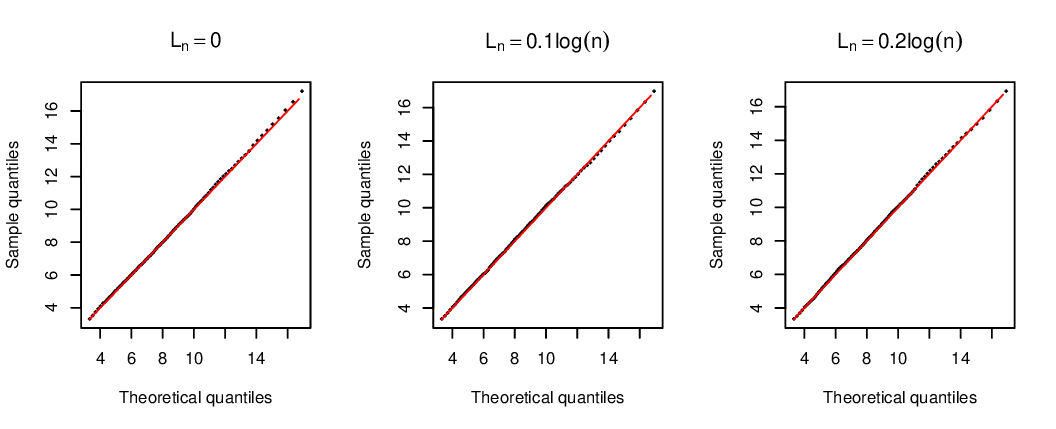}}
\caption{QQ plots under $H_{03}$ (n=500). The horizontal and vertical axes in each QQ-plot are the respective theoretical and empirical quantiles.
The straight lines correspond to $y=x$.}
\label{fig:p3}
\end{figure}